\documentclass[reqno,oneside]{amsart}
\usepackage{hyperref}
\usepackage{geometry}
\usepackage[ansinew]{inputenc}
\usepackage{anyfontsize}
\usepackage{graphicx}
\usepackage{amsmath}% amstex ?
\usepackage{amsthm}
\usepackage{amssymb, color}%
\usepackage[numbers, square]{natbib}
\usepackage{mathrsfs}
\usepackage{bbm}
\usepackage{tikz}
\usepackage{mathptmx}

\usepackage[numbers,square]{natbib}
\usepackage{color}
\allowdisplaybreaks

\numberwithin{equation}{section}

  \evensidemargin
\oddsidemargin
\newcommand{\stirling}[2]{\genfrac{[}{]}{0pt}{}{#1}{#2}}
\newcommand{\stirlingsec}[2]{\genfrac{\{}{\}}{0pt}{}{#1}{#2}}

\newcommand{\N}{\mathbb{N}}
\newcommand{\Z}{\mathbb{Z}}

\newcommand{\R}{\mathbb{R}}
\newcommand{\C}{\mathbb{C}}

\newcommand{\1}{\mathbbm{1}}

\newcommand{\eps}{\varepsilon}

\DeclareMathSymbol{\mlq}{\mathord}{operators}{``}
\DeclareMathSymbol{\mrq}{\mathord}{operators}{`'}

\newcommand*\xbar[1]{%
   \hbox{%
     \vbox{%
       \hrule height 0.5pt % The actual bar
       \kern0.25ex%         % Distance between bar and symbol
       \hbox{%
         \kern-0.05em%      % Shortening on the left side
         \ensuremath{#1}%
         \kern-0.1em%      % Shortening on the right side
       }%
     }%
   }%
}

\DeclareMathOperator{\E}{\mathbb{E}}

\DeclareMathOperator{\Cov}{Cov}

\DeclareMathOperator{\conv}{conv}
\DeclareMathOperator{\pos}{pos}

\newcommand{\ii}{{\rm{i}}}

\renewcommand{\P}{\mathbb{P}}

\renewcommand{\Re}{\operatorname{Re}}
\renewcommand{\Im}{\operatorname{Im}}

\newcommand{\Var}{\mathop{\mathrm{Var}}\nolimits}
\newcommand{\Lah}{\mathop{\mathrm{Lah}}\nolimits}

\newcommand{\eqdistr}{\stackrel{d}{=}}

\newcommand{\todistr}{\overset{d}{\underset{n\to\infty}\longrightarrow}}
\newcommand{\todistrD}{\overset{d}{\underset{n\to\infty}\Longrightarrow}}

\newcommand{\toas}{\overset{a.s.}{\underset{n\to\infty}\longrightarrow}}

\newcommand{\ton}{\overset{}{\underset{n\to\infty}\longrightarrow}}
\newcommand{\tond}{\overset{}{\underset{d\to\infty}\longrightarrow}}

\newcommand{\dd}{{\rm d}}
\newcommand{\eee}{{\rm e}}

\theoremstyle{plain}
\newtheorem{theorem}{Theorem}[section]
\newtheorem{lemma}[theorem]{Lemma}
\newtheorem{corollary}[theorem]{Corollary}
\newtheorem{proposition}[theorem]{Proposition}
\newtheorem{conjecture}[theorem]{Conjecture}

\theoremstyle{definition}
\newtheorem{definition}[theorem]{Definition}

\theoremstyle{remark}
\newtheorem{remark}[theorem]{Remark}

\makeindex
\makeglossary

\makeatletter
\def\@tocline#1#2#3#4#5#6#7{\relax
  \ifnum #1>\c@tocdepth % then omit
  \else
    \par \addpenalty\@secpenalty\addvspace{#2}%
    \begingroup \hyphenpenalty\@M
    \@ifempty{#4}{%
      \@tempdima\csname r@tocindent\number#1\endcsname\relax
    }{%
      \@tempdima#4\relax
    }%
    \parindent\z@ \leftskip#3\relax \advance\leftskip\@tempdima\relax
    \rightskip\@pnumwidth plus4em \parfillskip-\@pnumwidth
    #5\leavevmode\hskip-\@tempdima
      \ifcase #1
       \or\or \hskip 1em \or \hskip 2em \else \hskip 3em \fi%
      #6\nobreak\relax
    \dotfill\hbox to\@pnumwidth{\@tocpagenum{#7}}\par
    \nobreak
    \endgroup
  \fi}
\makeatother

\begin{document}

\author{Zakhar Kabluchko}
\address{Zakhar Kabluchko: Institut f\"ur Mathematische Stochastik,
Westf\"alische {Wilhelms-Uni\-ver\-sit\"at} M\"unster,
Or\-l\'e\-ans--Ring 10,
48149 M\"unster, Germany}
\email{zakhar.kabluchko@uni-muenster.de}

\author{Alexander Marynych}
\address{Alexander Marynych: Faculty of Computer Science and Cybernetics, Taras Shevchenko National University of Kyiv, Kyiv 01601,
Ukraine}
\email{marynych@unicyb.kiev.ua}

\title[Lah distribution]{Lah distribution: Stirling numbers, records on compositions,  and convex hulls of high-dimensional random walks}

\begin{abstract}
Let $\xi_1,\xi_2,\ldots$ be a sequence of independent copies of a random vector in $\R^d$ having an absolutely continuous distribution. Consider a random walk $S_i:=\xi_1+\cdots+\xi_i$, and let $C_{n,d}:=\conv(0,S_1,S_2,\ldots,S_n)$ be the convex hull of the first $n+1$ points it has visited. The polytope $C_{n,d}$ is called $k$-neighborly if for any indices $0\leq i_1 <\cdots < i_k\leq n$ the convex hull of the $k$ points $S_{i_1},\ldots, S_{i_k}$ is a $(k-1)$-dimensional face of $C_{n,d}$. We study the probability that $C_{n,d}$ is $k$-neighborly in various high-dimensional asymptotic regimes, i.e.\  when $n$, $d$, and possibly also $k$ diverge to $\infty$. There is an explicit formula for the expected number of $(k-1)$-dimensional faces of $C_{n,d}$ which involves Stirling numbers of both kinds. Motivated by this formula,  we introduce a distribution, called the Lah distribution, and study its properties. In particular, we provide a combinatorial interpretation of the Lah distribution in terms of random compositions and records, and explicitly compute its factorial moments. Limit theorems which we prove for the Lah distribution imply neighborliness properties of $C_{n,d}$. This yields a new class of random polytopes exhibiting phase transitions parallel to those discovered by Vershik and Sporyshev, Donoho and Tanner for random projections of regular simplices and crosspolytopes.
\end{abstract}

\keywords{Stirling numbers, Lah numbers, Lah distribution, records, random compositions, random walks, random polytopes, convex hulls, neighborliness, $f$-vectors, mod-Poisson convergence, central limit theorem, large deviations, Lambert $W$-function, threshold phenomena, conic intrinsic volumes, Weyl chambers}

\subjclass[2010]{Primary: 11B73, 60C05; Secondary: 60D05, 52A22, 52A23, 60F05, 60F10, 30C15, 26C10, 05A16, 05A18}

\maketitle

\tableofcontents

\section{Introduction and summary of main results}
\subsection{Introduction}
The aim of the present paper is to introduce and study a family of discrete probability distributions defined in terms of Stirling numbers of both kinds and Lah numbers. Recall, see for example~\cite[Section~6.1]{Graham1994}, that the Stirling numbers of the first kind, denoted by $\stirling{n}{k}$, count the  number of permutations of $n$ objects with exactly $k$ disjoint cycles, while the Stirling numbers of the second kind, denoted by $\stirlingsec{n}{k}$, count the number of ways to partition a set of $n$ elements into $k$ nonempty subsets.
Alternatively, Stirling numbers can be defined by the exponential generating functions via the identities
\begin{equation}\label{eq:stirling_gen_funct}
\frac 1 {k!} \left(\log \frac{1}{1-x}\right)^k = \sum_{n=k}^\infty \stirling{n}{k} \frac{x^n}{n!}\qquad\text{and}\qquad \frac{1}{k!} (\eee^x - 1)^k = \sum_{n=k}^\infty \stirlingsec{n}{k} \frac{x^n}{n!},\qquad k=0,1,2,\ldots.
\end{equation}
%Reviews of probability distributions related to Stirling numbers can be found in~\cite{charalambides_singh}.
The \textit{Lah number} $L(n,k)$ can be defined as the number of ways to partition the set  $\{1,\ldots,n\}$ into $k$ non-empty subsets and to linearly order the elements inside each subset. It is known that $L(n,k)$ is given by
\begin{equation}\label{eq:lah_numb_def}
L(n,k) = \sum_{j=k}^n \stirling {n}{j} \stirlingsec{j}{k} = \frac{(n-1)!}{(k-1)!} \binom nk = \frac{n!}{k!}\binom {n-1}{k-1}, \quad n\in \N,\quad k\in \{1,\ldots,n\}.
\end{equation}
These numbers were introduced by Ivo Lah~\cite{Lah1954} whose name they now bear; see entry  A105278 in~\cite{sloane} and~\cite{daboul} for their properties. We can now define the family of distributions we are interested in.

\begin{definition}
A random variable $X=\Lah(n,k)$ has a \textit{Lah distribution} with parameters $n\in\N$ and $k\in \{1,\ldots,n\}$ if
\begin{equation}\label{eq:lah_distr_def}
\P[X=j] = \frac{1}{L(n,k)} \stirling{n}{j} \stirlingsec{j}{k},
\qquad
j\in \{k,k+1,\ldots,n\}.
\end{equation}
Throughout the paper, we agree that $\Lah(n,k)$ denotes some random variable with distribution~\eqref{eq:lah_distr_def}.
\end{definition}

The special case of the Lah distribution with $k=1$ is well known to be the distribution of the number of cycles in a random uniform permutation of $\{1,\ldots,n\}$, or the number of  records in a sample of $n$ independent identically distributed (i.i.d.)\ observations with a continuous distribution function. We shall extend the latter interpretation to arbitrary $k$,
but the original motivation for introducing the Lah distribution comes from the study of threshold phenomena for  high-dimensional random polytopes initiated in the pioneering work of Vershik and Sporyshev~\cite{vershik_sporyshev_asymptotic_faces_random_polyhedra1992} and continued in a series of works of Donoho and Tanner~\cite{donoho_neighborliness_proportional,donoho_tanner_neighborliness,donoho_tanner,DonohoTanner}.
As has been suggested by Vershik in his Grassmannian approach to linear programming~\cite{vershik_sporyshev_estimation_simplex1983},  these threshold phenomena have multiple implications in high-dimensional statistics, signal processing, linear optimization and other fields; for numerous examples we refer to the above cited papers as well as~\cite{ALMT14,baldi_vershynin,donoho_tanner_sparse_nonnegative_sol,donoho_tanner_observed_universality,donoho_tanner_exponential_bounds,vershik_sporyshev_estimation_simplex1983,vershik_sporyshev_asymptotic_estimate_1986}.
Let us briefly recall the problem studied by Vershik, Sporyshev, Donoho and Tanner.
Consider $n$ i.i.d.\ standard Gaussian points $X_1,\ldots,X_n$ in the $d$-dimensional space, where $n\geq d+1$.
Their convex hull $G_{n,d} := \conv(X_1,\ldots,X_n)$ is called the Gaussian polytope. Let $f_{k-1}(G_{n,d})$ be the number of $(k-1)$-dimensional faces of $G_{n,d}$, for $k\in \{1,\ldots,d\}$. With probability $1$, every $(k-1)$-dimensional face is a simplex of the form $\conv(X_{i_1},\ldots, X_{i_k})$ for some $k$-tuple of pairwise different indices $i_1,\ldots, i_k \in \{1,\ldots,n\}$.  Clearly, $f_{k-1}(G_{n,d})$ is bounded from above by the number of such $k$-tuples, that is, by $\binom {n}{k}$. If this bound is attained, the polytope $G_{n,d}$ is called \textit{$k$-neighborly}; see~\cite[Chapter~7]{gruenbaum_book}. Vershik and Sporyshev~\cite{vershik_sporyshev_asymptotic_faces_random_polyhedra1992} studied the so-called proportional growth regime in which $d,n,k\to \infty$ in such a way that $k/d\to \rho$ and $d/n \to \delta$ for some constants $\rho \in (0,1)$ and $\delta\in (0,1)$. They proved the existence of what has been later called a weak threshold, that is, a positive function $\delta\mapsto \rho^{\text{GP}}_{\text{weak}}(\delta)$ such that
$$
\lim_{n,d,k\to\infty} \frac{\E f_{k-1}(G_{n,d})}{\binom{n}{k}} =1
\qquad
\text{ provided }
\rho < \rho^{\text{GP}}_{\text{weak}}(\delta).
$$
Later, Donoho and Tanner~\cite{donoho_tanner_neighborliness} proved the existence of what they called a strong threshold, that is, a positive function $\delta\mapsto \rho^{\text{GP}}_{\text{strong}}(\delta)$ such that
$$
\lim_{n,d,k\to\infty} \left(\binom{n}{k} - \E f_{k-1}(G_{n,d})\right) = 0
\qquad
\text{ provided }
\rho < \rho^{\text{GP}}_{\text{strong}}(\delta).
$$
From this relation, they deduced that
$$
\lim_{n,d,k\to\infty} \P\left[f_{k-1}(G_{n,d}) = \binom{n}{k}\right]  = 1
\qquad
\text{ provided }
\rho < \rho^{\text{GP}}_{\text{strong}}(\delta).
$$
The same conclusions hold for the projection of the regular simplex with $n$ vertices on a random uniform $d$-dimensional subspace, which has the same expected $f$-vector as $G_{n,d}$ by a result of Baryshnikov and Vitale~\cite{BV94}. Analogous theory exists, see~\cite{donoho_neighborliness_proportional}, for random projections of the regular crosspolytope or, equivalently, the symmetric Gaussian polytope defined as $\conv(\pm X_1,\ldots, \pm X_n)$. Going beyond the proportional growth setting, Donoho and Tanner~\cite{donoho_tanner} studied also the case when $\delta=0$.
Recently, there has been also interest in the threshold phenomena for random cones as the dimension goes to $\infty$; see~\cite{DonohoTanner,GKT2020_HighDimension1,HugSchneiderThresholdPhenomena,HugSchneiderThresholdPhenomenaPart2}.
%References: !!!!! Neighborliness introduced in \cite{SchneiderGale}, \cite{Gale}
%\cite{dyer_etal}, \cite{pivovarov}, \cite{bonnet_etal_threshold}, \cite{bonnet_oreilly}, \cite{bonnet_kabluchko_turchi}.

\subsection{Convex hulls of random walks}
In the present paper we shall be interested in neighborliness properties of a class of random polytopes defined as follows.
Let $\xi_1,\xi_2,\ldots$ be i.i.d.\ random variables with an absolutely continuous distribution on $\R^d$. These assumptions can be weakened; see Section~\ref{subsec:convex_hulls_random_walks} below for details. Consider the $d$-dimensional random walk $(S_i)_{i=0}^\infty$ defined by $S_i := \xi_1+\cdots+\xi_i$, $i\in \N$, and $S_0 := 0$. We are interested in the convex hull of $S_0,\ldots, S_n$ which will be denoted by
\begin{equation}\label{eq:def_C_n_d}
C_{n,d} := \conv(S_0,\dots, S_n)
=\{\lambda_0 S_0+\dots+\lambda_n S_n \colon \lambda_0,\dots,\lambda_n \geq 0, \lambda_0+\dots+\lambda_n =1\},
\qquad
n\geq d.
\end{equation}
Let $f_{\ell}(C_{n,d})$ be the number of $\ell$-dimensional faces of the polytope $C_{n,d}$, for $\ell\in \{0,\ldots,d-1\}$.
The following explicit formula for the expected face numbers of $C_{n,d}$ has been obtained in~\cite{KVZ17}  relying  on the methods of~\cite{KVZ17b}:
\begin{equation}\label{eq:E_F_k_C_n_main_theorem_introduct}
\E f_{k-1}(C_{n,d}) = \frac{2\cdot (k-1)!} {n!} \sum_{l=0}^{\infty} \stirling{n+1}{d-2l}  \stirlingsec{d-2l}{k},
\qquad
k \in \{1,\ldots,d\}.
\end{equation}
In terms of the Lah distribution introduced above, the formula can be stated as follows:
\begin{equation}\label{eq:E_F_k_C_n_Lah_main_theorem_introduct}
\frac{\E f_{k-1}(C_{n,d})}{\binom{n+1}{k}} = 2 \P[\Lah(n+1,k) \in \{d,d-2,d-4,\ldots\}].
\end{equation}
We are interested in the high-dimensional limit when $n,d$ and, possibly, also $k$, go to $\infty$ in a coupled manner. Let us argue that probabilistic limit theorems for the Lah distribution imply threshold phenomena for $C_{n,d}$. Suppose, for example, that in some asymptotic regime $n=n(d)$, $k=k(d)$ we were able to prove a weak law of large numbers of the form
\begin{equation}\label{eq:weak_law_intro}
\frac{\Lah (n+1, k)}{\E \Lah(n+1,k)} \tond 1 \text{ in probability}.
\end{equation}
As we shall see below, the right-hand side of~\eqref{eq:E_F_k_C_n_Lah_main_theorem_introduct} does not differ much from the distribution function in the sense that the approximation
$$
\frac{\E f_{k-1}(C_{n,d})}{\binom{n+1}{k}}  \approx \P[\Lah(n+1,k) \leq  d].
$$
can be justified. The weak law of large numbers~\eqref{eq:weak_law_intro} then implies that
$$
\lim_{d\to\infty} \frac{\E f_{k-1}(C_{n,d})}{\binom{n+1}{k}}
=
\begin{cases}
1, & \text{ if } \limsup_{d\to\infty}d^{-1}\E \Lah(n+1,k) < 1,\\
0, & \text{ if }  \liminf_{d\to\infty}d^{-1}\E \Lah(n+1,k) > 1,
\end{cases}
$$
which means that there is a threshold phenomenon if $d$ is near $\E \Lah(n+1,k)$. In a similar way, a central limit theorem for $\Lah (n+1,k)$ would imply a characterization of the limit in the critical window.

\subsection{Summary of results}
Our goal is to investigate the properties of the Lah distribution. In particular, limit theorems for $\Lah(n,k)$ which we shall prove in various asymptotic regimes  of $n$ and $k$ yield threshold phenomena for convex hulls of $d$-dimensional random walks as $d\to \infty$. Our main results can be summarized as follows.
\begin{itemize}
\item[(a)] We provide a combinatorial interpretation of Lah distributions $\Lah(n,k)$ in terms of random compositions and records, which also allows us to construct the whole family of  random variables $\Lah(n,k)$ simultaneously for all $n\in\N$ and $k\in\{1,\ldots,n\}$ in a consistent way on a common probability space. This yields stochastic monotonicity of $\Lah (n,k)$ in $n$ and $k$. This combinatorial construction is a subject of Section~\ref{sec:combinatorics}.
\item[(b)] We compute explicitly the expectation, the variance and the factorial moments of the Lah distribution. For example, we show that
$$
\E \Lah(n,k)= \frac 1 {L(n,k)} \sum_{j=k}^n j \stirling{n}{j} \stirlingsec{j}{k}=
\frac{k} {\binom {n-1}{k-1}} \sum_{i=1}^{n-k+1} \frac{1}{i} \binom {n-i}{k-1}
=\frac{k} {\binom {n-1}{k-1}} \sum_{i=1}^{n-k+1} \frac{(-1)^{i+1}}{i} \binom {n}{k+i-1}.
$$
The aforementioned moment results as well as other basic properties of the Lah distribution are presented in Section~\ref{sec:distr_properties}.
\item[(c)] We prove that for fixed $k\in \N$, the random variables $\Lah(n,k)$ converge in the mod-Poisson sense with speed $\lambda_n= k\log n$, which implies several limit theorems including the central limit theorem
    $$
\frac{\Lah(n,k) - k \log n}{\sqrt{k\log n}} \todistr {\rm N}(0,1)
    $$
    as well as the precise asymptotics for the large deviations probabilities. This regime of fixed $k$ is analyzed in Section~\ref{sec:constant_k}.
\item[(d)] In the regime when $k=k(n)$ grows linearly with $n$, that is $k(n)\sim \alpha n$ with $\alpha\in (0,1)$, we prove a central limit theorem and a large deviation principle for $\Lah(n,k)$; see Section~\ref{sec:growing_k}.
\item [(e)] We apply these results to establish the aforementioned threshold phenomena for convex hulls of random walks in various asymptotic regimes of $n$, $d$, $k$ in Section~\ref{sec:convex_hulls_random_walks}.
\item[(f)] We explain how the Lah distribution is related to the conic intrinsic volume sums of Weyl chambers in Section~\ref{sec:weyl}.
\end{itemize}

\section{Combinatorial construction of the Lah distribution}\label{sec:combinatorics}

In this section we shall establish connections between Lah distributions and several classical probabilistic and combinatorial models. This connection will allow us to construct the family $(\Lah(n,k))_{n\in\N,1\leq k\leq n}$ in a consistent (simultaneously in $n$ and in $k$) way on a common probability space. This, in turn, leads to a useful representation of $\Lah(n,k)$ and also establishes some basic qualitative properties of Lah distributions such as, for example, stochastic monotonicity.
We start by observing  that, for $k=1$, the formula for the Lah distribution takes the form
\begin{equation}\label{eq:records}
\P[\Lah(n,1)=j] = \frac{1}{n!} \stirling{n}{j}, \qquad j\in \{1,\ldots,n\}.
\end{equation}
This special case pops up  at many places in probability theory, for example as the distribution of the number of cycles in a uniform random permutation of $n$ elements, as the distribution of the number of records in an independent sample of size $n$ from a continuous distribution, or as the distribution of $\sum_{\ell=1}^n B_{\ell}$, where $B_1,B_{2}, \ldots$ are independent Bernoulli variables with distributions $\P[B_\ell=1]=1-\P[B_\ell=0]= 1/\ell$, $\ell\in\N$.

In order to extend these representations of $\Lah(n,1)$ to other values of $k$, we need to recall the notion of random compositions.

\subsection{Random compositions and records}\label{subsec:compositions_aldous_coupling}
A \textit{composition} of a positive integer $n$ into $k$ summands (blocks) is a representation of $n$ as a sum $i_1+i_2+\cdots+i_k$ of $k$ positive integers in which the order of summands is essential. Thus, $1+3$, $3+1$ and $2+2$ are three different compositions of $n=4$ into $k=2$ summands. By the standard ``stars-and-bars'' argument there are exactly $\binom{n-1}{k-1}$ different compositions of $n$ into $k$ summands. Throughout this paper we let $(b^{(n)}_1,b^{(n)}_2,\ldots,b^{(n)}_k)$ denote a random composition of $n$ into $k$ summands picked uniformly at random, that is, with distribution
\begin{equation}\label{eq:uniform_composition_distribution}
\P[(b^{(n)}_1,b^{(n)}_2,\ldots,b^{(n)}_k)=(i_1,i_2,\ldots,i_k)]=\frac{1}{\binom{n-1}{k-1}},
\end{equation}
for every $i_1,i_2,\ldots,i_k\in\N$ summing up to $n$.
%The uniform random composition is also known as the Bose-Einstein statistics.
The family of random compositions $(b^{(n)}_1,b^{(n)}_2,\ldots,b^{(n)}_k)$ can be defined in a consistent way (simultaneously in $n\in\N$ and $k\in\{1,2,\ldots,n\}$) using the so-called Aldous' construction. We present its simplified version here in a form borrowed  from \cite{NBerestycki:09}, see Section 2.1.3 therein. Start with a chain of length $n$ connecting the labeled vertices $U_1,U_2,\ldots,U_n$, see Figure~\ref{fig:aldous} (first row). This chain represents the unique composition of $n$ into a single block and corresponds to time $1$ of our construction. At time $2$, pick one of $n-1$ edges uniformly at random and remove it. The resulting two connected components, see Figure~\ref{fig:aldous} (second row), induce a uniformly distributed random composition of $n$ into two summands. Proceeding this way and removing at time $k\in\{1,\ldots,n\}$ an edge picked uniformly at random among the existing $n-k+1$ edges, results in a consistent (in $k$) family of random compositions given by the sizes of connected components counted from left to right. The number of blocks at time $k$  (or in the $k$-th row) is $k$. According to Lemma 2.1 in \cite{NBerestycki:09}, the composition obtained after removing $k-1$ edges is uniformly distributed on the set of all partitions of $n$ into $k$ summands. Note that this construction is also consistent in $n$ in the following sense. If we start with $n+1$ vertices, construct compositions $(b^{(n+1)}_1,b^{(n+1)}_2,\ldots,b^{(n+1)}_k)$ for $k=1,2,\ldots,n+1$, and then remove completely the $(n+1)$-th column and a duplicated row (which necessarily appears upon deleting the $(n+1)$-th column), we obtain a family of uniform random compositions of $n$ into $k$ blocks distributed as $(b^{(n)}_1,b^{(n)}_2,\ldots,b^{(n)}_k)$ for $k=1,2,\ldots,n$.

  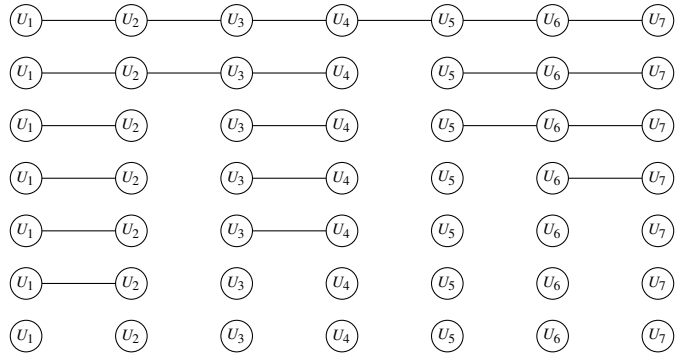
\begin{figure}[!hbtp]
    \centering
\begin{tikzpicture}[scale=0.70]
        \draw (-5.7,0) -- (-4.3,0);
        \draw (-3.7,0) -- (-2.3,0);
        \draw (-1.7,0) -- (-0.3,0);
        \draw (0.3,0) -- (1.7,0);
        \draw (2.3,0) -- (3.7,0);
        \draw (4.3,0) -- (5.7,0);

        \draw (-6,0) circle (0.3) node {\tiny{$U_1$}};
        \draw  (-4,0) circle (0.3) node {\tiny{$U_2$}};
        \draw  (-2,0) circle (0.3) node {\tiny{$U_3$}};
        \draw  (0,0) circle (0.3) node {\tiny{$U_4$}};
        \draw  (2,0) circle (0.3) node {\tiny{$U_5$}};
        \draw  (4,0) circle (0.3) node {\tiny{$U_6$}};
        \draw  (6,0) circle (0.3) node {\tiny{$U_7$}};

        \draw (-5.7,-1) -- (-4.3,-1);
        \draw (-3.7,-1) -- (-2.3,-1);
        \draw (-1.7,-1) -- (-0.3,-1);
        \draw (2.3,-1) -- (3.7,-1);
        \draw (4.3,-1) -- (5.7,-1);

        \draw (-6,-1) circle (0.3) node {\tiny{$U_1$}};
        \draw  (-4,-1) circle (0.3) node {\tiny{$U_2$}};
        \draw  (-2,-1) circle (0.3)  node {\tiny{$U_3$}};
        \draw  (0,-1) circle (0.3) node {\tiny{$U_4$}};
        \draw  (2,-1) circle (0.3) node {\tiny{$U_5$}};
        \draw  (4,-1) circle (0.3) node {\tiny{$U_6$}};
        \draw  (6,-1) circle (0.3) node {\tiny{$U_7$}};

        \draw (-5.7,-2) -- (-4.3,-2);
        \draw (-1.7,-2) -- (-0.3,-2);
        \draw (2.3,-2) -- (3.7,-2);
        \draw (4.3,-2) -- (5.7,-2);

        \draw (-6,-2) circle (0.3) node {\tiny{$U_1$}};
        \draw  (-4,-2) circle (0.3) node {\tiny{$U_2$}};
        \draw  (-2,-2) circle (0.3) node {\tiny{$U_3$}};
        \draw  (0,-2) circle (0.3) node {\tiny{$U_4$}};
        \draw  (2,-2) circle (0.3) node {\tiny{$U_5$}};
        \draw  (4,-2) circle (0.3) node {\tiny{$U_6$}};
        \draw  (6,-2) circle (0.3) node {\tiny{$U_7$}};

        \draw (-5.7,-3) -- (-4.3,-3);
        \draw (-1.7,-3) -- (-0.3,-3);
        \draw (4.3,-3) -- (5.7,-3);

        \draw (-6,-3) circle (0.3) node {\tiny{$U_1$}};
        \draw  (-4,-3) circle (0.3) node {\tiny{$U_2$}};
        \draw  (-2,-3) circle (0.3) node {\tiny{$U_3$}};
        \draw  (0,-3) circle (0.3) node {\tiny{$U_4$}};
        \draw  (2,-3) circle (0.3) node {\tiny{$U_5$}};
        \draw  (4,-3) circle (0.3) node {\tiny{$U_6$}};
        \draw  (6,-3) circle (0.3) node {\tiny{$U_7$}};

        \draw (-5.7,-4) -- (-4.3,-4);
        \draw (-1.7,-4) -- (-0.3,-4);

        \draw (-6,-4) circle (0.3) node {\tiny{$U_1$}};
        \draw  (-4,-4) circle (0.3) node {\tiny{$U_2$}};
        \draw  (-2,-4) circle (0.3) node {\tiny{$U_3$}};
        \draw  (0,-4) circle (0.3) node {\tiny{$U_4$}};
        \draw  (2,-4) circle (0.3) node {\tiny{$U_5$}};
        \draw  (4,-4) circle (0.3) node {\tiny{$U_6$}};
        \draw  (6,-4) circle (0.3) node {\tiny{$U_7$}};

        \draw (-5.7,-5) -- (-4.3,-5);

        \draw (-6,-5) circle (0.3) node {\tiny{$U_1$}};
        \draw  (-4,-5) circle (0.3) node {\tiny{$U_2$}};
        \draw  (-2,-5) circle (0.3) node {\tiny{$U_3$}};
        \draw  (0,-5) circle (0.3) node {\tiny{$U_4$}};
        \draw  (2,-5) circle (0.3) node {\tiny{$U_5$}};
        \draw  (4,-5) circle (0.3) node {\tiny{$U_6$}};
        \draw  (6,-5) circle (0.3) node {\tiny{$U_7$}};

        \draw (-6,-6) circle (0.3) node {\tiny{$U_1$}};
        \draw  (-4,-6) circle (0.3) node {\tiny{$U_2$}};
        \draw  (-2,-6) circle (0.3) node {\tiny{$U_3$}};
        \draw  (0,-6) circle (0.3) node {\tiny{$U_4$}};
        \draw  (2,-6) circle (0.3) node {\tiny{$U_5$}};
        \draw  (4,-6) circle (0.3) node {\tiny{$U_6$}};
        \draw  (6,-6) circle (0.3) node {\tiny{$U_7$}};
\end{tikzpicture}
\caption{Aldous' construction of the consistent family of uniform random compositions. In this example a consistent family of partitions of $n=7$ is: for $k=1$, $7=7$; for $k=2$, $7=4+3$; for $k=3$, $7=2+2+3$; for $k=4$, $7=2+2+1+2$; for $k=5$, $7=2+2+1+1+1$; for $k=6$, $7=2+1+1+1+1+1$ and, for $k=7$, $7=1+1+1+1+1+1+1$.}
\label{fig:aldous}
  \end{figure}

So far the labels $U_1,U_2,\ldots,U_n$ of vertices in Aldous' construction did not play a role but now we shall exploit them to construct a consistent family $(\Lah(n,k))_{n\in\N,1\leq k\leq n}$. Let $U_1,U_2,\ldots,U_n$ be a sample of independent uniformly distributed on $[0,1]$ random variables which is also independent of the above edge-removing process. For fixed $n\in\N$ and $k=1,\ldots,n$ take a composition $(b^{(n)}_1,b^{(n)}_2,\ldots,b^{(n)}_k)$ induced by the $k$-th row of Aldous' construction. We say that a vertex $U_i$ is a record with respect to the composition $(b^{(n)}_1,b^{(n)}_2,\ldots,b^{(n)}_k)$ if it is a record in the block it occupies, that is, it is larger than all previous elements inside this block counting from left to right. The main result of this section if given by the next proposition.

\begin{proposition}\label{prop:representation}
The total number $X_{n,k}$ of records with respect to a uniform random composition of $n$ into $k$ summands has the $\Lah(n,k)$ distribution.
\end{proposition}

The easiest way to prove Proposition~\ref{prop:representation} is via generating functions; but we shall also give a combinatorial proof.
%We shall give two proofs of Proposition~\ref{prop:representation}, one using generating functions, and a combinatorial one.
Recall that $[x^n]f(x)$ denotes the coefficient of $x^n$ in the Taylor or Laurent expansion of $f(x)$ around $0$. The following lemma will be  useful on many occasions.
\begin{lemma}\label{lem:stirling_product}
For all $n,k,j\in \N$ with $k\leq j\leq n$ we have
$$
\frac{k!}{n!}\stirling{n}{j} \stirlingsec{j}{k}  = [t^j] [x^n] \left((1-x)^{-t} - 1\right)^k.
$$
\end{lemma}
\begin{proof}
Using both identities in~\eqref{eq:stirling_gen_funct} we have
$$
\frac{\left((1-x)^{-t} - 1\right)^k}{k!}
=
\frac{\left(\eee^{t\log \frac 1 {1-x}} - 1\right)^k}{k!}
=
\sum_{j=k}^\infty
\left(\log \frac 1 {1-x}\right)^j  \stirlingsec{j}{k} \frac{t^j}{j!}
=
\sum_{j=k}^\infty \sum_{n=j}^\infty \frac{1}{n!}\stirling{n}{j} \stirlingsec{j}{k} t^j x^n.
$$
The claim follows by equating the coefficients.
\end{proof}

\begin{proof}[Proof of Proposition~\ref{prop:representation} using generating functions]
For $t\in \R$ and $n\in\N$, let $\phi_n(t)$ be the generating function of the number of records in a sample of size $n$, that is,
$$
\phi_n(t)=\E t^{\Lah(n,1)}=\sum_{j=1}^{n}\frac{1}{n!}\stirling{n}{j}t^j=\frac{t(t+1)\cdots(t+n-1)}{n!},
$$
see~\eqref{eq:records}. Conditioning on the event $(b^{(n)}_1,b^{(n)}_2,\ldots,b^{(n)}_k)=(i_1,i_2,\ldots,i_k)$ we obtain, for $k\in\N$, $|x|<1$ and $t\in \R$,
\begin{align*}
\sum_{n=k}^{\infty}\binom{n-1}{k-1}\E t^{X_{n,k}}x^n=\sum_{n=k}^{\infty}\sum_{\substack{i_1+\cdots+i_k=n\\ i_1,\ldots,i_k\geq 1}} \prod_{\ell=1}^{k} (\phi_{i_\ell}(t)x^{i_\ell})=\left(\sum_{i=1}^{\infty}\phi_i(t)x^i\right)^k=((1-x)^{-t}-1)^k=\sum_{n=k}^{\infty}\binom{n-1}{k-1}\E t^{\Lah(n,k)}x^n,
\end{align*}
where the last equality follows from Lemma~\ref{lem:stirling_product} and equations~\eqref{eq:lah_distr_def} and~\eqref{eq:lah_numb_def}.
\end{proof}

\begin{proof}[Combinatorial proof of Proposition~\ref{prop:representation}]
Fix $n\in \N$ and $k\in \{1,\ldots, n\}$. Consider the set $\mathbf L_{n,k}$ of all pairs $(\sigma, \pi)$, where $\pi= (A_1,\ldots,A_k)$ is an ordered partition of the set $\{1,\ldots,n\}$ into $k$ non-empty blocks (and the order in which the blocks appear is essential), while $\sigma:\{1,\ldots,n\}\to\{1,\ldots,n\}$ is a permutation preserving	 $\pi$ meaning that $\sigma$ only permutes the elements inside the blocks of $\pi$ but not between the blocks. The total number of pairs $(\sigma, \pi)$ in which $\sigma$ has exactly $j$ cycles is given by $\stirling{n}{j} k! \stirlingsec{j}{k}$, which follows from the definition of the Stirling numbers of both kinds. The total number of pairs $(\sigma, \pi)$ in $\mathbf L_{n,k}$ is $k! L(n,k)$, which either follows from~\eqref{eq:lah_numb_def}, or by recalling that the Lah number $L(n,k)$ counts the number of partitions of $\{1,\ldots,n\}$ into $k$ blocks and putting linear order on the elements of each block. (The one-line notation of the restriction of $\sigma$ to each block corresponds to a linear order on that block). Now, let $(\Sigma, \Pi)$ be random and uniformly distributed on the finite set $\mathbf L_{n,k}$. As we argued above, the number of cycles of $\Sigma$ has the Lah distribution $\Lah (n,k)$. On the other hand, let us take some deterministic ordered partition $\pi= (I_1,\ldots, I_k)$ of $\{1,\ldots,n\}$ into $k$ blocks. The number of permutations $\sigma$ preserving $\pi$ is $|I_1|!\ldots |I_k|!$. The number of ordered partitions of $\{1,\ldots,n\}$ with prescribed block sizes $i_1=|I_1|,\ldots,i_k=|I_k|$ is given by $n!/i_1!\ldots i_k!$. Hence, the block sizes of $\Pi$ form a random, uniform  composition of $n$ in $k$ summands. Conditionally on $\Pi$, the restrictions of $\Sigma$ to these blocks are independent and uniform random permutations of the elements inside the blocks. Recall now that the number of cycles of a uniform random permutation of $i_j$ elements has the same distribution as the number of records in a uniform sample of size $i_j$. Hence, the number of cycles of $\Sigma$ has the same distribution as $X_{n,k}$, and the proof is complete.
%The Lah number $L(n,k)$ counts the number of ways to partition the set $\{1,\ldots,n\}$ into $k$ blocks and putting linear order on each of the blocks. Any partition with $k$ linear ordered blocks of sizes $i_1,\ldots,i_k$ can be written as $\{(j_1^{(1)},\ldots, j_{i_1}^{(1)}),\ldots, (j_{1}^{(k)}, \ldots, j_{i_k}^{(k)})\}$, and partitions differing only by permuting the blocks are considered equal. Any linearly ordered block can be viewed as the one-line notation for some permutation of the corresponding elements. By writing any such permutation as a product of cycles and
%Thus, partitions with $k$ linearly ordered blocks are in bijections
%If some of the linearly ordered blocks is $(j_1,\ldots,j_\ell)$, then we say that the element $j_k$ is a record if it is larger than $j_1,\ldots, j_{k-1}$. The total number of all partitions with linear order
\end{proof}

In the sequel we shall frequently use the following representation of the Lah distribution which is an immediate consequence of Proposition~\ref{prop:representation}.

\begin{proposition}
Let $(b^{(n)}_1,b^{(n)}_2,\ldots,b^{(n)}_k)$ denote the uniform random composition of $n$ into $k$ parts, that is, a random composition with distribution~\eqref{eq:uniform_composition_distribution}. Moreover, let $(Z^{(j)}_{i})_{i,j\in\N}$ be an array of mutually independent (and independent of $(b^{(n)}_1,b^{(n)}_2,\ldots,b^{(n)}_k)$) random variables such that $Z^{(j)}_{n}\overset{{\rm d}}{=}\Lah(n,1)$, that is, has distribution~\eqref{eq:records}, for $n,j\in\N$. Then,
\begin{equation}\label{eq:representation_as_a_sum_over_blocks_of_composition}
\Lah(n,k)\overset{{\rm d}}{=}\sum_{j=1}^{k}Z^{(j)}_{b^{(n)}_j},
\end{equation}
where $\overset{{\rm d}}{=}$ denotes equality in distribution.
\end{proposition}

\subsection{P\'olya urn coupling}\label{subsec:polya}
The coupling $(X_{n,k})_{n\in\N,1\leq k\leq n}$ of the Lah distributions constructed above has the property that by its very definition  $X_{n,k}\leq X_{n,k+1}$ a.s. However, the monotonicity in $n$, i.e.\ the inequality $X_{n,k}\leq X_{n+1,k}$, may fail in general. It turns out that, for every fixed $k\in \N$, there is another coupling of the sequence $\Lah (n,k)$, $n\in \{k,k+1,\ldots\}$,  which is non-decreasing in $n$.  To construct it, let $U_1,U_2,\ldots$ be independent random variables with the uniform distribution on $[0,1]$. Consider an urn containing $k$ balls of $k$ different colors and carrying labels $U_1,\ldots,U_k$.  Suppose that, at some step, there are $n-1$ balls in the urn.  Draw one ball from the urn uniformly at random and return it to the urn together with one more ball which has the same color and carries label $U_n$, and proceed in this way. We say that $U_n$ is a local record if $U_n$ is larger than the labels of all balls which were already in the urn and had the same color as the ball with the label $U_n$. Let $b_{j}^{(n)}$ be the number of balls of color $j$ when the total number of balls in the urn is $n$,  and let $Y_{n,k}$ be the number of local records at this time. Then,  $(b_1^{(n)},\ldots, b_{k}^{(n)})$ has the same distribution as the uniform random composition; see~\cite[Chapter~40]{johnson_kotz_balakrishnan_book}.  Consequently, $Y_{n,k}$ has the Lah distribution $\Lah (n,k)$.
Observe that by construction, we have $Y_{n,k}\leq Y_{n+1,k}$ for all $n\geq k$.
%which provides another proof of~\eqref{eq:Lah_monotone_n}. We were unable to adopt this argument to prove~\eqref{eq:Lah_monotone_k}.

%%%%%%%%%%%%%%The section about the construction of Lah distributions and permutations is after  \end{document}

\subsection{Stochastic monotonicity}
From the above constructions we immediately obtain the following result on stochastic monotonicity. It seems to be a non-trivial task to deduce it from from the definition of the Lah distribution given in~\eqref{eq:lah_distr_def} alone.

\begin{proposition}\label{prop:stoch_monotonicity}
The Lah distributions $\Lah(n,k)$ satisfy the following stochastic monotonicity properties: for $n\in\N$ and $k\in\{1,2,\ldots,n-1\}$ we have
\begin{equation}\label{eq:Lah_monotone_k}
\Lah(n,k)\overset{{\rm d}}{\leq}\Lah(n,k+1),
\end{equation}
and, for $n\in\N$ and $k\in\{1,2,\ldots,n\}$,
\begin{equation}\label{eq:Lah_monotone_n}
\Lah(n,k)\overset{{\rm d}}{\leq}\Lah(n+1,k),
\end{equation}
where for two real-valued random variables $X,Y$ we write $X\overset{{\rm d}}{\leq}Y$ iff $\P[X\leq t]\geq \P[Y\leq t]$ for all $t\in\R$.
\end{proposition}
\begin{proof}
The relations follow from the fact that the couplings constructed in Sections~\ref{subsec:compositions_aldous_coupling} and~\ref{subsec:polya} satisfy $X_{n,k}\leq X_{n,k+1}$ (in the Aldous coupling) and $Y_{n,k}\leq Y_{n+1,k}$ (in the P\'olya urn coupling), a.s. Both inequalities follow directly from the definitions of the corresponding couplings.
\end{proof}

\begin{corollary}
For every $n\in \N$, the expectation of $\Lah(n,k)$ is a nondecreasing function of $k\in \{1,\ldots,n\}$. For every $k\in \N$, the expectation of $\Lah(n,k)$ is a nondecreasing function of $n\in \{k,k+1,\ldots\}$.
\end{corollary}

\section{Basic properties of the Lah distribution}\label{sec:distr_properties}

We start by providing a representation for the generating function of a Lah-distributed random variable $\Lah(n,k)$ which is defined by
\begin{equation}\label{eq:P_n_k_def}
P_{n,k}(t) := \E t^{\Lah(n,k)} = \frac 1 {L(n,k)} \sum_{j=k}^n t^j \stirling{n}{j}\stirlingsec{j}{k},
\qquad t\in\C.
\end{equation}
\begin{lemma}\label{lem:gener_funct_formula}
For all $n\in \N$, $k\in \{1,\ldots,n\}$ and $t\in\C$ we have
\begin{align}
P_{n,k}(t)
&=
\frac {1}{\binom {n-1}{k-1}} [x^n] \left((1-x)^{-t} - 1\right)^k \label{eq:P_n_k_coeff}\\
&=
\frac {1}{\binom {n-1}{k-1}}  \sum_{m=1}^k (-1)^{k-m} \binom k m  \frac{\Gamma(tm + n)}{\Gamma(tm) n!}
.
\label{eq:gener_funct_formula}
\end{align}
\end{lemma}
\begin{proof}
%Our starting point is the identity (Reference)
%$$
%\frac{k!}{n!}\stirling{n}{j} \stirlingsec{j}{k}  = [t^j] [x^n] \left((1-x)^{-t} - 1\right)^k.
%$$
To prove~\eqref{eq:P_n_k_coeff}, observe that by Lemma~\ref{lem:stirling_product},
$$
\frac{k!}{n!} \sum_{j=k}^n t^j \stirling{n}{j} \stirlingsec{j}{k}  =  [x^n] \left((1-x)^{-t} - 1\right)^k.
$$
Division by $\binom{n-1}{k-1}$ gives $\E t^{\Lah(n,k)}$. It remains to prove~\eqref{eq:gener_funct_formula}.
Using the binomial formula and~\eqref{eq:P_n_k_coeff}, we obtain
$$
P_{n,k}(t)
=
\frac {1}{\binom {n-1}{k-1}} [x^n] \sum_{m=0}^k (-1)^{k-m} \binom k m (1-x)^{-tm}
=
\frac {1}{\binom {n-1}{k-1}} \sum_{m=1}^k (-1)^{k-m} \binom k m [x^n] (1-x)^{-tm}
.
$$
We dropped the term  with $m=0$ since it vanishes. To complete the proof, recall the Taylor series
$$
(1-x)^{-tm} = \sum_{n=0}^\infty  \frac{tm (tm+1)(tm+2) \ldots (tm + n-1)}{n!} x^n
=
\sum_{n=0}^\infty    \frac{\Gamma(tm + n)}{\Gamma(tm) n!} x^n.
$$
\end{proof}

\subsection{Expectation and factorial moments of the Lah distribution}\label{sec:moments}
\subsubsection{Exact formulas for the expectation}
We are going to state exact formulas for the factorial moments of the Lah distribution or, more precisely, for expressions differing from the factorial moments by a missing normalizing factor of $1/L(n,k)$. We begin with the expectation.
\begin{theorem}[Expectation]\label{theo:lah_expect}
For all $n,k\in \N$ with $n\geq k$ we have
\begin{align}
\sum_{j=k}^n j \stirling{n}{j} \stirlingsec{j}{k}
&=
\frac{n!}{(k-1)!} [x^{n-k+1}] ((1+x)^n \log (1+x))\label{eq:expect_coeff1}\\
&=
(-1)^{n-k} \frac{n!}{(k-1)!} [x^{n-k+1}] ((1+x)^{-k} \log (1+x)).\label{eq:expect_coeff2}
\end{align}
Equivalently, expanding $\log(1+x)$ and the other terms in Taylor series  and multiplying out, we have
\begin{equation}\label{eq:expect_lah_binomial_sum}
\sum_{j=k}^n j \stirling{n}{j} \stirlingsec{j}{k}
=
\frac{n!}{(k-1)!} \sum_{i=1}^{n-k+1} \frac{1}{i} \binom {n-i}{k-1}
=
\frac{n!}{(k-1)!} \sum_{i=1}^{n-k+1} \frac{(-1)^{i+1}}{i} \binom {n}{k+i-1}.
\end{equation}
\end{theorem}
%All formulas from the above theorem have been verified numerically: Expectations Stirling Distributions.nb
\begin{proof}%[Proof of Theorem~\ref{theo:lah_expect}]
The starting point of the proof is the formula
$$
\frac{k!}{n!} \sum_{j=k}^n \stirling{n}{j} \stirlingsec{j}{k} t^j = [x^n] \left((1-x)^{-t} - 1\right)^k
$$
which follows from Lemma~\ref{lem:gener_funct_formula}. Since the function $((1-x)^{-t}-1)^k$ is analytic in $(x,t)$ if $(x,t)$ stays in a sufficiently small neighborhood of the point $(0,1)$, we may differentiate it any number of times in $x$ and $t$ and interchange the order of derivatives.
Differentiating the above formula in $t$ and putting $t=1$, we obtain
$$
\frac{k!}{n!} \sum_{j=k}^n  j \stirling{n}{j} \stirlingsec{j}{k}
=
- k  [x^n]\left( \left(\frac{x}{1-x}\right)^k \frac{\log (1-x)}{x}\right)
=
- k  [x^{n-k+1}] ((1-x)^{-k} \log (1-x)).
$$
Changing $x$ to $-x$, we obtain~\eqref{eq:expect_coeff1}. To prove~\eqref{eq:expect_coeff2}, we rewrite~\eqref{eq:expect_coeff1} using the Cauchy formula as
\begin{equation}\label{eq:oint_1}
\sum_{j=k}^n  j \stirling{n}{j} \stirlingsec{j}{k}
=
(-1)^{n-k} \frac{n!}{(k-1)!} \frac 1 {2\pi \ii} \oint_{\gamma} \frac{\log (1+x)}{(1+x)^k} \frac{\dd x}{x^{n-k+2}},
\end{equation}
where the integration contour $\gamma$ is a small counterclockwise circle centered at zero. Making the substitution $1+x = \frac 1 {1+y}$, we get
\begin{equation}\label{eq:oint_2}
\sum_{j=k}^n  j \stirling{n}{j} \stirlingsec{j}{k}
=
 \frac{n!}{(k-1)!} \frac 1 {2\pi \ii} \oint_{\gamma'} \frac{(1+y)^n\log (1+y)}{(1+y)^{n-k+2}} \dd y
\end{equation}
for some small counterclockwise contour $\gamma'$ around $0$.
Using the Cauchy formula one more time, we arrive at~\eqref{eq:expect_coeff2}.
\end{proof}

\begin{remark}
Alternatively, the first equality in~\eqref{eq:expect_lah_binomial_sum} can be derived from~\eqref{eq:representation_as_a_sum_over_blocks_of_composition} as follows.  First, note that for all $i=1,\ldots,k$ and all $j=1,\ldots,n-k+1$,
\begin{equation}\label{eq:b_1_n_distribution}
\P[b_i^{(n)}=j]=\P[b_1^{(n)}=j]=\frac{\binom{n-j-1}{k-2}}{\binom{n-1}{k-1}}.
\end{equation}
Thus, from~\eqref{eq:representation_as_a_sum_over_blocks_of_composition}, and with $H_k$ denoting the $k$-th harmonic number, we have
\begin{align*}
L(n,k)\E\Lah(n,k)&=\frac{n!}{k!}\binom{n-1}{k-1}\E\left(\E\left(\sum_{i=1}^{k}Z^{(i)}_{b_i^{(n)}}\Big|(b_1^{(n)},b_2^{(n)},\ldots,b_k^{(n)})\right)\right)=\frac{n!}{(k-1)!}\binom{n-1}{k-1}\E H_{b_1^{(n)}}\\
&=\frac{n!}{(k-1)!}\binom{n-1}{k-1}\sum_{j=1}^{n-k+1}H_j \P[b_1^{(n)}=j]=\frac{n!}{(k-1)!}\sum_{j=1}^{n-k+1}H_j \binom{n-j-1}{k-2}\\
&=\frac{n!}{(k-1)!}\sum_{j=1}^{n-k+1}\sum_{i=1}^{j}\frac{1}{i} \binom{n-j-1}{k-2}=\frac{n!}{(k-1)!}\sum_{i=1}^{n-k+1}\frac{1}{i}\sum_{j=i}^{n-k+1}\binom{n-j-1}{k-2}\\
&=\frac{n!}{(k-1)!}\sum_{i=1}^{n-k+1}\frac{1}{i}\binom{n-i}{k-1}.
\end{align*}
\end{remark}

\begin{remark}
The Narumi polynomials $s_{\ell;a}(z)$, $\ell\in \N_0$, with parameter $a\in \Z$ are defined by the formula
$$
\sum_{\ell=0}^\infty \frac{s_{\ell;a}(z)}{\ell!} t^\ell = \left(\frac{t}{\log(1+t)}\right)^{a} (1+t)^z;
$$
see~\cite{narumi}. With this notation,  Theorem~\ref{theo:lah_expect} takes the form
$$
\sum_{j=k}^n j \stirling{n}{j} \stirlingsec{j}{k} = (-1)^{n-k} k\binom{n}{k} s_{n-k; -1}(-k) = k\binom{n}{k}  s_{n-k; -1}(n).
$$
More generally, by taking the $p$-th derivative of the function $((1-x)^{-t} - 1)^k$ at $t=1$ it is possible to express the $p$-th factorial moment of the Lah distribution through the Narumi polynomials with $a = - p$.
For example, for the second factorial moment we get
$$
\sum_{j=k}^n j(j-1) \stirling{n}{j} \stirlingsec{j}{k}
=
\frac{(-1)^{n+k} n!s_{n-k;-2}(-k)}{(k-2)! (n-k)!} -   \frac{(-1)^{n+k} n!s_{n-k-1;-2}(-k)}{(k-1)! (n-k-1)!}.
$$
Expressions for higher factorial moments obtained in this way become more complicated, but we shall present a relatively simple general formula in Theorem~\ref{theo:fact_moments_lah}.  Note that the Narumi polynomials satisfy the functional equation $s_{\ell; a}(z) = s_{\ell;a} (\ell - a - 1 - z)$ which can be shown by using the Cauchy formula together with the same substitution as the one used to pass from~\eqref{eq:oint_1} to~\eqref{eq:oint_2}.
\end{remark}

\subsubsection{Exact formula for factorial moments}
The next theorem states a formula for the $p$-th factorial moment of the Lah distribution, up to a factor of $p!/L(n,k)$.
\begin{theorem}[Factorial moments]\label{theo:fact_moments_lah}
For all $n\in \N$, $k\in \{1,\ldots,n\}$ and $p\in \N$ we have
$$
\sum_{j=k}^n \stirling{n}{j} \stirlingsec{j}{k} \binom jp
=
n! \sum_{i=1}^{n-k+1} \binom{n-i}{k-1} \sum_{m=1}^{\min \{k,p\}} \frac{\stirlingsec{p}{m}}{(k-m)!}   \frac{\stirling{i+m-1}{p}}{(i+m-1)!} .
$$
\end{theorem}
%The above theorem has been verified numerically: Expectations Stirling Distributions.nb

For small values of $p$, this formula allows to express the $p$-th factorial moment of the Lah distribution  in terms of binomial coefficients and the generalized harmonic numbers
$$
H_n^{(m)} = \frac 1 {1^m} + \frac 1 {2^m} + \frac 1 {3^m} +\cdots + \frac 1 {n^m},
\qquad
H_n := H_n^{(1)}.  %= \frac 1 {1} + \frac 1 {2^m} + \frac 1 {3^m} +\ldots + \frac 1 {n^m}
$$
Indeed, if $p$ is ``small'', then the numbers $\stirlingsec{p}{m}$ on the right-hand side are explicit constants, while the numbers $\stirling{i}{p}$ can be expressed in terms of the generalized harmonic numbers, for example
$$
\stirling{i}{1} = (i-1)!,
\qquad
\stirling{i}{2} = (i-1)! H_{i-1},
\qquad
\stirling{i}{3} = \frac 12 (i-1)! \left((H_{i-1})^2 - H_{i-1}^{(2)}\right),\qquad \ldots.
$$
Specifically, for $p=1$ we recover the first formula in~\eqref{eq:expect_lah_binomial_sum}, while for $p=2$, we obtain after some straightforward computations the following expression (which can easily be combined with~\eqref{eq:expect_lah_binomial_sum} to write down an exact formula for the variance of the Lah distribution).
%The next corollary has been verified numerically: Expectations Stirling Distributions.nb
\begin{corollary}
For all $n\in \N$ and $k\in \{1,\ldots,n\}$ we have
$$
\sum_{j=k}^n j(j-1) \stirling{n}{j} \stirlingsec{j}{k}
=
%\frac{n!}{(k-1)!} \sum_{i=1}^{n-k+2} \left( \binom {n-i}{k-1}\frac{2 H_{i-1} + 1}{i} + 2(k-1) \binom {n-i+1}{k-1}\frac{H_{i-1}}{i}\right) .
\frac{2\cdot n!}{(k-1)!} \sum_{i=1}^{n-k+1}  \binom {n-i}{k-1} \left(\frac{ H_{i}\cdot (1+ik)}{i(i+1)} - \frac{1}{i^2}\right).
$$
%\begin{align*}
%\sum_{j=k}^n j \stirling{n}{j} \stirlingsec{j}{k}
%&=
%\frac{n!}{(k-1)!} \sum_{m=k}^n \binom {m-1}{k-1} \frac{1}{n-m+1}.
%\frac{n!}{(k-1)!} \sum_{i=1}^{n-k+1} \binom {n-i}{k-1} \frac{1}{i},\\
%If, additionally, %$k\geq 2$, then
%\sum_{j=k}^n j(j-1) \stirling{n}{j} \stirlingsec{j}{k}
%=
%\frac{n!}{(k-1)!} \sum_{i=1}^{n-k+2} \left( \binom {n-i}{k-1}\frac{2 H_{i-1} + 1}{i} + 2(k-1) \binom {n-i+1}{k-1}\frac{H_{i-1}}{i}\right) .
%\frac{2\cdot n!}{(k-1)!} \sum_{i=1}^{n-k+1}  \binom {n-i}{k-1} \left(\frac{ H_{i}\cdot (1+ik)}{i(i+1)} - \frac{1}{i^2}\right).
%\sum_{j=k}^n j(j-1)(j-2) \stirling{n}{j} \stirlingsec{j}{k}
%&=
%%\frac{n!}{(k-1)!} \sum_{i=1}^{n-k+2} \left( \binom {n-i}{k-1}\frac{2 H_{i-1} + 1}{i} + 2(k-1) \binom {n-i+1}{k-1}\frac{H_{i-1}}{i}\right) .
%\frac{6\cdot n!}{(k-1)!} \sum_{i=1}^{n-k+1}  \binom {n-i}{k-1} \left(S_i \cdot\left(\frac{3k-6}{i+1} + \frac{(k-2)(k-1)}{i+2} + \frac 1 {i-1} %\right)\right. \\
%&\qquad + \left.H_i\cdot \left( \frac{(k-2)(k-1)}{(i+2)(i+1)} - \frac 1 {i(i-1)}\right) +\frac{1}{i^2(i-1)}\right),
%\end{align*}
\end{corollary}
Let us also mention that for $k=1$, the identity of Theorem~\ref{theo:fact_moments_lah} takes the following form: for all $p\in \N$ and $n\in \N$,
$$
\frac 1 {n!} \sum_{j=1}^{n} \stirling{n}{j} \binom{j}{p} = \sum_{j=1}^n \frac{1}{j!} \stirling{j}{p}.
$$
This equality is well known and, in fact, both sides are equal to $\frac 1{n!} \stirling{n+1}{p+1}$; see Entries (6.15) and (6.21) of~\cite{Graham1994}.

\begin{proof}[Proof of Theorem~\ref{theo:fact_moments_lah}]
Let $D_x^n \Big|_{x=x_0} f(x)$ denote the $n$-th derivative of a function $f$ evaluated at $x=x_0$. The starting point of the proof is the formula
$$
k!\sum_{j=k}^n \stirling{n}{j} \stirlingsec{j}{k} t^j = D_x^n \Big|_{x=0} \left((1-x)^{-t} - 1\right)^k
$$
which follows from Lemma~\ref{lem:gener_funct_formula}.
Taking the $p$-th derivative at $t=1$ we arrive at
$$
\sum_{j=k}^n \stirling{n}{j} \stirlingsec{j}{k} j^{\underline {p}}
=
D_x^n \Big|_{x=0} D_t^p \Big|_{t=1}  \frac{\left((1-x)^{-t} - 1\right)^k}{k!},
$$
where
$$
j^{\underline{p}} := j(j-1) \cdots (j-p+1)
$$
denotes the falling factorial.
 Our next goal is to write the right-hand side as a function of $1-t$. Extracting the factor $1/(1-x)$ and using the binomial formula, we arrive at
\begin{align*}
\sum_{j=k}^n \stirling{n}{j} \stirlingsec{j}{k} j^{\underline{p}}
&=
D_x^n \Big|_{x=0} D_t^p \Big|_{t=1}
\left(
\frac 1 {(1-x)^{k}}  \frac{\left((1-x)^{1-t} - 1 + x\right)^k} {k!}
\right)\\
&=
D_x^n \Big|_{x=0} D_t^p \Big|_{t=1}
\left(\frac 1 {(1-x)^{k}}
\sum_{m=0}^k \frac1 {k!} \binom km ((1-x)^{1-t}-1)^m x^{k-m}
\right)
\\
&=
D_x^n \Big|_{x=0} D_t^p \Big|_{t=1}
\left(\frac 1 {(1-x)^{k}}
\sum_{m=0}^k \frac{((1-x)^{1-t}-1)^m}{m!} \frac{x^{k-m}}{(k-m)!}
\right).
\end{align*}
Introducing the variable $s:=t-1$, we can write
\begin{align*}
\sum_{j=k}^n \stirling{n}{j} \stirlingsec{j}{k} j^{\underline{p}}
&=
D_x^n \Big|_{x=0} D_s^p \Big|_{s=0}
\left(\frac 1 {(1-x)^{k}}
\sum_{m=0}^k \frac{((1-x)^{-s}-1)^m}{m!} \frac{x^{k-m}}{(k-m)!}
\right)\\
&= D_x^n \Big|_{x=0}
\left(\frac 1 {(1-x)^{k}}
\sum_{m=0}^k \frac{x^{k-m}}{(k-m)!}  D_s^p \Big|_{s=0} \frac{(\eee^{s\log \frac 1 {1-x} }-1)^m}{m!}
\right)\\
&= D_x^n \Big|_{x=0}
\left(\frac 1 {(1-x)^{k}}
\sum_{m=0}^k
 \frac{x^{k-m}}{(k-m)!} \left(\log \frac 1 {1-x}\right)^p \stirlingsec{p}{m}\right),
\end{align*}
where we have used the second relation in~\eqref{eq:stirling_gen_funct} for the last passage. Interchanging the order of summation, we obtain
\begin{align*}
\sum_{j=k}^n \stirling{n}{j} \stirlingsec{j}{k} j^{\underline{p}}
&=
\sum_{m=0}^k
\stirlingsec{p}{m} \frac{1}{(k-m)!}
D_x^n \Big|_{x=0}
\left(\frac 1 {(1-x)^{k}}
x^{k-m}
\left(\log \frac 1 {1-x}\right)^p\right)\\
&=
n! \sum_{m=0}^k
\stirlingsec{p}{m} \frac{1}{(k-m)!}
[x^{n-k+m}] \left(\frac 1 {(1-x)^{k}}
\left(\log \frac 1 {1-x}\right)^p\right).
\end{align*}
Now we use the formulas
$$
\left(\log \frac 1 {1-x}\right)^p = p! \sum_{i=p}^\infty \frac{x^i}{i!} \stirling{i}{p},
\qquad
\frac 1 {(1-x)^{k}} = \sum_{j=0}^\infty x^j \binom {j+k-1}{j}.
$$
Multiplying these two series and evaluating the coefficient of $x^{n-k+m}$, we get
\begin{align*}
\sum_{j=k}^n \stirling{n}{j} \stirlingsec{j}{k} j^{\underline{p}}
&=
n! p!
\sum_{m=0}^k
\stirlingsec{p}{m} \frac{1}{(k-m)!}
\sum_{i=p}^{n-k+m} \stirling{i}{p} \frac 1 {i!} \binom {n+m-i-1}{k-1}\\
&=
n! p!
\sum_{m=0}^k \sum_{i=p}^{n-k+m}
\stirlingsec{p}{m} \stirling{i}{p} \frac{1}{i!(k-m)!}
\binom {n+m-i-1}{k-1}.
\end{align*}
Observe that the summation range in the first sum can be replaced by $m\in \{1,\ldots, \min \{k,p\}\}$ because $\stirlingsec{p}{m} = 0$ for $m=0$ and $m>p$. After dividing by $p!$ this yields
$$
 \sum_{j=k}^n \stirling{n}{j} \stirlingsec{j}{k} \binom jp
=
n! \sum_{m=1}^{\min \{k,p\}} \frac{1}{(k-m)!} \stirlingsec{p}{m}  \sum_{i=p}^{n-k+m} \frac{1}{i!} \stirling{i}{p}  \binom {n+m-i-1}{k-1}.
$$
To complete the proof, introduce the summation index $j=i-m+1$ and interchange the order of summation.
\end{proof}

\subsubsection{Asymptotics of the expectation}
Based on the exact expression given in Theorem~\ref{theo:lah_expect}, we are able to derive the following
\begin{theorem}[Asymptotics of the expectation]\label{theo:expect_asympt}
Let $n\to\infty$ and $k=k(n)\in \{1,\ldots,n\}$ be a function of $n$. Then,
\begin{equation}\label{eq:lah_distr_expect_asympt}
\E \Lah(n,k) \sim
\begin{cases}
k \log (n/k), &\text{ if $k=o(n)$,}\\
\frac{\alpha \log (1/\alpha)}{1-\alpha} \cdot n, &\text{ if $k \sim \alpha n$ for some $\alpha \in (0,1)$,}\\
n, &\text{ if $k \sim n$.}
\end{cases}
\end{equation}
We write $a_n\sim b_n$ if $a_n/b_n\to 1$ as $n\to\infty$.
%\begin{equation}\label{eq:lah_distr_expect_asympt1}
%\E \Lah(n,k) \sim -\frac{\alpha n  \log \alpha}{1-\alpha}.
%\end{equation}
%If $k=k(n) \sim n$ as $n\to\infty$, then
%\begin{equation}\label{eq:lah_distr_expect_asympt1a}
%\E \Lah(n,k) n.
%\end{equation}
%Finally, if $k=o(n)$ as $n\to\infty$, then
%\begin{equation}\label{eq:lah_distr_expect_asympt2}
%\E \Lah(n,k) \sim k \log (n/k).
%\end{equation}
\end{theorem}
\begin{proof}
According to~\eqref{eq:expect_lah_binomial_sum}, see also~\eqref{eq:lah_numb_def}, we have
\begin{equation}\label{eq:lah_expect_transform}
\E \Lah(n,k)
=
k \sum_{i=1}^{n-k+1} \frac 1 i \frac{\binom {n-i}{k-1}}{\binom{n-1}{k-1}}
%&=
%k \sum_{i=1}^{n-k+1} \frac 1 i \frac{(n-i)(n-i-1)\ldots (n-i-k+2)}{(n-1)(n-2)\ldots (n-k+1)}\\
=
k \sum_{i=1}^{n-k+1} \frac 1 i
\prod_{m=1}^{i-1} \frac{n-k-m+1}{n-m}.
%\frac{(n-k)(n-k-1)\ldots (n-k-i+2)}{(n-1)(n-2)\ldots (n-i+1)}.
\end{equation}
Now suppose that  $k\sim  \alpha n$ for some $\alpha \in (0,1)$ or $\alpha = 1$.  If $i\in \N$ is fixed, then
$$
\lim_{n\to\infty} \frac 1 i
\prod_{m=1}^{i-1} \frac{n-k-m+1}{n-m}
%\frac{(n-k)(n-k-1)\ldots (n-k-i+2)}{(n-1)(n-2)\ldots (n-i+1)}
=  \frac {(1-\alpha)^{i-1}}{i}.
$$
Moreover, we have the bound
$$
\max_{m=1,\ldots,n-k} \frac{n-k-m+1}{n-m} \leq \frac {n-k+1}{n} \ton 1-\alpha.
$$
It follows that for some sufficiently small $\eps>0$ and all sufficiently large $n$ we have
$$
\frac 1 i \prod_{m=1}^{i-1} \frac{n-k-m+1}{n-m}
%\frac{(n-k)(n-k-1)\ldots (n-k-i+2)}{(n-1)(n-2)\ldots (n-i+1)}
% \leq \frac 1i \left(\frac{n-k+1}{n}\right)^{i-1}
\leq \frac {(1-\eps)^{i-1}}{i}, \qquad
i\in \{1,\ldots, n-k+1\}.
$$
Note that the right-hand side is summable in $i\in \N$. Interchanging the limit and the sum by the Lebesgue dominated convergence theorem, we obtain
$$
\lim_{n\to\infty}
\sum_{i=1}^{n-k+1} \frac 1 i
\prod_{m=1}^{i-1} \frac{n-k-m+1}{n-m}
%\frac{(n-k)(n-k-1)\ldots (n-k-i+2)}{(n-1)(n-2)\ldots (n-i+1)}
=
\sum_{i=1}^\infty \frac {(1-\alpha)^{i-1}}{i}
=
\begin{cases}
\frac{-\log \alpha}{1-\alpha}, & \text{ if } \alpha \in (0,1),\\
1, &\text{ if } \alpha = 1.
\end{cases}
$$

Let us now consider the case $k=o(n)$. Note that we do not require that $k\to\infty$. The idea is to show that in the sum on the right-hand side of~\eqref{eq:lah_expect_transform}, the summands with $i<n/k$ are approximately equal to $1/i$, whereas the contribution of the remaining summands is $O(1)$. Take some large constant $A>1$ and let $n $ be sufficiently large in the following. We start with a lower estimate. Recalling~\eqref{eq:lah_expect_transform}, dropping summands with $i>n/(Ak)$ and using the inequality $\prod_{j=1}^M (1-x_j) \geq 1- \sum_{j=1}^M x_j$ which is valid for arbitrary numbers $x_1,\ldots,x_M\in [0,1]$  and can be easily established by induction, we get
$$
\frac 1k \E \Lah(n,k)
=
\sum_{i=1}^{n-k+1} \frac 1 i
\prod_{m=1}^{i-1} \left( 1 - \frac{k-1}{n-m}\right)
\geq
\sum_{i=1}^{[n/(A k)]} \frac 1 i
\left(1 - \sum_{m=1}^{i-1} \frac{k-1}{n-m}\right).
$$
For $1\leq m < i \leq n/(Ak)$ we have  $n-m\geq n - n/(Ak) \geq n/2$ and hence
$$
\sum_{m=1}^{i-1} \frac{k-1}{n-m} \leq \sum_{m=1}^{i-1} \frac {k}{n/2} \leq \frac{2ki}{n} \leq \frac 2 A.
$$
It follows that
$$
\frac 1k \E \Lah(n,k)
\geq
\left(1 - \frac 2A\right) \sum_{i=1}^{[n/(A k)]} \frac 1 i = \left(1 - \frac 2A\right) \log \frac nk - O_A(1).
$$
Since $\log (n/k)\to\infty$ and $A$ can be arbitrarily large, we arrive at the lower bound
$$
\liminf_{n\to\infty} \frac{\E \Lah(n,k)}{ k \log (n/k)} \geq 1.
$$
To prove the upper bound, we shall combine~\eqref{eq:representation_as_a_sum_over_blocks_of_composition} and the elementary estimate $H_n=\log n+O(1)$, as follows:
$$
\E \Lah(n,k)=k \E H_{b_1^{(n)}}=k (\E \log b_1^{(n)} +O(1))\leq k (\log \E b_1^{(n)} +O(1))=k (\log (n/k) +O(1)),
$$
where we have used Jensen's inequality and the fact that
$
n=\E (b_1^{(n)}+\cdots+b_k^{(n)})=k\E b_1^{(n)}.
$
Therefore,
$$
\limsup_{n\to\infty} \frac{\E \Lah(n,k)}{ k \log (n/k)} \leq 1,
$$
which completes the proof.
\end{proof}
\begin{remark}
Let us mention a strange connection of~\eqref{eq:lah_distr_expect_asympt} to a seemingly unrelated problem studied in~\cite{cilleruelo}.
Let $L_n$ be the least common multiple of a  random set $A_n$ of integers obtained in the following way: every number from the set $\{1,\ldots,n\}$ is included in $A_n$ with probability $\alpha\in (0,1)$, independently from the others. Then, Theorem~1.1 in~\cite{cilleruelo} states that
$$
\frac{\log L_n}{n} \ton \frac{\alpha  \log (1/\alpha)}{1-\alpha}
\quad
\text{ in probability}.
$$
The expression on the right-hand side is the same as in~\eqref{eq:lah_distr_expect_asympt}. Moreover, Theorem~1.2 in~\cite{cilleruelo} bears similarity with the $k=o(n)$ case of~\eqref{eq:lah_distr_expect_asympt}.  We were not able to explain this coincidence. A central limit theorem for $\log L_n$ has been proved in~\cite[Corollary~1.5]{alsmeyer_kabluchko_Marynych_LCM}. The asymptotic variance given in~\cite[Remark~1.3]{alsmeyer_kabluchko_Marynych_LCM} does not coincide with the asymptotic variance of the Lah distribution given in Theorem~\ref{thm:clt_central_regime}.
\end{remark}

\subsection{Log-concavity and unimodality}

Well known properties of the Stirling numbers of both kinds yield the following proposition.

\begin{proposition}\label{prop:log_concave}
For each $n\in \N$ and $k\in \{1,\ldots,n\}$, the Lah distribution $\Lah (n,k)$ is log-concave, that is
$$
\P[\Lah(n,k) = i]^2 \geq \P[\Lah(n,k) = i-1]\P[\Lah(n,k) = i+1] \qquad \text{ for all } i\in \{k,\ldots,n\}.
$$
\end{proposition}
\begin{proof}
It is well known (see, e.g., \cite[Corollary~3.2]{sibuya}) that the Stirling numbers are log-concave, that is
$$
\stirling{n}{j}^2 \geq \stirling{n}{j+1}\stirling{n}{j-1}.
$$
By~\cite[Theorem~3.3]{sibuya}, the sequence $(\stirlingsec{j+1}{k}/\stirlingsec{j}{k})_{j=k,k+1,\ldots}$ is strictly decreasing for every $k\geq 2$ (and is identically equal to $1$ for $k=1$)  which means that
$$
\stirlingsec{j}{k}^2 \geq \stirlingsec{j+1}{k}\stirlingsec{j-1}{k},
$$
with a strict inequality for $k\geq 2$.  Multiplying these two inequalities, we arrive at the claim.
\end{proof}

\begin{corollary}\label{cor:unimodal}
For each $n\in \N$ and $k\in \{1,\ldots,n\}$, the Lah distribution $\Lah(n,k)$ is unimodal. That is, there exists $m_{n,k}\in \{k,\ldots,n\}$ such that $ i\mapsto \P[\Lah(n,k)=i]$ is nondecreasing for $i\leq m_{n,k}$ and nonincreasing for $i\geq m_{n,k}$.
\end{corollary}

\subsection{Zeroes of the generating function}
In the following, we shall prove central limit theorems for the Lah distribution.  A natural approach towards this is the Harper method~\cite{harper}, for which one needs to verify that the zeroes of $P_{n,k}(t)$ are real and negative. Numerical simulations, see Figure~\ref{fig:zeroes}, show that the zeroes are not real except in the special case $k=1$ and suggest the following
\begin{conjecture}
All complex zeroes of $P_{n,k}$ have negative real parts.
\end{conjecture}
In fact, this conjecture would also be sufficient to apply Harper's method, see, e.g., \cite[Theorem~3.1]{lebowitz_etal}. We shall not investigate the properties of zeroes here and mention only one result. In the special case $z=1$ it is well known.

\begin{proposition}
Let $k\in \N$ and let $n>k$ be integer. Then, for every $z\in \{1,2,\ldots, \lfloor \frac{n-1}{k}\rfloor\}$ we have
$$
\sum_{j=k}^n \stirling{n}{j} \stirlingsec{j}{k} (-z)^j =0.
$$
\end{proposition}
%Verified numerically: Expectations Stirling Distribution.nb
\begin{proof}
It suffices to show that $P_{n,k}(-z) = 0$. We use~\eqref{eq:P_n_k_coeff}.
Note that $(1-x)^{z}$ is  a polynomial in $x$ of degree $z$. Hence, $((1-x)^{z}-1)^k$ is a polynomial in $x$ of degree $z k$. If $z\leq \lfloor \frac{n-1}{k}\rfloor$, then $zk < n$ and the coefficient of $x^n$ in this polynomial vanishes. Then,~\eqref{eq:P_n_k_coeff} implies that $P_{n,k}(-z)=0$.
\end{proof}

\begin{figure}[t]
\begin{center}
\includegraphics[width=0.49\textwidth ]{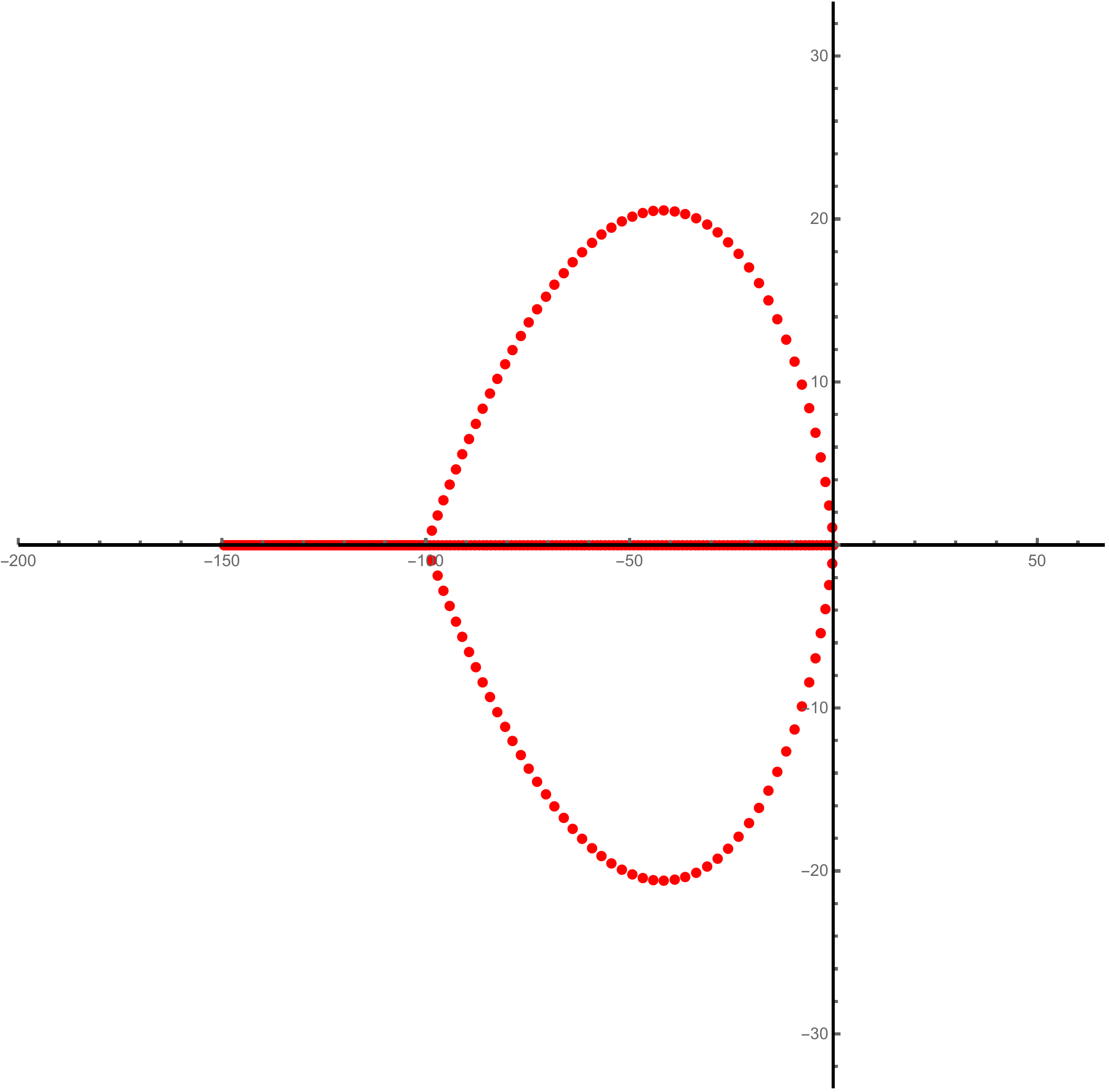}
\includegraphics[width=0.49\textwidth ]{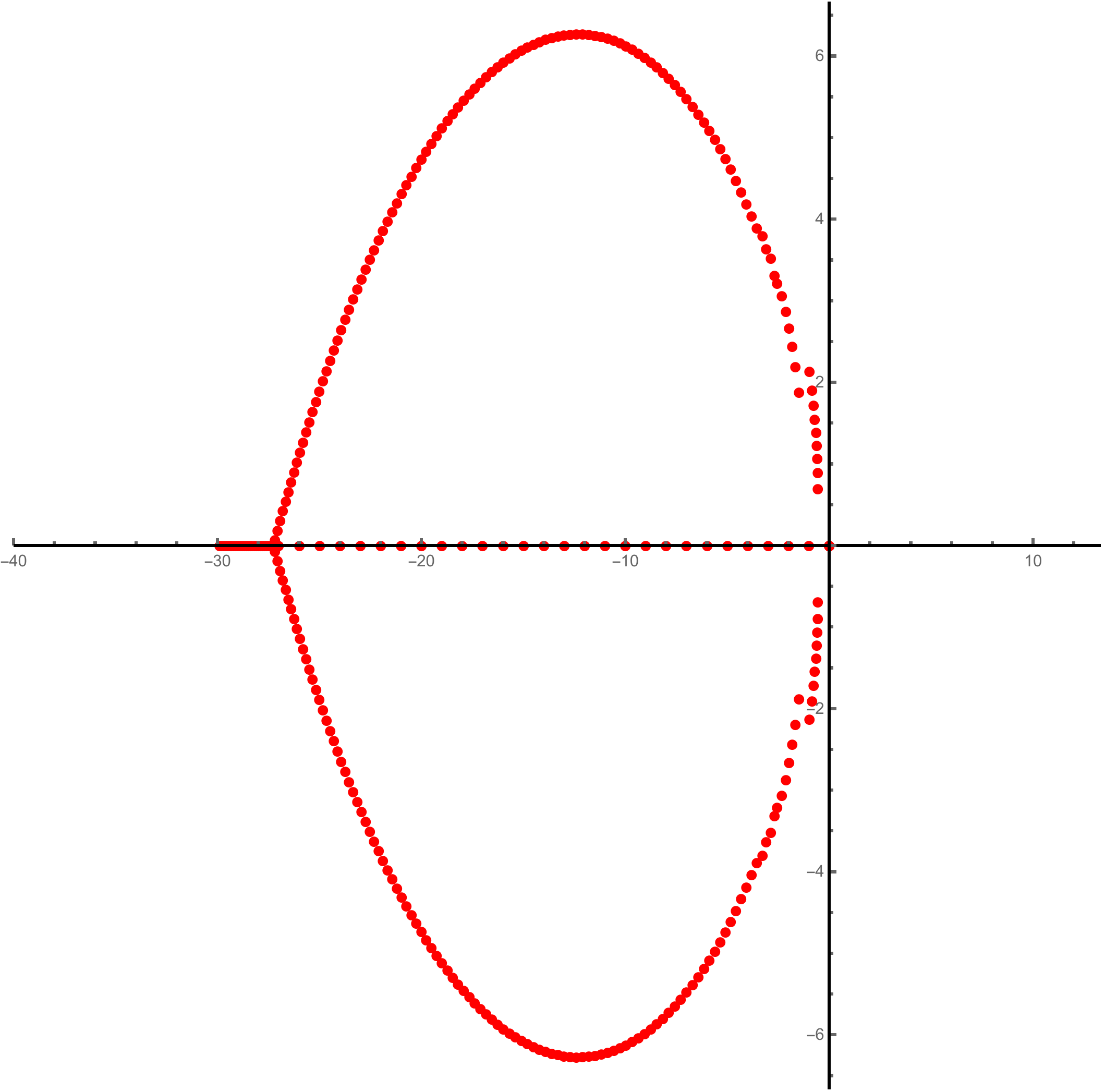}
%Expectations Stirling distribution.nb
\end{center}
\caption
{
Complex zeroes of the polynomial $P_{n,k}$ with $n=300$ and $k=2$ (left) and $k=10$ (right).
}
\label{fig:zeroes}
\end{figure}

\section{Limit theorems for the Lah distribution: The constant \texorpdfstring{$k$}{k} regime}\label{sec:constant_k}

\subsection{Mod-Poisson convergence and its consequences}
In this section we shall state and prove limit theorems for the random variables $\Lah(n,k)$ in the regime when $k\in \N$ is fixed and $n\to\infty$. The basic tool we shall use is the notion of mod-Poisson convergence introduced by Kowalski and Nikeghbali in~\cite{kowalski_nikeghbali_poisson}; see also~\cite{barbour_kowalski_nikeghbali,delbaen_kowalski_nikeghbali,jacod_kowalski_nikeghbali_mod_Gauss,kowalski_nikeghbali_zeta,meliot_nikeghbali_statmech} for a more general notion of mod-$\phi$-convergence and~\cite{feray_meliot_nikeghbali_book} for a monograph treatment of the subject.

Let $(X_n)_{n\in\N}$ be a sequence of random variables with values in $\R$ whose Laplace transforms $\E \eee^{z X_n}$ exist finitely for all $z\in \C$ and $(\lambda_n)_{n\in\N}$ a sequence of positive numbers with $\lim_{n\to\infty} \lambda_n = +\infty$.  The sequence $(X_n)_{n\in\N}$ is said to \textit{converge in the mod-Poisson sense} with speed $(\lambda_n)_{n\in \N}$ if
\begin{equation}\label{eq:mod_poi_def}
\lim_{n\to\infty} \frac{\E\eee^{zX_n}}{\eee^{\lambda_n(\eee^z-1)}} = \Psi(z)
\end{equation}
uniformly on compact subsets of some open set $\mathcal D\subset \C$ containing the real axis. Here, $\Psi:\mathcal D \to\C$ is some analytic function. In the literature, several non-equivalent definitions of mod-Poisson (and, more generally, mod-$\phi$) convergence exist, which differ by the shape of the domain $\mathcal{D}$. The notion we use here is close but not equivalent to the definition used in the book~\cite{feray_meliot_nikeghbali_book}, see Definition~1.1.1 therein, where $\mathcal{D}$ is assumed to be a vertical strip containing the imaginary axis. Nevertheless, most important corollaries of the mod-$\phi$ convergence continue to  hold under the assumption that $\mathcal D\subset \C$ is an open set containing a segment of the real line, see~\cite[Remark 2.10]{kabluchko_marynych_sulzbach_edgeworth}. As we shall see below in Theorem~\ref{theo:mod_poi}, definition~\eqref{eq:mod_poi_def} is more suitable for the Lah distribution.

To interpret the above definition, recall that the generating function of the Poisson distributed random variable with parameter $\lambda_n$ is given by
$$
\E \eee^{z \text{Poi} (\lambda_n)} = \eee^{\lambda_n (\eee^z - 1)},
$$
which is the denominator in~\eqref{eq:mod_poi_def}. Thus,~\eqref{eq:mod_poi_def} states the heuristic approximation
\begin{equation}\label{eq:mod_Poi_heuristic}
\mlq\mlq X_n \eqdistr \text{Poi}(\lambda_n) + \Xi +  o(1) \mrq\mrq,
\end{equation}
where $\Xi$ is a ``random variable'' with ``moment generating function'' $\E \eee^{z \Xi} = \Psi(z)$ that is independent of $\text{Poi}(\lambda_n)$, and $o(1)$ converges to $0$ in distribution. Even though usually no random variable $\Xi$ having the required moment generating function $\Psi(z)$ exists, a lot of limit theorems for $X_n$ have the same form as they would do for the sequence of ``random variables'' $\text{Poi}(\lambda_n) + \Xi$.

The next theorem states that for fixed $k\in \N$, the random variables $(\Lah(n,k))_{n\in \N}$ converge in the mod-Poisson sense with speed $\lambda_n = k\log n$, as $n\to\infty$.
\begin{theorem}[Mod-Poisson convergence]\label{theo:mod_poi}
Let $k\in \N$ be fixed.
Then,
\begin{equation}\label{eq:mod_poi}
\lim_{n\to\infty} \frac{\E \eee^{z \Lah(n,k)}}{\eee^{(k \log n) (\eee^z - 1)}} = \frac{\Gamma(k)}{\Gamma(k \eee^z)}
\end{equation}
for every $z\in \mathcal D_{Lah}$, where $\mathcal D_{Lah} := \{t\in \C: \cos \Im t >0\}\supset \R$. Moreover, this convergence is uniform as long as $z$ stays in any compact subset $K$ of $\mathcal D_{Lah}$, and the speed of convergence is $O(n^{-\eps(K)})$ for some $\eps(K)>0$.
\end{theorem}

The proof of Theorem~\ref{theo:mod_poi} is postponed to Section~\ref{subsec:proof_mod_poi}.

\begin{theorem}[Central limit theorem]\label{theo:clt}
Let $k\in \N$ be fixed.  Then,
%For every $n\geq k$ let $X_n\sim \Lah(n,k)$ be a Lah-distributed random variable. Then,
$$
\frac{\Lah(n,k) - k \log n}{\sqrt{k\log n}} \todistr {\rm N}(0,1).
$$
\end{theorem}
\begin{proof}
The claim follows from the mod-Poisson convergence~\eqref{eq:mod_poi} by~\cite[Proposition~2.4(2)]{kowalski_nikeghbali_poisson}. Note that the cited result only requires uniformity of the convergence in a small neighborhood of the origin which is secured by~\eqref{eq:mod_poi}.
\end{proof}

All subsequent results of this section follow essentially from the corresponding general results on random profiles obtained in~\cite{kabluchko_marynych_sulzbach_edgeworth}. This reference better fits our needs since we  have uniform convergence in a horizontal strip rather than a vertical one, preventing us from referring to the standard results on the mod-$\phi$ convergence. Note that Assumptions A1--A3 of~\cite{kabluchko_marynych_sulzbach_edgeworth} can be easily verified to hold with
\begin{multline}\label{eq:kabluchko_marynych_sulzbach_edgeworth_assumptions}
\mathbb{L}_n(j)=\P[\Lah(n,k) = j],\quad  w_n=k\log n,\quad \beta_{\pm}=\pm\infty\quad \mathcal{D}=\R \times (-\pi \ii/2,+\pi \ii/2),\\
\quad\text{and}\quad \phi(\beta)=\eee^{\beta}-1,\quad W_{\infty}({\beta})=\Gamma(k)/\Gamma(k \eee^{\beta})\quad \text{for}\quad \beta\in\mathcal{D},
\end{multline}
whereas Assumption~A4 will be checked in Remark~\ref{rem:A4}.

\begin{theorem}[Local limit theorem]\label{theo:local_limit_theo}
For every fixed $k\in\N$ we have
$$
\lim_{n\to\infty} \sqrt {\log n }\sup_{m\in \Z} \left|\P[\Lah(n,k) = m] - \frac{1}{\sqrt{2\pi k \log n}} \exp\left\{ -  \frac{(m-k\log n)^2}{2
k\log n} \right\} \right|= 0.
$$
\end{theorem}
\begin{proof}
In view of~\eqref{eq:mod_poi} and~\eqref{eq:kabluchko_marynych_sulzbach_edgeworth_assumptions} this follows from Theorem~2.7 of~\cite{kabluchko_marynych_sulzbach_edgeworth}. Moreover, a general Edgeworth asymptotic expansion with an arbitrary number of terms could be derived from~\cite[Theorem~2.1]{kabluchko_marynych_sulzbach_edgeworth}.
\end{proof}

An illustration of Theorems~\ref{theo:clt} and~\ref{theo:local_limit_theo} is shown in Figure~\ref{fig:clt_constant_k}.

\begin{theorem}[Precise asymptotics of large deviations]\label{theo:precise_ldp}
Let $(x_n)_{n\in \N}$ be a sequence of real numbers converging to $x>0$ and such that $kx_n\log n$ is integer for all $n\in \N$. Then,
\begin{align*}
\P[\Lah(n,k) =  k x_n \log n]
&\sim
\frac{n^{-k(x_n \log x_n -x_n+1)}}{\sqrt{2\pi k x \log n}} \frac{\Gamma(k)}{\Gamma(kx)},\\
\P[\Lah(n,k) \geq  k x_n \log n]
&\sim
\frac{x}{x-1} \frac{n^{-k(x_n \log x_n -x_n+1)}}{\sqrt{2\pi k x \log n}} \frac{\Gamma(k)}{\Gamma(kx)}, \qquad \text{ if } x>1,\\
\P[\Lah(n,k) \leq k x_n \log n]
&\sim
\frac{1}{1-x} \frac{n^{-k(x_n \log x_n -x_n+1)}}{\sqrt{2\pi k x \log n}} \frac{\Gamma(k)}{\Gamma(kx)}, \qquad \text{ if } x<1.
\end{align*}
\end{theorem}
\begin{proof}
The first claim follows from~Theorem 2.8 in \cite{kabluchko_marynych_sulzbach_edgeworth} applied with $r=0$, $K$ being an arbitrary segment of the real line which contains $x$, and $\beta_n(kx_n\log n)=\log x_n$. The second claim follows by  summation, c.f.~\cite[Theorem~3.2.2]{feray_meliot_nikeghbali_book}. The last claim follows from the same results applied to  $-\Lah (n,k)$.
\end{proof}

\begin{proposition}[Location of the mode]\label{prop:mode}
For every fixed $k\in\N$ there is $N_1\in \N$ such that for all integer $n>N_1$, all maximizers of the function $m\mapsto \P[\Lah(n,k) = m]$ are among the following two numbers:
$$
\left\lfloor k\log n - \frac{k\Gamma'(k)}{\Gamma(k)} -\frac 12 \right\rfloor, \quad \left \lceil k\log n - \frac{k\Gamma'(k)}{\Gamma(k)} -\frac 12\right\rceil.
$$
\end{proposition}

\begin{proof}
This follows from Theorem 2.11 of~\cite{kabluchko_marynych_sulzbach_edgeworth}.
\end{proof}

\begin{remark}
A random variable $X^{\theta}_n$ has the Ewens or the Karamata-Stirling distribution with parameters $n\in\N$ and $\theta>0$ if
$$
\P[X^{\theta}_n = j] = \frac{\theta^j}{\theta(\theta+1) \cdots (\theta + n - 1)} \stirling{n}{j}, \qquad j\in \{1,\ldots, n\}.
$$
It is well known that $X^{\theta}_n$ has the same distribution as the number of cycles in the Ewens random permutation. Coincidentally, if $\theta=k$ happens to be integer, the sequence $(X^{\theta}_n)_{n\in \N}$ satisfies the same mod-Poisson convergence as $(\Lah (n,k))_{n\in\N}$; see~\cite{kabluchko_marynych_sulzbach_mode}. Let us mention that the Lah distribution could be generalized by introducing an additional parameter $\theta>0$ (the probability that the random variable $\Lah(n,k;\theta)$ takes the value $j$ is by definition  proportional to $\theta^j\stirling{n}{j} \stirlingsec{j}{k}$, for $j\in \{k,\ldots,n\}$). Most of our results could  be  generalized to arbitrary $\theta>0$, but since we have no applications for this general setting,  we restrict ourselves to the case $\theta=1$.

%Large deviations: \cite{favaro_feng}, \cite{feng}.
\end{remark}

\begin{figure}[t]
\begin{center}
\includegraphics[width=0.7\textwidth ]{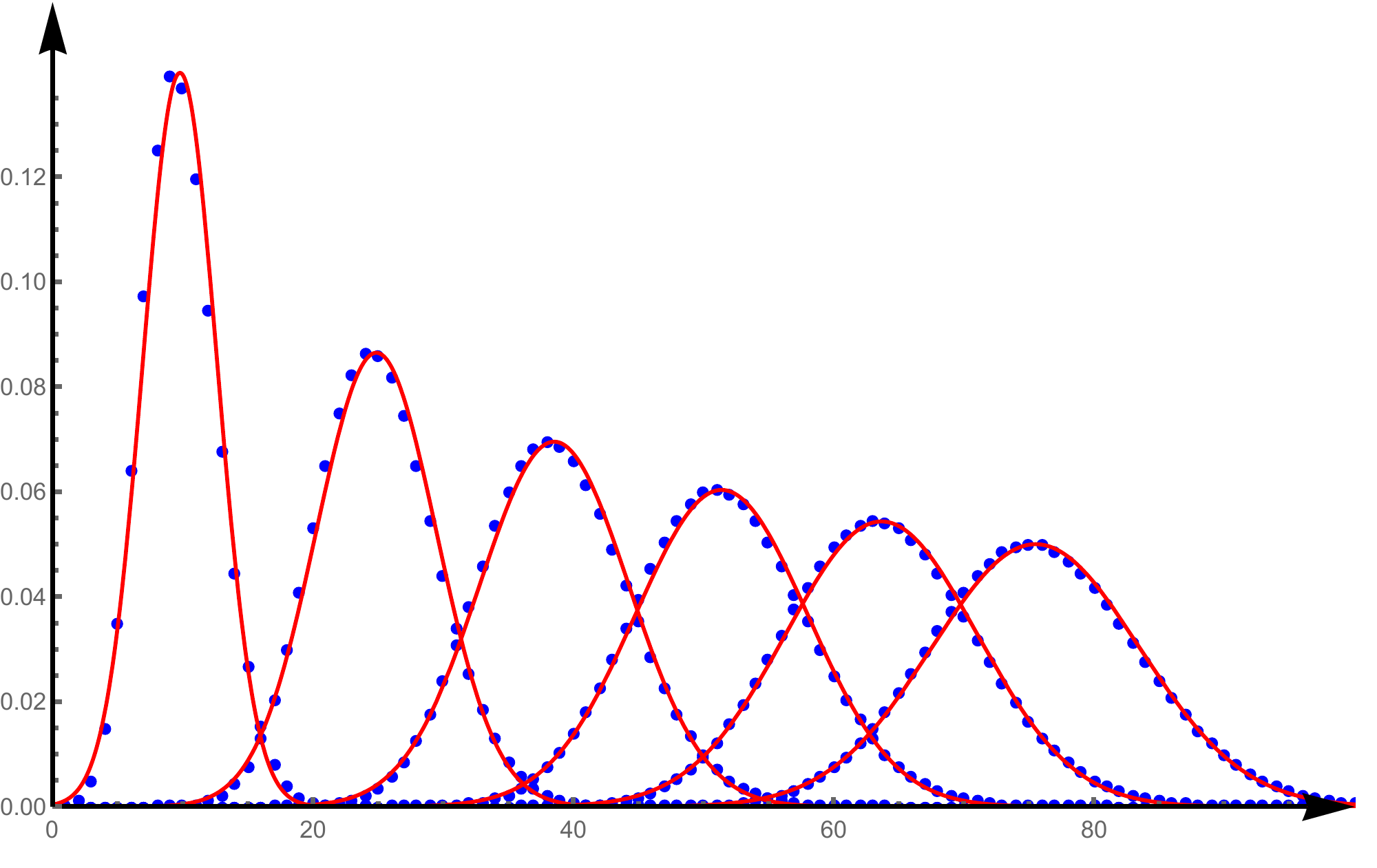}
%Lah Distributions 1.nb
\end{center}
\caption
{
The probability mass functions of Lah distributions with $n=10000$ and $k\in \{1,3,5,\ldots,11\}$ (blue dots) together with the approximating normal densities (red curves); see Theorems~\ref{theo:clt} and~\ref{theo:local_limit_theo}. The parameters of the normal densities are the true expectation and variance of the Lah distribution (for the normalization given in Theorems~\ref{theo:clt} and~\ref{theo:local_limit_theo}, the convergence is rather slow).
}
\label{fig:clt_constant_k}
\end{figure}

\subsection{Proof of Theorem~\ref{theo:mod_poi}}\label{subsec:proof_mod_poi}
Recall the definition of $P_{n,k}(t)$ given in~\eqref{eq:P_n_k_def}. We need to prove that
$$
\lim_{n\to\infty} \frac{P_{n,k}(\eee^z)}{n^{k (\eee^z - 1)}} = \frac{\Gamma(k)}{\Gamma(k \eee^z)}.
$$
Since the number of summands in~\eqref{eq:gener_funct_formula} is fixed, we can consider the asymptotics of each summand separately. For every $m\in \{1,\ldots,k\}$, the $m$-th summand with $t= \eee^z$  satisfies
$$
(-1)^{k-m} \binom k m  \frac{\Gamma(m\eee^z + n)}{\Gamma(m\eee^z) n!}
=
(-1)^{k-m} \binom k m \frac{n^{m \eee^z  - 1}}{\Gamma(m \eee^z)} (1+ O(1/n))
$$
by the formula $\Gamma(n + \alpha)/\Gamma(n+\beta) = n^{\alpha - \beta}(1+ O(1/n))$ as $n\to\infty$, which holds locally uniformly in $\alpha,\beta\in\C$, see, for example, Theorem in \cite{Fields:1970}. If $z$ stays in a compact subset $K$ of $\mathcal{D}_{Lah} = \{t\in \C: \cos \Im t >0\}$, then $\Re \eee^z > \eps(K)>0$ stays bounded away from $0$ for some sufficiently small $\eps(K)\in (0,1)$. It follows that the summand with $m=k$ dominates in the following sense:
$$
P_{n,k}(\eee^z)
=
\frac {1}{\binom {n-1}{k-1}} \left(\frac{n^{ k \eee^z  - 1}}{\Gamma(\eee^zk)}
(1+O(1/n))
+
\sum_{m=1}^{k-1} (-1)^{k-m} \binom k m  \frac{n^{m \eee^z  - 1}}{\Gamma(m \eee^z)} (1+ O(1/n))
\right)
=
\frac{\Gamma(k)}{\Gamma(k \eee^z)} n^{k (\eee^z - 1)} (1 + O(n^{-\eps(K)})),
$$
which proves the claim.  Observe that in the case when $\cos \Im z \leq 0$ this argument does not apply.
\hfill $\Box$

\begin{remark}\label{rem:A4}
Assumption A4 of~\cite{kabluchko_marynych_sulzbach_edgeworth} can be verified in a similar way by observing that given $a\in (0,\pi)$ and a compact set $K\subset \R$, for all $\beta\in K$ it holds that
$$
n^{-k(\eee^\beta-1)}\int_a^\pi |P_{n,k}(\eee^{\beta + \ii u})| \dd u
\leq
C n^{-k(\eee^\beta-1)} \sum_{m=1}^k \int_a^\pi \frac{\left|\Gamma(\eee^{\beta + \ii u} m + n)\right|}{n!n^{k-1}} \dd u
\leq
C \sum_{m=1}^k \int_a^\pi |n^{\eee^{\beta + \ii u} m - \eee^{\beta}k}|\dd u
\leq
C n^{- \delta}
$$
for some constants $C= C(K,k)>0$ and $\delta= \delta(K,k,a)>0$.
\end{remark}

\subsection{Strong law of large numbers}
%Recall from Proposition~\ref{prop:representation} that the number $X_{n,k}$ of records with respect to a uniform random composition of $n$ into $k$ summands has the $\Lah(n,k)$ distribution.
For the P\'olya urn coupling $(Y_{n,k})_{n=k,k+1,\ldots}$ constructed in Section~\ref{subsec:polya}, where $k$ is fixed, the following strong law of large numbers holds.

\begin{proposition}\label{prop:slln}
For every fixed $k\in\N$ we have
$$
\frac{Y_{n,k}}{\log n}\toas k.
$$
\end{proposition}
\begin{proof}
Differentiating~\eqref{theo:mod_poi} and plugging $z=0$, which is legitimate since the convergence is uniform in a neighborhood of the origin, we obtain
$$
\lim_{n\to\infty}\frac{\E \Lah(n,k)}{k\log n}=\lim_{n\to\infty}\frac{\Var \Lah(n,k)}{k\log n}=1.
$$
Thus, by Chebyshev's inequality, for every $\eps>0$,
$$
\P\left[\left|\frac{Y_{n,k}}{\E Y_{n,k}}-1\right|>\eps\right]\leq \frac{\Var \Lah(n,k)}{\eps^2 (\E \Lah(n,k))^2}\sim\frac{\eps^{-2}}{k\log n},\quad n\to\infty.
$$
By the Borel-Cantelli lemma,
$$
\frac{Y_{\lfloor \eee^{n^2}\rfloor,k}}{\E Y_{\lfloor \eee^{n^2}\rfloor,k}}\toas 1.
$$
The result now follows from the standard sandwich argument using monotonicity of $(Y_{n,k})_{n=k,k+1,\ldots}$. Indeed, for every $m\geq 3$ there exists $n\in\N$ such that $\lfloor \eee^{n^2}\rfloor  \leq m<\lfloor \eee^{(n+1)^2}\rfloor $. Therefore,
$$
\frac{Y_{\lfloor \eee^{n^2}\rfloor,k}}{\E Y_{\lfloor \eee^{n^2}\rfloor,k}}\frac{\E Y_{\lfloor \eee^{n^2}\rfloor,k}}{\E Y_{\lfloor \eee^{(n+1)^2}\rfloor,k}} \leq \frac{Y_{m,k}}{\E Y_{m,k}}\leq \frac{Y_{\lfloor \eee^{(n+1)^2}\rfloor,k}}{\E Y_{\lfloor \eee^{(n+1)^2}\rfloor,k}}\frac{\E Y_{\lfloor \eee^{(n+1)^2}\rfloor,k}}{\E Y_{\lfloor \eee^{n^2}\rfloor,k}}.
$$
Sending $n\to\infty$ completes the proof.
\end{proof}

%\subsection{Strong law of large numbers}
%It is possible to prove strong law of large numbers for $k=const$, but it seems that it is true only in the monotone coupling of the %Lah distributions.

\section{Limit theorems for the Lah distribution: regimes of growing \texorpdfstring{$k$}{k}}\label{sec:growing_k}
Throughout this section we assume that $n\to\infty$ and $k=k(n)\to \infty$. There are two main regimes: the \textit{central regime} in which
\begin{equation}\label{eq:k_sim_alpha_n}
\lim_{n\to\infty} \frac{k(n)} n = \alpha \qquad \text{ for some constant } \alpha \in (0,1),
\end{equation}
and the \textit{intermediate regime}, in which $k(n) = o(n)$. We begin with the central regime.

\subsection{Central limit theorem in the central regime}
\begin{theorem}[CLT in the central regime]\label{thm:clt_central_regime}
Assume~\eqref{eq:k_sim_alpha_n}. Then,
$$
\frac{\Lah(n,k) - \E \Lah(n,k)}{\sqrt n} \todistr {\rm N}\left(0,-\left(\frac{\alpha}{1-\alpha} + \frac{\alpha(\alpha+1)\log \alpha}{(1-\alpha)^2} + \frac{\alpha^2 \log^2 \alpha}{(1-\alpha)^3}\right)\right),
$$
where ${\rm N}(m,\sigma^2)$ denotes a normal random variable with mean $m\in\R$ and variance $\sigma^2>0$.
\end{theorem}

\begin{figure}[t]
\begin{center}
\includegraphics[width=0.7\textwidth ]{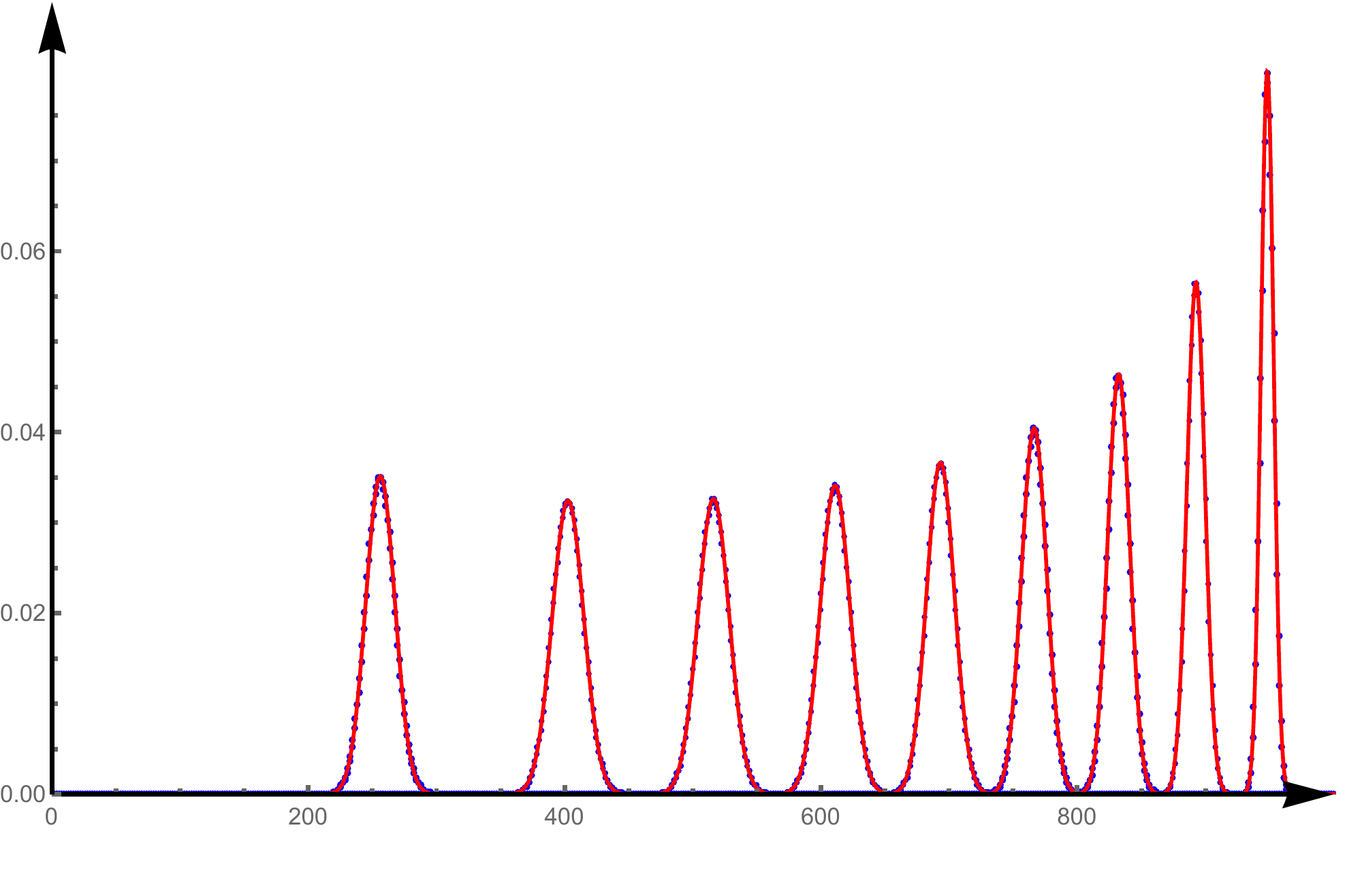}
%Lah Distributions 1.nb
\end{center}
\caption
{
The probability mass functions of  Lah distributions with $n=1000$ and $k=\alpha n$  with $\alpha \in \{\frac 1 {10}, \frac 2{10}, \ldots, \frac{9}{10}\}$ (blue dots) together with the approximating normal densities (red curves).
}
\label{fig:clt_linear_k}
\end{figure}
\begin{remark}
Despite of the minus sign in front of the formula for the variance, the latter is positive. For $\alpha\to 0$ and $\alpha\to 1$ the variance vanishes; the maximal value is attained at $\alpha =0.23517\ldots$. An illustration of Theorem~\ref{thm:clt_central_regime} is shown in Figure~\ref{fig:clt_linear_k}.
%Lah Distribution Check the Variance.nb
\end{remark}

The proof of Theorem~\ref{thm:clt_central_regime} relies on the representation~\eqref{eq:representation_as_a_sum_over_blocks_of_composition} and a multivariate central limit theorem for the number of blocks of fixed sizes in a uniform random composition $(b_1^{(n)},\ldots,b_k^{(n)})$ of $n$.
For $j\in\N$, let us denote by $N_j^{(n)}$ the number of blocks of size $j$ in the composition $(b_1^{(n)},\ldots,b_k^{(n)})$. Thus,
$$
N_j^{(n)}:=\sum_{i=1}^{k}\1_{\{b_i^{(n)}=j\}}.
$$
Note that, by~\eqref{eq:b_1_n_distribution}, $\E N_j^{(n)}=k\P[b_1^{(n)}=j]=k\binom{n-1}{k-1}^{-1}\binom{n-j-1}{k-2}\sim k\alpha(1-\alpha)^{j-1}$, where the asymptotic equivalence holds whenever~\eqref{eq:k_sim_alpha_n} is in force. In particular, this implies that under assumption~\eqref{eq:k_sim_alpha_n}, the random variables $b_1^{(n)}$ (and, thus $b_i^{(n)}$ for every fixed $i=1,\ldots,k$) converge in distribution, as $n\to\infty$, to a geometric law on $\N$ with success probability $\alpha$, see~\cite[Section~4]{diaconis_freedman} for much stronger results.

\begin{theorem}[Central limit theorem for $(N_1^{(n)},N_2^{(n)},\ldots,)$]\label{thm:clt_compositions}
Assume~\eqref{eq:k_sim_alpha_n}. Then, as $n\to\infty$,
$$
\left(\frac{N_j^{(n)}-k\P[b_1^{(n)}=j]}{\sqrt{k}}\right)_{j\geq 1}\todistr \left(\mathcal{N}_j\right)_{j\geq 1},
$$
in $\R^{\infty}$ endowed with the product topology, where $(\mathcal{N}_j)_{j\geq 1}$ is a centred Gaussian vector with the covariance
$$
\Cov(\mathcal{N}_i,\mathcal{N}_j)=\E \mathcal{N}_i\mathcal{N}_j=p_i\1_{\{i=j\}}-p_i p_j-\frac{p_i p_j}{1-\alpha}(\alpha i -1)(\alpha j-1),\quad i,j\in\N,
$$
and  $p_j:=\alpha(1-\alpha)^{j-1}$, for $j\in\N$.
\end{theorem}

\begin{remark}
Let us mention an interpretation of the random vector $(\mathcal{N}_j)_{j\geq 1}$ as a conditional distribution. If $\mathcal{G}_1,\mathcal{G}_2,\ldots$ are independent centered Gaussian variables with $\Var \mathcal{G}_j = p_j$, then $(\mathcal{N}_j)_{j\geq 1}$ has the same distribution as $(\mathcal{G}_j)_{j\geq 1}$ conditioned on the event $\{\sum_{j=1}^\infty \mathcal{G}_j =0, \sum_{j=1}^\infty j \mathcal{G}_j = 0\}$. This can be easily verified using the formulas for the covariance matrix of the conditional Gaussian distribution.
\end{remark}

Theorem~\ref{thm:clt_compositions} is known and has been rediscovered several times. Its proofs are based on a representation of the distribution of $(b_1^{(n)},\ldots,b_k^{(n)})$ as the law of $k$ independent geometrically distributed random variables conditioned on their sum to be $n$. More precisely, we have
\begin{equation}\label{eq:conitional_law}
\P[(b_1^{(n)},\ldots,b_k^{(n)})\in\cdot]=\P[(G_1,\ldots,G_k)\in\cdot|G_1+\cdots+G_k=n],
\end{equation}
where $G_1,\ldots,G_k$ are independent random variables having the same geometric law on $\N$ with parameter $\theta$. Note that this representation holds for arbitrary $\theta\in(0,1)$ and we are free to choose it as we wish. It is convenient to put  $\theta:=\theta_n=k/n$, so that the mean of $G_1+\ldots+G_k$ is $n$. This identifies the random composition as a special case of the generalized allocation scheme introduced by V.\ F.\ Kolchin in~\cite{kolcin_gen_allocation} and much studied thereafter; see, e.g., \cite{kolcin_branching} and~\cite[Chapter~VIII]{kolcin_book_allocations}. In particular, the convergence of one-dimensional distributions in Theorem~\ref{thm:clt_compositions} is contained in Theorem~1 of~\cite{kolcin_gen_allocation}. The full statement of  Theorem~\ref{thm:clt_compositions} is a special case of the general results of Holst, see~\cite[Theorem~2]{holst1979}  or~\cite[Theorem~2]{holst_urn_models}, but it requires some effort to see this. The proofs by Holst rely on the paper by Le Cam~\cite{lecam1958} who studied a related question for exponential random variables. A special case of Holst's results, from which Theorem~\ref{thm:clt_compositions} follows directly, can be found in the paper by Ivchenko~\cite[Theorem~4]{ivchenko}.  The corresponding multidimensional local limit theorem was derived by Trunov~\cite[Theorem~2.1]{trunov} who did not rely on~\cite{lecam1958}. See also~\cite[Theorem~3]{Vershik+Yakubovich:2003} for a weak law of large numbers, and~\cite[Section~4.4]{sachkov_book_probabilistic} for a similar result about partitions instead of compositions. To keep the paper self-contained we shall give a sketch of the proof of Theorem~\ref{thm:clt_compositions} in the Appendix.

The papers~\cite{holst1979,holst_urn_models,ivchenko,trunov} use the sequence $k\theta_n(1-\theta_n)^{j-1}=k\P[G_1=j]$ to center $N_j^{(n)}$. Let us check that it can be replaced by the sequence $\E N_j^{(n)} = k\P[b_1^{(n)}=j]$ used in Theorem~\ref{thm:clt_compositions}.

\begin{lemma}\label{eq:centerings_equivalent}
Assume~\eqref{eq:k_sim_alpha_n} and put $\theta_n:=k/n$. Then, for every fixed $j\in \N$, we have
$$
\left|k\P[b_1^{(n)}=j]-k\theta_n(1-\theta_n)^{j-1}\right|=O(1),\quad n\to\infty.
$$
\end{lemma}
\begin{proof}
By~\eqref{eq:b_1_n_distribution}, we have
\begin{align*}
\P[b_1^{(n)}=j]&=\frac{\binom{n-j-1}{k-2}}{\binom{n-1}{k-1}}=\frac{k-1}{n-1}\left(1-\frac{k-2}{n-j}\right)\left(1-\frac{k-2}{n-j+1}\right)\cdots\left(1-\frac{k-2}{n-2}\right)\\
&=\left(\frac{k}{n}+O\left(\frac{1}{n}\right)\right)\left(1-\frac{k}{n}+O\left(\frac{1}{n}\right)\right)\left(1-\frac{k}{n}+O\left(\frac{1}{n}\right)\right)\cdots\left(1-\frac{k}{n}+O\left(\frac{1}{n}\right)\right)\\
&=\frac{k}{n}\left(1-\frac{k}{n}\right)^{j-1}+O\left(\frac{1}{n}\right)=\theta_n\left(1-\theta_n\right)^{j-1}+O\left(\frac{1}{n}\right),
\end{align*}
%where the constants in the Landau symbols may depend on $j$.
which implies the claim after multiplication by $k$.
\end{proof}

\begin{proof}[Proof of Theorem~\ref{thm:clt_central_regime} using Theorem~\ref{thm:clt_compositions}]
Recall that $Z^{(j)}_{n}\overset{{\rm d}}{=}\Lah(n,1)$, is an array of mutually independent random variables such that $Z^{(j)}_{n}$ has distribution~\eqref{eq:records}, for $n,j\in\N$. Put
$$
\widehat{Z}^{(j)}_{n}:=\sum_{i=1}^{n}Z_j^{(i)},\quad n,j\in\N.
$$
Representation~\eqref{eq:representation_as_a_sum_over_blocks_of_composition} is equivalent to the following one:
$$
\Lah(n,k)\overset{d}{=}\sum_{j=1}^{n}\widehat{Z}^{(j)}_{N_j^{(n)}}.
$$
By Donsker's theorem and using that $\E Z^{(i)}_{j}=H_j$ and $\Var (Z^{(i)}_{j})=H_j-H_j^{(2)}$, for every $j\in\N$,
\begin{equation}\label{eq:donsker1}
\left(\frac{\widehat{Z}^{(j)}_{\lfloor nt\rfloor}-H_j nt}{\sqrt{n}}\right)_{t\geq 0}\todistrD \left(\sqrt{H_j-H_j^{(2)}}B_j(t)\right)_{t\geq 0},
\end{equation}
in the Skorokhod space $D[0,\,\infty)$ endowed with the standard $J_1$-topology, where $B_1,B_2,\ldots$ are independent standard Brownian motions. Moreover, the convergences in~\eqref{eq:donsker1} hold also mutually for all $j\in\N$ by independence.  Combining this with Theorem~\ref{thm:clt_compositions} we obtain that, for every $M\in\N$,
\begin{multline}\label{eq:joint_clt}
\left(\left(\frac{\widehat{Z}^{(j)}_{\lfloor nt\rfloor}-H_j nt}{\sqrt{n}}\right)_{t\geq 0},\frac{N_j^{(n)}}{n},\frac{N_j^{(n)}-k\P[b_1^{(n)}=j]}{\sqrt{n}}\right)_{j=1,\ldots,M}\\
\todistrD \left(\left(\sqrt{H_j-H_j^{(2)}}B_j(t)\right)_{t\geq 0},\alpha^2(1-\alpha)^{j-1},\sqrt{\alpha}\mathcal{N}_j\right)_{j=1,\ldots,M},
\end{multline}
in the product topology on $(D[0,\,\infty)\times\R\times \R)^M$, where we have also used that $k\sim \alpha n$ and $\lim_{n\to\infty}\P[b_1^{(n)}=j]=\alpha(1-\alpha)^{j-1}$.

Applying the mapping $D[0,\,\infty)\times\R\times\R\ni (f(\cdot),x,y)\mapsto f(x)+H_j y\in\R$ which is a.s.~continuous at the point given by the right-hand side of~\eqref{eq:joint_clt}, yields
$$
\left(\frac{\widehat{Z}_{N_j^{(n)}}^{(j)}-kH_j \P[b_1^{(n)}=j]}{\sqrt{n}}\right)_{j=1,\ldots,M}\todistr\left(\sqrt{H_j-H_j^{(2)}}B_j(\alpha^2(1-\alpha)^{j-1})+H_j\sqrt{\alpha}\mathcal{N}_j\right)_{j=1,\ldots,M}.
$$
Summation over $j=1,\ldots,M$ gives
\begin{equation}\label{eq:donsker2}
\frac{\sum_{j=1}^{M}\widehat{Z}_{N_j^{(n)}}^{(j)}-k\sum_{j=1}^{M}H_j \P[b_1^{(n)}=j]}{\sqrt{n}}\todistr\sum_{j=1}^{M}\left(\sqrt{H_j-H_j^{(2)}}B_j(\alpha^2(1-\alpha)^{j-1})+H_j\sqrt{\alpha}\mathcal{N}_j\right),
\end{equation}
and this relation holds for every fixed $M\in\N$. Note that as $M\to\infty$, the right-hand side of~\eqref{eq:donsker2} converges to
$$
\sum_{j=1}^{\infty}\left(\sqrt{H_j-H_j^{(2)}}B_j(\alpha^2(1-\alpha)^{j-1})+H_j\sqrt{\alpha}\mathcal{N}_j\right)
$$
and this series converges almost surely because
$$
\sum_{j=1}^{\infty}\left(\sqrt{H_j-H_j^{(2)}}\E|B_j(\alpha^2(1-\alpha)^{j-1})|+H_j\sqrt{\alpha}\E|\mathcal{N}_j|\right)<\infty.
$$
According to Theorem 3.2 in \cite{billingsley} it remains to prove that for every fixed $\varepsilon>0$,
\begin{equation}\label{eq:donsker3}
\lim_{M\to\infty}\limsup_{n\to\infty}
\P\left[\left|\sum_{j=M+1}^{n}\widehat{Z}_{N_j^{(n)}}^{(j)}-k\sum_{j=M+1}^{n}H_j \P[b_1^{(n)}=j]\right|>\varepsilon\sqrt{n}\right]=0.
\end{equation}
By Markov's inequality, it suffices to check that
\begin{equation}\label{eq:donsker4}
\lim_{M\to\infty}\limsup_{n\to\infty}\frac{\sum_{j=M+1}^{n}\E\left|\widehat{Z}_{N_j^{(n)}}^{(j)}-kH_j \P[b_1^{(n)}=j]\right|}{\sqrt{n}}=0.
\end{equation}
In order to prove~\eqref{eq:donsker4} we argue as follows. Using Wald's identity followed by the formula for the conditional variance, we derive, for all $j\in\N$,
\begin{align*}
\E\left|\widehat{Z}_{N_j^{(n)}}^{(j)}-kH_j \P[b_1^{(n)}=j]\right|&=\E\left|\widehat{Z}_{N_j^{(n)}}^{(j)}-\E\widehat{Z}_{N_j^{(n)}}^{(j)}\right|\leq \left((\Var (\widehat{Z}_{N_j^{(n)}}^{(j)})\right)^{1/2}\\
&= \left(\E\Var(\widehat{Z}_{N_j^{(n)}}^{(j)}|N_j^{(n)}))+\Var(\E(\widehat{Z}_{N_j^{(n)}}^{(j)}|N_j^{(n)}))\right)^{1/2}\leq \left(\E H_{N_j^{(n)}}+\Var (H_{N_j^{(n)}})\right)^{1/2}\\
&\leq \left(\E H^2_{N_j^{(n)}}\right)^{1/2}\leq {\rm const}\cdot (\E N_j^{(n)} )^{1/2}\leq {\rm const}\cdot\sqrt{n}\sqrt{\P[b_1^{(n)}=j]},
\end{align*}
where '${\rm const}$' denotes absolute constants whose values are of no importance. It remains to note that
\begin{multline*}
\lim_{M\to\infty}\limsup_{n\to\infty}\sum_{j=M+1}^{n}\sqrt{\P[b_1^{(n)}=j]}=\lim_{M\to\infty}\limsup_{n\to\infty}\sum_{j=M+1}^{n-k+1}\sqrt{\frac{\binom{n-j-1}{k-2}}{\binom{n-1}{k-1}}}\\
=\lim_{M\to\infty}\limsup_{n\to\infty}\sum_{j=M+1}^{n-k+1}\sqrt{(k-1)\frac{(n-k-j+2)\cdots(n-k)}{(n-j)\cdots(n-1)}}=0,
\end{multline*}
as readily follows from the inequality $\alpha n/2\leq k \leq n$ that holds for sufficiently large $n$.

Combining~\eqref{eq:donsker2} and~\eqref{eq:donsker3} we obtain
$$
\frac{\Lah(n,k)-\E\Lah(n,k)}{\sqrt{n}}\todistr \sum_{j=1}^{\infty}\left(\sqrt{H_j-H_j^{(2)}}B_j(\alpha^2(1-\alpha)^{j-1})+H_j\sqrt{\alpha}\mathcal{N}_j\right).
$$
The limiting variable is obviously normal, has zero mean and its variance can be derived as follows:
\begin{align*}
&\hspace{-1cm}\Var\left(\sum_{j=1}^{\infty}\left(\sqrt{H_j-H_j^{(2)}}B_j(\alpha^2(1-\alpha)^{j-1})+H_j\sqrt{\alpha}\mathcal{N}_j\right)\right)\\
&=\E\left(\sum_{j=1}^{\infty}\left(\sqrt{H_j-H_j^{(2)}}B_j(\alpha^2(1-\alpha)^{j-1})+H_j\sqrt{\alpha}\mathcal{N}_j\right)\right)^2\\
&=\sum_{j=1}^{\infty}(H_j-H_j^{(2)})\alpha^2(1-\alpha)^{j-1}+\E\left(\sum_{j=1}^{\infty}H_j\sqrt{\alpha}\mathcal{N}_j\right)^2\\
&=\sum_{j=1}^{\infty}(H_j-H_j^{(2)})\alpha^2(1-\alpha)^{j-1}+\alpha\sum_{i,j=1}^{\infty}H_i H_j\left(p_i\1_{\{i=j\}}-p_i p_j-\frac{p_i p_j}{1-\alpha}(\alpha i-1)(\alpha j-1)\right)\\
&=\alpha\sum_{j=1}^{\infty}(H_j-H_j^{(2)})p_j+\alpha\left(\sum_{j=1}^{\infty}H_j^2 p_j -\left(\sum_{j=1}^{\infty}H_j p_j\right)^2-\frac{1}{1-\alpha}\left(\sum_{j=1}^{\infty}H_j p_j(\alpha j-1)\right)^2\right),
\end{align*}
where $p_j=\alpha(1-\alpha)^{j-1}$ for $j\in\N$. We only show how to deal with the sums containing $H_j^{(2)}$ and $H_j^2$, the other being elementary. We have
\begin{align*}
\sum_{j=1}^{\infty}H_j^2 p_j-\sum_{j=1}^{\infty}H_j^{(2)} p_j&=\sum_{j,k,l=1}^{\infty}\frac{1}{kl}\alpha(1-\alpha)^{j-1}\1_{\{k,l\leq j\}}-\sum_{j,k=1}^{\infty}\frac{1}{k^2}\alpha(1-\alpha)^{j-1}\1_{\{k\leq j\}}\\
&=\sum_{k,l=1}^{\infty}\frac{1}{kl}(1-\alpha)^{\max(k,l)-1}-\sum_{k=1}^{\infty}\frac{1}{k^2}(1-\alpha)^{k-1}\\
&=\sum_{k=1}^{\infty}\frac{1}{k}\left(\sum_{l=1}^{k-1}\frac{(1-\alpha)^{k-1}}{l}+\sum_{l=k}^{\infty}\frac{(1-\alpha)^{l-1}}{l}\right)-\sum_{k=1}^{\infty}\frac{1}{k^2}(1-\alpha)^{k-1}\\
&=\sum_{k=1}^{\infty}\frac{1}{k}H_{k-1}(1-\alpha)^{k-1}+\sum_{k=1}^{\infty}\frac{1}{k}H_{k}(1-\alpha)^{k-1}-\sum_{k=1}^{\infty}\frac{1}{k^2}(1-\alpha)^{k-1}\\
&=2\sum_{k=1}^{\infty}\frac{H_{k-1}}{k}(1-\alpha)^{k-1}=\frac{2}{1-\alpha}\int_{\alpha}^{1}\left(\sum_{k=1}^{\infty}H_{k-1}(1-t)^{k-1}\right){\rm d}t\\
&=-\frac{2}{1-\alpha}\int_{\alpha}^{1}\frac{\log t}{t}{\rm d}t=\frac{\log^2\alpha}{1-\alpha}.
\end{align*}
The proof is complete.
\end{proof}

\subsection{Cumulant generating function in the central regime}
In this and the following section we establish a large deviations principle under assumption~\eqref{eq:k_sim_alpha_n}. First we look at the cumulant generating function.
\begin{proposition}\label{prop:laplace_trafo_conv}
Assume that~\eqref{eq:k_sim_alpha_n} holds. Then,  for every fixed $t\in \R$, the limit
\begin{equation}\label{eq:gaertner_ellis}
\varphi_\alpha(t) := \lim_{n\to\infty} \frac 1n \log \E \eee^{t \Lah(n,k)},
\end{equation}
exists finitely, and, moreover, the function $\varphi_\alpha: \R\to\R$ is given by
\begin{equation}\label{eq:varphi_alpha}
\varphi_\alpha(t) = 2 \alpha \log \alpha + (1-\alpha) \log (1-\alpha) + (\alpha-1) \log r(t) + \alpha t - \alpha (\eee^t + 1) \log (1- r(t)),
\end{equation}
where $r(t) = r(t; \alpha)\in (0,1)$ is a unique non-zero solution to the equation
\begin{equation}\label{eq:implicit_r_t}
(1-r(t))^{\eee^t + 1} - 1 +  r(t) (1+\alpha \eee^t) = 0.
\end{equation}
\end{proposition}
\begin{remark}\label{rem:sol_unique}
Let us argue that Equation~\eqref{eq:implicit_r_t} has a unique solution $r(t)\in(0,1)$ for every $t\in\R$. For every fixed $\alpha\in (0,1)$ and $t\in \R$, the function $h(r) = h_{\alpha,t}(r) := (1-r)^{\eee^t + 1} - 1 + r(1 + \alpha \eee^t)$, defined for $r\in [0,1]$, is strictly convex since its first derivative
$$
h'(r) = -(\eee^t + 1) (1-r)^{\eee^t} + (1+\alpha \eee^t), \quad r\in [0,1],
$$
strictly increases. Moreover, we have $h(0)=0$ and $h'(0) <0$. To complete the proof of the existence and the uniqueness of zero of $h$ in the interval $(0,1)$,  observe that $h(1)>0$.  Moreover, it follows from the above that $h'(r(t))>0$. By the implicit function theorem we conclude that $r(t)$ is differentiable on $(0,1)$. Clearly, $r(0;\alpha) = 1-\alpha$.
\end{remark}

\begin{proof}[Proof of Proposition~\ref{prop:laplace_trafo_conv}]
Fix any $t\in\R$. By Lemma~\ref{lem:gener_funct_formula} we have
$$
\E \eee^{t \Lah(n,k)}
=
P_{n,k}(\eee^t)
=
\frac {1}{\binom {n-1}{k-1}} [x^n] \left((1-x)^{-\eee^t} - 1\right)^k.
$$
We shall now use the classical saddle-point method to derive the asymptotics of $[x^n] ((1-x)^{-\eee^t} - 1)^k$. With this respect Theorem VIII.8 in \cite{Flajolet_book} perfectly fits our needs. Note that
$$
[x^n] \left((1-x)^{-\eee^t} - 1\right)^k=[x^{n-k}]\left(\frac{(1-x)^{-e^t}-1}{x}\right)^k=:[x^{n-k}](B(x))^k,
$$
and the function $x\mapsto B(x)$, for every fixed $t\in\R$, is analytic in the interior of the unit disk, has non-negative coefficients and $B(0)\neq 0$. Thus, applying Theorem VIII.8 in \cite{Flajolet_book} with $A(x):=1$, $B(x)=x^{-1}((1-x)^{-e^t}-1)$,  replacing $(N,n)$ there by $(n-k,k)$ and taking $\lambda=\lambda_n=n/k-1$, we get\footnote{Note that the claim of Theorem VIII.8 holds locally uniformly in $\lambda$.}
\begin{equation}\label{eq:saddle-point-asymp}
[x^n] \left((1-x)^{-\eee^t} - 1\right)^k=(B(\zeta_n(t)))^k (\zeta_n((t)))^{-n+k-1}(2\pi k \widehat{\zeta}_n(t))^{-1/2}(1+o(1)),\quad n\to\infty,
\end{equation}
where $\zeta_n(t)$ is the unique root of
\begin{equation}\label{eq:saddle-point-equation}
\frac{\zeta_n(t)B'(\zeta_n(t))}{B(\zeta_n(t))}=\lambda_n,
\end{equation}
and
$$
\widehat{\zeta}_n(t):=\frac{{\rm d}}{{\rm d}z^2}\left(\log B(z)-\lambda_n\log z\right)\Big|_{z=\zeta_n(t)}.
$$
Substituting $B(x)=x^{-1}((1-x)^{-e^t}-1)$ into~\eqref{eq:saddle-point-equation} and simplifying, we obtain
$$
\frac{e^t\zeta_n(t)}{1-\zeta_n(t)-(1-\zeta_n(t))^{e^t+1}}=\frac{n}{k}.
$$
Sending $n\to\infty$ we see from the discussion in Remark~\ref{rem:sol_unique} that
$$
\lim_{n\to\infty}\zeta_n(t)=r(t),
$$
where $r(t)$ is given by~\eqref{eq:implicit_r_t} and also that $\widehat{\zeta_n}(t)$ converges to a finite non-zero constant. Thus, taking logarithms in~\eqref{eq:saddle-point-asymp}, dividing by $n$ and sending $n\to\infty$, yields
$$
\lim_{n\to\infty}n^{-1}\log\left([x^n] \left((1-x)^{-\eee^t} - 1\right)^k\right)=\alpha\log B(r(t))-(1-\alpha)\log r(t).
$$
Combining this with
\begin{equation}\label{eq:asympt_binomial_coeff}
\lim_{n\to\infty} \frac 1n \log \binom{n-1}{k-1} = -\alpha \log \alpha - (1-\alpha) \log (1-\alpha),
\end{equation}
which follows from the Stirling formula, we obtain~\eqref{eq:varphi_alpha} after elementary manipulations.
\end{proof}

\begin{remark}[On linear growth of cumulants]
Recall that the $\ell$-th cumulant of a random variable $X$ is defined as
$$
\kappa_\ell(X) := D^\ell_t \Big|_{t=0} \left(\log \E \eee^{t X}\right),
$$
and we assume that $X$ has finite exponential moments in a neighborhood of $0$. Let, as before, $k=k(n) \sim \alpha n$ for some $\alpha \in (0,1)$. Refining the arguments used in the proof of Proposition~\ref{prop:laplace_trafo_conv}, it is possible to show that the cumulants of  $\Lah(n,k)$ grow linearly in the sense that
\begin{equation}\label{eq:cumulants_linearly}
\lim_{n\to\infty} \frac{\kappa_\ell(\Lah(n,k))}{n} = \varphi_\alpha^{(\ell)}(0)
\qquad
\text{ for all } \ell\in\N.
\end{equation}
Indeed, this relation can be obtained from~\eqref{eq:gaertner_ellis} by differentiating it $\ell\in \N$ times. To justify that the limit and the derivative can be interchanged, it suffices to check that the assertion of Proposition~\ref{prop:laplace_trafo_conv} continues to hold locally uniformly for complex $t$ in a small neighborhood of $\R$. Then, since all involved functions are analytic, we can interchange $D^\ell_t$ and the large $n$ limit. The validity of Proposition~\ref{prop:laplace_trafo_conv} in a small complex neighborhood of $\R$ follows essentially from  the stability of the saddle point under small analytic perturbations of the parameter $t$. Making this argument rigorous (and in particular, checking the local uniformity of convergence) is standard but tedious, and we omit the details.

Taking $\ell = 1$ and $\ell = 2$ in~\eqref{eq:cumulants_linearly} yields the formulae
$\E \Lah(n,k) \sim \varphi'_\alpha(0) n$
and
$
\Var \Lah(n,k) \sim \varphi''_\alpha(0) n
$,
and the expression for $\varphi_\alpha(t)$ given in~\eqref{eq:varphi_alpha} yields, after lengthy calculations, that
\begin{align*}
\varphi'_\alpha(0) = -\frac{\alpha \log \alpha}{1-\alpha},
\qquad
\varphi''_\alpha(0) = -\frac{\alpha}{1-\alpha} - \frac{\alpha(\alpha+1)\log \alpha}{(1-\alpha)^2} - \frac{\alpha^2 \log^2 \alpha}{(1-\alpha)^3}.
\end{align*}
Both formulas agree with previously derived results: formula~\eqref{eq:lah_distr_expect_asympt} for the expectation, and the variance of the limiting normal law in Theorem~\ref{thm:clt_central_regime}. Note that the linear growth of cumulants~\eqref{eq:cumulants_linearly} implies the CLT by the method of moments giving an alternative way to prove Theorem~\ref{thm:clt_central_regime}.
\end{remark}

\subsection{Large deviations in the central regime}
%%%%Numerical computations related to the large deviations are in the following file: Lah distribution check the variance.nb
Recall that we assume~\eqref{eq:k_sim_alpha_n}. In the next theorem we shall state a large deviations principle (see~\cite{dembo_zeitouni_book}) for the Lah distribution which, in particular, implies that, as $n\to\infty$,
\begin{align}
\P[\Lah(n,k) \leq  (\beta+o(1)) n]  &= \exp\{ - n I_\alpha(\beta) + o(n) \}, \qquad \text{ if } \beta\in \left(\alpha, -\frac{\alpha \log \alpha}{1-\alpha}\right),\label{eq:LDP_Lah_leq}\\
\P[\Lah(n,k) \geq  (\beta+o(1)) n]  &= \exp\{ - n I_\alpha(\beta) + o(n) \}, \qquad \text{ if } \beta\in \left(-\frac{\alpha \log \alpha}{1-\alpha},1\right)
\label{eq:LDP_Lah_geq}
\end{align}
for a rate function $I_{\alpha}$ which we shall explicitly identify.
%Let $I_\alpha:[0,1]\mapsto [0,\infty]$ be the function defined by
%\begin{equation}\label{eq:rate_function_def}
%I_{\alpha}(\beta):=-\lim_{n\to\infty}n^{-1}\log \P[n^{-1}\Lah(n,k)>\beta],\quad \beta\in[0,1],
%\end{equation}
%provided that this limit exists. In the theory of large deviations such a function is called the rate function.

\begin{theorem}[LDP in the central regime]\label{theo:ldp_central}
Assume~\eqref{eq:k_sim_alpha_n}. Then, the sequence of random variables $\frac 1n \Lah(n,k)$, $n\in \N$, satisfies a large deviations principle with a convex rate function $I_{\alpha}:[0, 1] \to [0,\infty]$  defined as follows. For $\beta\in (\alpha,1)$ we have
\begin{align}
I_{\alpha}(\beta)
&=
\sup_{t\in \R} (\beta t - \varphi_\alpha(t))\label{eq:I_1}\\
&=
-\alpha \log \left(h^{-1}\left(\frac \beta \alpha\right) - 1\right) + \log \left(1- h^{-1}(\beta)\right) + \beta \log \left( - \frac{\log h^{-1}(\frac \beta\alpha)}{\log h^{-1}(\beta)}\right) - \alpha \log \alpha - (1-\alpha) \log (1-\alpha),
\label{eq:I_2}
\end{align}
where $h^{-1}(\cdot)$ is the inverse function of
\begin{equation}\label{eq:h_def}
h(x) := \varphi_x'(0) = -\frac {x \log x}{1-x}, \quad x\in(0,1)\cup (1,\infty),\quad h(0)=0,\quad h(1) :=1.
\end{equation}
For $\beta<\alpha$, we have  $I_\alpha(\beta) = +\infty$. Finally, at the boundary points $\beta=\alpha$ and $\beta=1$ we have $I_\alpha(\alpha) = \lim_{\beta\downarrow \alpha} I_{\alpha}(\beta)$ and $I_\alpha(1) = \lim_{\beta\uparrow 1} I_{\alpha}(\beta)$, namely
\begin{align}
I_\alpha(\alpha)
&=
\log \left(1- h^{-1}(\alpha)\right) -\alpha \log \left( - \log h^{-1}(\alpha)\right) - \alpha \log \alpha - (1-\alpha) \log (1-\alpha)
<+\infty
,\label{eq:I_s_alpha}\\
I_\alpha(1)
&=
-\alpha \log \left(h^{-1}\left(\frac 1 \alpha\right) - 1\right) + \log \log h^{-1}\left(\frac 1\alpha\right) - \alpha \log \alpha - (1-\alpha) \log (1-\alpha) < +\infty.\label{eq:I_s_1}
\end{align}
\end{theorem}
%%%%% Formula for I_alpha: Numerically correct in the sense that the I_alpha[h(alpha)]=0 and the first two derivatives are correct: Lah Distribution Check The Variance.nb
%%%%Formulas for the boundary values: Numerically correct: Lah Distribution Check The Variance.nb
\begin{remark}\label{rem:I_alpha_beta_properties}
Using the relation $h(1/x) = h(x)/x$ together with~\eqref{eq:I_2} one easily verifies that the function $I_\alpha(\beta)$ vanishes at  $\beta = h(\alpha) = -\frac{\alpha \log \alpha}{1-\alpha}$, which is not surprising in view of the fact that $\E \Lah(n,k) \sim h(\alpha) n$ by Theorem~\ref{theo:expect_asympt}. A lengthy computation shows that
$$
\frac{\dd}{\dd \beta}I_\alpha(\beta) = \log\left(- \frac{\log h^{-1}(\frac \beta\alpha)}{\log h^{-1}(\beta)}\right), \qquad \beta\in (\alpha,1).
$$
One then verifies that the derivative of $I_\alpha(\beta)$ vanishes at $\beta = h(\alpha)$, is positive for $\beta>h(\alpha)$ and negative for $\beta<h(\alpha)$. Hence, $I_\alpha(\beta)>0$ for all $\beta\neq h(\alpha)$.
\end{remark}
\begin{proof}[Proof of Theorem~\ref{theo:ldp_central}]
By Proposition~\ref{prop:laplace_trafo_conv}, $\varphi_\alpha(t):= \lim_{n\to\infty} \frac 1n \log \E \eee^{t \Lah(n,k)}$ exists finitely for every $t\in \R$. Moreover, it is a differentiable function of $t$; see Remark~\ref{rem:sol_unique}.   A large deviation principle with a rate function given by~\eqref{eq:I_1} is then implied by the G\"artner-Ellis theorem; see Theorem~2.3.6 and Exercise~2.3.20 of~\cite{dembo_zeitouni_book}. Note also that~\eqref{eq:I_1} implies that $I_\alpha$ is convex.

To prove~\eqref{eq:I_2}, one can use the asymptotics of the Stirling numbers $\stirling{n}{k}$ and $\stirlingsec{n}{k}$ in the central regime $k\sim \alpha n$, which is known from the works of Moser and Wyman~\cite{moser_wyman,moser_wyman_second}; see also~\cite[Sections~3.6,3.7]{sachkov_book_combinatorial} and~\cite[Section~VIII.8.2]{Flajolet_book}. A review of these and other asymptotic regimes of  $k$ can be found in~\cite{louchard_first_kind} and~\cite{louchard_second_kind}. For our purposes the most convenient reference is \cite{Timashov}. In particular, when $k \sim \alpha n$ for some $\alpha\in (0,1)$, we have by formulas (14) and (33) in \cite{Timashov}, respectively,
%One could try to derive Theorem~\ref{theo:ldp_central} (or even a local version of Theorem~\ref{thm:clt_compositions}) from these asymptotic results, but since they involve inverse functions, the calculations become extremely tedious and the proof strategy used above seems more %efficient and enlightening to us.
%Writing
%$$
%\frac{k!}{n!}\stirling{n}{k} = [x^n] \left(\log \frac 1 {1-x}\right)^k
%= \frac 1 {2\pi \ii} \oint_{|x|=r} \frac{\left(\log \frac 1 {1-x}\right)^k}{x^{n+1}} \dd x,
%\quad
%\frac{k!}{n!}\stirlingsec{n}{k} = [x^n] \left(\eee^x-1\right)^k
%=
%\frac 1 {2\pi \ii} \oint_{|x|=r} \frac{(\eee^{x}-1)^k}{x^{n+1}} \dd x,
%$$
%and using the saddle-point method as in~\cite[Theorem VIII.8]{Flajolet_book}, one obtains that in the regime when $k \sim \alpha n$ for some $%\alpha\in (0,1)$ we have
\begin{align}
\lim_{n\to\infty}
\frac 1n \log \left(\frac{k!}{n!}\stirling{n}{k} \right)
&=
-\log (1-h^{-1}(\alpha))  + \alpha \log (-\log h^{-1}(\alpha)), \label{eq:stirling_1_log_asympt}
\\
\lim_{n\to\infty}
\frac 1n \log \left(\frac{k!}{n!}\stirlingsec{n}{k} \right)
&=
\alpha \log (h^{-1}(1/\alpha) - 1) - \log \log h^{-1}(1/\alpha), \label{eq:stirling_2_log_asympt}
\end{align}
where $h^{-1}$ is the inverse of the function $h$ defined in~\eqref{eq:h_def}. Note that $h$ has strictly positive derivative
%because alpha-1-log alpha>0
and satisfies $\lim_{x\downarrow 0} h(x) = 0$ as well as $\lim_{x\uparrow +\infty} h(x) = +\infty$, which implies that $h^{-1}$ is well-defined and strictly increasing. Observe also that $h(1)=1$ (by continuity), hence $h^{-1}(1)=1$ and, consequently, $h^{-1}(\alpha) < 1 < h^{-1}(1/\alpha)$.
Combining these relations with~\eqref{eq:asympt_binomial_coeff}, one obtains in the regime when $k\sim \alpha n$ and $j\sim \beta n$ with $0<\alpha <\beta <1$ that
\begin{multline}\label{eq:rate_LDP_explicit}
\lim_{n\to\infty}
%\frac 1n \log \left(\frac 1 {L(n,k)}\stirling{n}{j} \stirlingsec{j}{k}\right)
\frac 1n \log \P[\Lah(n,k) = j]
=
\lim_{n\to\infty}
\frac 1n \log \left(\frac 1 {\binom{n-1}{k-1}} \cdot \frac{j!}{n!}\stirling{n}{j} \cdot \frac{k!}{j!}\stirlingsec{j}{k}\right)
\\=
\alpha \log \left(h^{-1}\left(\frac \beta \alpha\right) - 1\right) - \log \left(1- h^{-1}(\beta)\right) - \beta \log \left( - \frac{\log h^{-1}(\frac \beta\alpha)}{\log h^{-1}(\beta)}\right) + \alpha \log \alpha + (1-\alpha) \log (1-\alpha).
\end{multline}
Given this, an LDP with a rate function given by~\eqref{eq:I_2} follows by standard arguments.
Namely, by~\cite[Theorem 4.1.11]{dembo_zeitouni_book} it suffices to check that for all $t\in [\alpha,1]$ we have
\begin{align}
&\inf_{\eps>0} \limsup_{n\to\infty} \frac 1 {n} \log \P \left[\frac 1n \Lah(n,k) \in [t - \eps,t +\eps]\right]
\leq
%\text{RHS}\;\eqref{eq:rate_LDP_explicit}, \label{eq:weak_LDP1}
-I_\alpha(t), \label{eq:weak_LDP1}
\\
&\inf_{\eps>0} \liminf_{n\to\infty} \frac 1 {n} \log \P\left[\frac 1n\Lah(n,k) \in [t-\eps,t+\eps]\right]
\geq
%\text{RHS}\;\eqref{eq:rate_LDP_explicit}. \label{eq:weak_LDP2}
-I_\alpha(t), \label{eq:weak_LDP2}
\end{align}
where $I_\alpha$ is defined by~\eqref{eq:I_2},~\eqref{eq:I_s_alpha} and~\eqref{eq:I_s_1}.
For $t \in (\alpha,1)$ both claims follow immediately from~\eqref{eq:rate_LDP_explicit} together with the union bound.
Let us treat the boundary case $t=\alpha$, the case $t=1$ being similar. One easily checks that the limit of the right-hand side of~\eqref{eq:rate_LDP_explicit} as $\beta\downarrow \alpha$ coincides with $-I_\alpha(\alpha)$ as defined in~\eqref{eq:I_s_alpha}. The unimodality together with the union bound and~\eqref{eq:rate_LDP_explicit} with $\beta= \alpha + \eps$ imply~\eqref{eq:weak_LDP1}. To prove~\eqref{eq:weak_LDP2} use~\eqref{eq:rate_LDP_explicit} with $\beta = \alpha +\eps/2$ and let $\eps\downarrow 0$.
\end{proof}

%Using the relation $h(1/x) = h(x)/x$ together with~\eqref{eq:I_2} one easily verifies that the function $I_\alpha(\beta)$ vanishes %at  $\beta = h(\alpha) = -\frac{\alpha \log \alpha}{1-\alpha}$, which is not surprising in view of the fact that $\E \Lah(n,k) \sim %h(\alpha) n$ by Theorem~\ref{theo:expect_asympt}. It follows from Theorem~\ref{theo:ldp_central} and the convexity of $I_\alpha$ %that
%\begin{align}
%\P[\Lah(n,k) \leq  (\beta+o(1)) n]  &= \exp\{ - n I_\alpha(\beta) + o(n) \}, \qquad \text{ if } \beta\in \left(\alpha, %-\frac{\alpha \log \alpha}{1-\alpha}\right),\label{eq:LDP_Lah_leq}\\
%\P[\Lah(n,k) \geq  (\beta+o(1)) n]  &= \exp\{ - n I_\alpha(\beta) + o(n) \}, \qquad \text{ if } \beta\in \left(-\frac{\alpha \log %\alpha}{1-\alpha},1\right).
%\label{eq:LDP_Lah_geq}
%\end{align}

\begin{remark}
The function $h^{-1}$ appearing above can be expressed as $h^{-1}(z) = z/v(z)$, for $z>0$, where $v(z)>0$ is the solution to the equation $v\eee^{-v} = z \eee^{-z}$ which is different from $v=z$ (and for $z=1$, where there is only one solution, we put $v(1)=1$).
Using the standard notation for the branches of the Lambert $W$-function, see Remark~\ref{rem:weak_threshold} below, we have $h^{-1}(z) = -z/W_0(-z\eee^{-z})$ for $z\geq 1$ and $h^{-1}(z) = -z/W_{-1}(-z\eee^{-z})$ for $0<z\leq 1$.
\end{remark}

\subsection{The intermediate regime}
In this section we prove a weak law of large numbers for the Lah distribution in the intermediate regime, that is, when $k=k(n)$ satisfies
\begin{equation}\label{eq:intermediate}
\lim_{n\to\infty}\frac{k(n)}{n}  = 0
\;\;
\text { but }
\;\;
\lim_{n\to\infty} k(n) = \infty.
\end{equation}
Recall from Theorem~\ref{theo:expect_asympt} that $\E \Lah (n,k) \sim k \log (n/k)$ in this regime.
\begin{theorem}[Weak LLN in the intermediate regime]\label{theo:weak_LLN_intermediate}
Under~\eqref{eq:intermediate} we have
\begin{equation}\label{eq:Var_intermediate}
%1
%\leq
%\liminf_{n\to\infty} \frac{\Var\Lah(n,k)}{\E \Lah (n,k)}
%\leq
%\limsup_{n\to\infty} \frac{\Var\Lah(n,k)}{\E \Lah (n,k)}
%<
%2.
\lim_{n\to\infty} \frac{\Var\Lah(n,k)}{\E \Lah (n,k)}  = 1.
% \qquad \text{ as } n\to\infty.
\end{equation}
Consequently, the following weak law of large numbers holds in $L^2$ and hence in probability:
$$
\frac{\Lah(n,k)}{k \log (n/k)} \ton 1.
$$
\end{theorem}
\begin{proof}
It suffices to prove~\eqref{eq:Var_intermediate} because the weak LLN follows from~\eqref{eq:Var_intermediate} by Chebyshev's inequality. Recalling the representation~\eqref{eq:representation_as_a_sum_over_blocks_of_composition}, conditioning on the random uniform composition $(b_1^{(n)},\ldots,b_{k}^{(n)})$ and using the formula of the total variance, we can write
\begin{align}
\Var \Lah (n,k)
&=
\Var \E \left[\sum_{j=1}^{k}Z^{(j)}_{b^{(n)}_j}\,\Big |\,(b_1^{(n)},\ldots,b_{k}^{(n)})\right] +
\E   \Var \left[\sum_{j=1}^{k}Z^{(j)}_{b^{(n)}_j}\,\Big |\,(b_1^{(n)},\ldots,b_{k}^{(n)})\right]\notag\\
&=
\Var \left[\sum_{j=1}^{k} H_{b^{(n)}_j}\right]
+
\E  \left[\sum_{j=1}^{k} \left(H_{b^{(n)}_j} - H_{b^{(n)}_j}^{(2)}\right) \right] \label{eq:tech:est}
\\
&\leq
\Var  \left[\sum_{j=1}^{k} H_{b^{(n)}_j}\right]
+
\E  \left[\sum_{j=1}^{k} H_{b^{(n)}_j} \right],\notag
\end{align}
where we recall that $H_N = \sum_{\ell=1}^N\frac 1 \ell$ is the $N$-th harmonic number and $H_N^{(2)} = \sum_{\ell=1}^N \frac{1}{\ell^2}$.   Again according to~\eqref{eq:representation_as_a_sum_over_blocks_of_composition}, the second term on the right-hand side is nothing else but $\E \Lah (n,k)$. Note in passing that~\eqref{eq:tech:est} together with the estimate $H_N^{(2)}< \pi^2/6$  yields
$$
\Var \Lah (n,k) \geq \E \Lah (n,k) - \pi^2 k /6 \sim k\log (n/k), \qquad n\to\infty,
$$
which proves the lower bound in~\eqref{eq:Var_intermediate}.
To complete the proof of~\eqref{eq:Var_intermediate}, it suffices to check that
\begin{equation}\label{eq:Var_b_1}
\Var H_{b^{(n)}_1} = O(1), \qquad  \text{ as } n\to\infty.
\end{equation}
and
\begin{equation}\label{eq:intermediate_cov_negative}
\Cov \left(H_{b^{(n)}_i}, H_{b^{(n)}_j}\right) \leq 0,
\qquad \text{ for all } 1\leq i <j \leq k.
\end{equation}
Indeed, then we have
$$
\Var \left[\sum_{j=1}^{k} H_{b^{(n)}_j}\right]
=
\sum_{j=1}^k \Var H_{b^{(n)}_j} + 2 \sum_{1\leq i <j \leq k} \Cov \left(H_{b^{(n)}_i}, H_{b^{(n)}_j}\right)
\leq
\sum_{j=1}^k \Var H_{b^{(n)}_j}
=O(k).
$$
Let us prove~\eqref{eq:Var_b_1} first. The main step is the following lemma.
\begin{lemma}\label{lem:b_1_n_distr}
Under assumption~\eqref{eq:intermediate}, the random variables $(k/n) b_{1}^{(n)}$ converge in distribution, as $n\to\infty$, to the standard exponential law $\text{Exp}(1)$.
\end{lemma}
\begin{proof}
Using formula~\eqref{eq:b_1_n_distribution} and the hockey-stick identity we obtain
$$
\P[b_1^{(n)} \geq j] = \frac{\sum_{m=j}^{n-k+1}\binom{n-m-1}{k-2}}{\binom{n-1}{k-1}}=\frac{\sum_{m=k-2}^{n-j-1} \binom{m}{k-2}}{\binom{n-1}{k-1}} = \frac{\binom {n-j}{k-1}}{\binom{n-1}{k-1}},
\qquad j=1,\ldots, n-k+1.
$$
For all $x>0$ it follows that
$$
\P[(k/n) b_1^{(n)}\geq x]
=
\P[b_{1}^{(n)} \geq nx/k]
=
\P[b_1^{(n)}\geq \lceil nx/k\rceil]
=
\frac{\binom {n-j_n}{k-1}}{\binom{n-1}{k-1}}
=
\frac{(n-j_n - k+2)\ldots (n-j_n)}{(n-k+1)\ldots (n-1)}
$$
with $j_n =\lceil nx/k\rceil \sim nx/k$. Taking the logarithm and using that $\log (1-x) = -x +O(x^2)$ as $x\to 0$, we obtain
\begin{align*}
\log \P[(k/n) b_1^{(n)}\geq x]
&=
\sum_{\ell=1}^{k-1} \log \left(1 - \frac{j_n-1}{n-k+\ell}\right)
=
-\sum_{\ell=1}^{k-1}  \left(\frac{j_n-1}{n-k+\ell} + O\left(\frac{j_n^2}{(n-k)^2}\right)\right)\\
&=
- (j_n-1) \sum_{i=n-k+1}^{n-1} \frac{1}{i} + O\left( \frac{kj_n^2}{(n-k)^2}\right)
=
x + o(1),
\end{align*}
and the proof is complete.
\end{proof}
We can now prove~\eqref{eq:Var_b_1} as follows. By the Skorokhod representation theorem, we may pass to a different probability space and, after taking the logarithm, write Lemma~\ref{lem:b_1_n_distr} in the form
\begin{equation}\label{eq:log_b_1_n}
\log b_{1}^{(n)}  - \log (n/k) \toas G,
\end{equation}
where $G:= \log \text{Exp}(1)$. This already suggests that $\Var \log b_{1}^{(n)}$ should be of order $O(1)$. We shall now justify this by a uniform integrability argument. We claim that for every $p\geq 1$,
\begin{equation}\label{eq:p_moment_bounded}
\sup_{n\in \N} \E \left|\log b_{1}^{(n)}  - \log (n/k)\right|^p <\infty.
\end{equation}
To prove this it suffices to check that
\begin{equation}\label{eq:log_+_log_-}
\sup_{n\in \N} \E \log_+^p ((k/n) b_{1}^{(n)})) <\infty,
\qquad
\sup_{n\in \N} \E \log_-^p ((k/n) b_{1}^{(n)})) <\infty,
\end{equation}
where for $x>0$ we defined $\log_+(x) = \max(\log x, 0)$ and $\log_-(x) = \max\{-\log x,0\}$, so that $|\log x| = \log_+x + \log _-x$. The first claim in~\eqref{eq:log_+_log_-} follows from the estimate $\log_+^p x  = O(x)$ together with the identity $\E [(k/n) b_{1}^{(n)}] = 1$ which holds by exchangeability. To prove the second claim in~\eqref{eq:log_+_log_-}, we first observe that $\P[b_1^{(n)} = j]$ is a decreasing function of $j$, which follows from the explicit formula~\eqref{eq:b_1_n_distribution}. Hence,
$$
\E \log_-^p ((k/n) b_{1}^{(n)}))
=
\sum_{j=1}^{\lfloor n/k \rfloor} \log_-^p (jk/n) \, \P[b_1^{(n)} = j]
\leq
\P[b_1^{(n)} = 1] \sum_{j=1}^{\lfloor n/k \rfloor} |\log(jk/n)|^p
=
\frac{k-1}{n-1} \sum_{j=1}^{\lfloor n/k \rfloor} |\log(jk/n)|^p,
$$
which is bounded as a  Riemann sum for $\int_0^1 |\log x|^p \dd x <\infty$. This completes the proof of~\eqref{eq:p_moment_bounded}.

Relation~\eqref{eq:log_b_1_n}, combined with the uniform integrability established in~\eqref{eq:p_moment_bounded}, implies  that
$$
\E |\log b_{1}^{(n)}  - \log (n/k)| = O(1),
\qquad
\E (\log b_{1}^{(n)}  - \log (n/k))^2 = O(1).
$$
Observe that $|\log b_1^{(n)} - H_{b_1^{(n)}}|$ is bounded by a non-random constant. Using the triangle inequality and the inequality $(a+b)^2\leq 2a^2 + 2b^2$, we get
\begin{align*}
& \E \left|H_{b_1^{(n)}} - \log (n/k) \right| \leq  \E |\log b_{1}^{(n)}  - \log (n/k)| +O(1) = O(1),\\
&\E \left(H_{b_1^{(n)}}  - \log (n/k)\right)^2
\leq
2 \E \left(H_{b_1^{(n)}} - \log {b_1^{(n)}}\right)^2
+
2 \E (\log {b_1^{(n)}}  - \log (n/k))^2
=
O(1).
\end{align*}
Applying the triangle inequality to the first relation and expanding the square in the second relation yields
$$
\E  H_{b_{1}^{(n)}} = \log (n/k) + O(1),
\qquad
\E H^2_{b_{1}^{(n)}}    = 2 \log (n/k) \E H_{b_{1}^{(n)}} - \log^2 (n/k) + O(1).
$$
It follows that
\begin{align*}
\Var H_{b_{1}^{(n)}}
&=
\E H^2_{b_{1}^{(n)}} - (\E H_{b_{1}^{(n)}})^2
=
2 \log (n/k) \E H_{b_{1}^{(n)}} - \log^2 (n/k) - (\E H_{b_{1}^{(n)}})^2 + O(1)\\
&=
-(\log (n/k)-\E H_{b_{1}^{(n)}})^2 + O(1)
=
O(1).
\end{align*}

We now proceed to the proof of~\eqref{eq:intermediate_cov_negative}. First we need to recall the notion of negative association; see~\cite{bulinski_shashkin_book}. For $\ell\in\N$ let $\mathcal M(\ell)$ denote the set of all real-valued, bounded, Borel functions on $\R^\ell$ that are nondecreasing in each coordinate. A random vector $(X_1,\ldots,X_k)$ is called \textit{negatively associated} if for every disjoint sets $I,J\subset \{1,\ldots,k\}$ and every functions $f\in \mathcal M(|I|)$, $g\in \mathcal M(|J|)$ it holds that
\begin{equation}\label{eq:negative_association}
\Cov(f(X_i: i\in I), g(X_j: j\in J))\leq 0.
\end{equation}
Although the next lemma, claiming the negative association of the uniform random compositions, sounds classical and will be proved by standard methods, we did not find it among the numerous similar examples listed in~\cite{bulinski_shashkin_book} and~\cite{joag_dev_proschan}.
\begin{lemma}
For every $n\in \N$ and $k\in \{1,\ldots,n\}$,  the random uniform composition $(b_1^{(n)},\ldots, b_{k}^{(n)})$ is negatively associated.
\end{lemma}
\begin{proof}
Recall that $(b_1^{(n)},\ldots,b_{k}^{(n)})$ has the same distribution as the vector $(G_1,\ldots,G_k)$ of i.i.d.\ geometric variables with arbitrary parameter $\theta\in (0,1)$ conditioned on the event $\{G_1+\ldots+G_k = n\}$, see equation~\eqref{eq:conitional_law}.  According to Theorem~1.23 of~\cite{bulinski_shashkin_book} or Theorem~2.6 of~\cite{joag_dev_proschan}, to prove negative association, it suffices to check that for every set $I\subset \{1,\ldots,n\}$ and for every function $f\in \mathcal M(I)$ the function
$$
n\mapsto \E \left[f(G_i, i\in I)\, \Big | \, \sum_{i\in I} G_i = n\right]
$$
is nondecreasing in $n$. There is no loss of generality in assuming that $I=\{1,\ldots,k\}$, so that our task reduces to proving that
$$
n\mapsto \E \left[f(b_1^{(n)},\ldots, b_k^{(n)})\right]
$$
is nondecreasing in $n$. To prove this, it suffices to construct the vectors $(b_1^{(n)},\ldots, b_k^{(n)})$, with $n= k,k+1,\ldots$, on a common probability space in such a way that $b_j^{(n)} \leq b_j^{(n+1)}$ for all $j=1,\ldots,k$ and $n\geq k$. But such a coupling using the P\'olya urn has already been constructed in Section~\ref{subsec:polya}. Clearly, in that coupling the number of balls of each color is nondecreasing in time, and the proof is complete.
\end{proof}
To complete the proof of~\eqref{eq:intermediate_cov_negative}, and thus of Theorem~\ref{theo:weak_LLN_intermediate}, it remains to observe that the function $x\mapsto H_{\lfloor x\rfloor}$, $x\in [1,n]$, is bounded and nondecreasing, which allows us to apply~\eqref{eq:negative_association} with $I=\{i\}$ and $J=\{j\}$.
\end{proof}

\begin{remark}\label{rem:CLT_intermedate_conj}
It is natural to conjecture that in the intermediate regime, all cumulants are asymptotically equivalent to the cumulants of the Poisson distribution with parameter $k\log (n/k)$, meaning that $\kappa_\ell (\Lah (n,k)) \sim k \log (n/k)$ for all $\ell\in \N$. This statement would imply a CLT in the intermediate regime.
More generally, we conjecture that
$$
\frac{\Lah(n,k) - \E \Lah (n,k)}{\sqrt{\Var \Lah (n,k)}} \todistr {\rm N}(0,1)
\;\Leftrightarrow\;
\Var \Lah (n,k) \ton \infty
\;\Leftrightarrow\;
n-k(n)\ton \infty.
$$
%centered by its expectation and normalized its standard deviation, satisfies a CLT in any regime of $k=k(n)$ such that $\Var \Lah (n,k)\to\infty$, which is satisfied if and only if  $n-k(n)\to\infty$.
%It remains open to compute the exact asymptotics of the variance of $\Lah(n,k)$ in the intermediate regime, where we only know from~\eqref{eq:Var_intermediate} that $\Var \Lah (n,k) = \Theta(k\log (n/k))$, as well as in the regime when $k/n\to 1$.
Since the existing results are sufficient to obtain a fairly complete picture of the threshold phenomena in the following Section~\ref{sec:convex_hulls_random_walks}, we refrain from studying these more technical questions here.
\end{remark}

%\section{Applications to convex hulls of random walks and Weyl chambers}
\section{Threshold phenomena for convex hulls of random walks}\label{sec:convex_hulls_random_walks}
\subsection{Formula for the expected number of faces}\label{subsec:convex_hulls_random_walks}
Let $\xi_1,\dots,\xi_n$ be a collection of possibly dependent random $d$-dimensional vectors  with partial sums
$$
S_i = \xi_1 + \dots + \xi_i,\quad  1\leq i\leq n,\quad  S_0=0.
$$
The sequence $S_0,S_1,\dots,S_n$ will be referred to as a \emph{random walk}. Let $n\geq d$.
We impose the following assumptions on the joint distribution of  the increments. %$(\xi_1,\dots,\xi_n)$:
\begin{enumerate}
\item[$(\text{Ex})$] \textit{Exchangeability:} For every permutation $\sigma$ of the set $\{1,\dots,n\}$ we have the following distributional equality of joint distributions:
$$
(\xi_{\sigma(1)},\dots,  \xi_{\sigma(n)}) \eqdistr (\xi_1,\dots,\xi_n).
$$
\item[$(\text{GP})$] \textit{General position:}
For every $1\leq i_1 < \dots < i_d\leq n$ the probability that the vectors $S_{i_1}, \dots,S_{i_d}$ are linearly dependent is $0$.
\end{enumerate}
For example, it is known from~\cite[Proposition~2.5]{KVZ17b} and~\cite[Example~1.1]{KVZ17} that Conditions $(\text{Ex})$ and $(\text{GP})$ are satisfied if $\xi_1,\dots,\xi_n$ are independent identically distributed, and for every hyperplane $H_0\subset \R^d$ passing through the origin we have $\P[S_i\in H_0] = 0$, for all $1\leq i\leq n$. Moreover, the second condition can be replaced by $\P[\xi_1\in H]=0$ for every affine hyperplane $H\subset \R^d$.

We are interested in the convex hull of this random walk which is a random polytope $C_{n,d}=\conv(S_0,S_1,\dots, S_n)\subset \R^d$ defined by~\eqref{eq:def_C_n_d}. For $\ell\in \{0,\ldots, d\}$, the number of $\ell$-dimensional faces of the polytope $C_{n,d}$ is denoted by $f_\ell(C_{n,d})$. As has already been mentioned in the introduction, see formula~\eqref{eq:E_F_k_C_n_main_theorem_introduct}, the following explicit formula has been obtained in~\cite{KVZ17}.

\begin{theorem}[Exact formula for expected face numbers]\label{theo:expected_walk}
Let $(S_i)_{i=0}^{n}$ be a random walk in $\R^d$, with $n\geq d$,  whose increments satisfy conditions $(\text{Ex})$ and $(\text{GP})$.
Then, for all $k \in \{1,\ldots,d\}$, %the expected number of $k$-faces of the convex hull $C_n$ is given by the formula
\begin{equation}\label{eq:E_F_k_C_n_main_theorem}
\E f_{k-1}(C_{n,d}) = \frac{2\cdot (k-1)!}{n!} \sum_{l=0}^{\infty}\stirling{n+1}{d-2l}  \stirlingsec{d-2l}{k}.
\end{equation}
\end{theorem}

In the following it will be more convenient to consider the polytope $C_{n-1,d}$ since it is defined as a convex hull of $n$ points. It is known~\cite[Remark~1.5]{KVZ17} that this polytope is simplicial with probability $1$. This means that for every $k\in \{1,\ldots, d\}$, all $(k-1)$-dimensional faces of $C_{n-1,d}$ are simplices. The maximal possible number of $(k-1)$-dimensional faces of $C_{n-1,d}$ is $\binom n k$. If the number of $(k-1)$-faces attains this bound, the polytope $C_{n-1,d}$ is said to be $k$-neighborly. We shall therefore be interested in the asymptotic behavior of the quantity $f_{k-1}(C_{n-1,d})/\binom nk$.
The formula of Theorem~\ref{theo:expected_walk} can be stated as follows:
\begin{equation}\label{eq:E_f_k_Lah}
\frac {\E f_{k-1}(C_{n-1,d})}{\binom{n}{k}}=\frac{2}{L(n,k)} \sum_{l=0}^{\infty}\stirling{n}{d-2l}  \stirlingsec{d-2l}{k}=2 \P[\Lah(n,k) \in \{d,d-2,d-4,\ldots\}].
\end{equation}
It follows from the well-known identities
$$
\sum_{j=k}^n (-1)^{n-j} \stirling{n}{j} \stirlingsec{j}{k} = 0
\quad
(\text{if } n>k),
\qquad
\sum_{j=k}^n  \stirling {n}{j} \stirlingsec{j}{k} = L(n,k)
$$
that
$$
\P[\Lah(n,k) \text{ takes even value}] = \P[\Lah(n,k) \text{ takes odd value}] = \frac{1}{2}
\quad
\text { for } n>k.
$$
Consequently, we can rewrite~\eqref{eq:E_f_k_Lah} as follows:
\begin{equation}\label{eq:E_f_k_Lah_dual}
1 - \frac {\E f_{k-1}(C_{n-1,d})}{\binom nk}
=
2 \P[\Lah(n,k) \in \{d+2,d+4,\ldots\}].
\end{equation}

In the rest of this section we shall use~\eqref{eq:E_f_k_Lah} and~\eqref{eq:E_f_k_Lah_dual} to uncover threshold phenomena for convex hulls of random walks. To describe our problem, let us fix some very large dimension $d$. Let us also take some $k\in \{1,\ldots,d\}$ which may be either fixed, or depend on $d$ in some way. We ask whether the number of $(k-1)$-dimensional faces of $C_{n-1,d-1}$ is equal or close to the maximal possible number $\binom nk$. If $n$ is not much larger than $d$, we expect  $f_{k-1}(C_{n-1,d-1})$ to be close or even equal to $\binom nk$. On the other hand, if $n$ is sufficiently large, we expect $f_{k-1}(C_{n-1,d-1})/ \binom nk$ to approach $0$. Somewhere in between there should be a threshold at which a phase transition occurs. Following Donoho and Tanner~\cite{donoho_neighborliness_proportional,donoho_tanner_neighborliness,donoho_tanner,DonohoTanner} we distinguish between weak and strong thresholds, which are statements about $\E f_{k-1}(C_{n-1,d-1})/ \binom nk$ and $\binom nk - \E f_{k-1}(C_{n-1,d-1})$, respectively. As we shall see in the following, the phase transitions occur surprisingly late. For example, for fixed $k$ the weak transition occurs if $n$ is near $\eee^{d/k}$.

%Let us remark that a formula similar to~\eqref{eq:E_F_k_C_n_main_theorem} exists for the expected $f$-vector of the \textit{positive} hull of a random walk~\cite{godland_kabluchko_positive_hulls}. Its asymptotic analysis is similar and will be omitted.

\subsection{Threshold phenomena for face numbers: the regime of constant \texorpdfstring{$k$}{k}}
We begin by analyzing the case in which $k$ is constant.
\begin{theorem}[Weak threshold in the constant $k$ regime]\label{theo:phase_trans_expect_constant_k}
Let $d\to\infty$ and $n=n(d)> d$ be a function of $d$ such that
$$
\gamma := \lim_{d\to\infty} \frac{\log n(d)}{d} \in [0,+\infty].
$$
Then, for every fixed $k\in\N$, we have
$$
\lim_{d\to\infty} \frac {\E f_{k-1}(C_{n-1,d})}{\binom nk}
=
\begin{cases}
1, &\text{ if } \gamma< 1/k,\\
0, &\text{ if } \gamma >1/k.
\end{cases}
$$
Moreover, in the critical case when $\gamma=1/k$, more precisely if $\log n(d) = \frac 1 k (d + c\sqrt d + o(\sqrt d))$ for some fixed $k\in \N$ and some constant $c\in\R$, then
$$
\lim_{d\to\infty} \frac {\E f_{k-1}(C_{n-1,d})}{\binom nk}  = \frac 1 {\sqrt {2\pi}} \int_{-\infty}^{-c} \eee^{-x^2/2} \dd x.
$$
\end{theorem}
\begin{proof}
Consider first the case $\gamma > 1/k$. By~\eqref{eq:E_f_k_Lah} we have
$$
\frac {\E f_{k-1}(C_{n-1,d})}{\binom nk}
\leq
2 \P[\Lah(n,k)\leq d]
=
2 \P[\Lah(n,k)\leq  k x_n \log n]
$$
with $x_n:= d/(k\log n) \to 1/(\gamma k) < 1$ as $d\to\infty$. By Theorem~\ref{theo:precise_ldp} (or Theorem~\ref{theo:clt}), the probability on the right-hand side goes to $0$.
Let now $\gamma <1/k$.
Then, by~\eqref{eq:E_f_k_Lah_dual},
\begin{align*}
1 - \frac {\E f_{k-1}(C_{n-1,d})}{\binom nk}
%=
%2\P[\Lah(n,k) \in \{d+2,d+4,\ldots\}]\\
\leq
2 \P[\Lah(n,k) \geq d]
=
2 \P[\Lah(n,k)\geq  k x_n \log n]
\end{align*}
with $x_n:= d/(k\log n) \to 1/(\gamma k) > 1$.
By Theorem~\ref{theo:precise_ldp}, the probability on the right-hand side goes to $0$.

Consider now the critical case, that is, let $\log n = \frac 1k (d+ c\sqrt d + o(\sqrt d))$. If the right-hand side of~\eqref{eq:E_f_k_Lah} could be replaced by the simpler quantity $\P[\Lah(n,k)\leq d]$, the claim could be deduced from the central limit theorem as follows:
\begin{align}
\P[\Lah(n,k)\leq d]
&=
\P\left[\frac{\Lah(n,k) - k\log n}{\sqrt{k\log n}}\leq \frac{d - k\log n}{\sqrt{k\log n}} \right]\notag\\
&=
\P\left[\frac{\Lah(n,k) - k\log n}{\sqrt{k\log n}}\leq - c + o(1)\right]\notag\\
&\tond
\frac 1 {\sqrt {2\pi}} \int_{-\infty}^{-c} \eee^{-x^2/2} \dd x\label{eq:proof_crit_case_clt}
\end{align}
by Theorem~\ref{theo:clt}. Unfortunately, the right-hand side of~\eqref{eq:E_f_k_Lah} runs over $d-\ell$ with even $\ell=0,2,\ldots$ (and there is a factor $2$ compensating for that), and more efforts are needed to prove the claim. Recall that the Lah distribution is unimodal by Corollary~\ref{cor:unimodal}. If $m_{n,k}$ denotes the largest mode of $\Lah(n,k)$, then by Proposition~\ref{prop:mode} we have
$$
m_{n,k}=k\log n +O(1)=d + c\sqrt d + o(\sqrt d).
$$
%Theorem~\ref{theo:clt} implies\footnote{Even more precise statement is given in Proposition~\ref{prop:mode}.} $m_{n,k} = k\log n + o(\sqrt {\log n})%$, as $n\to\infty$. Indeed, assuming that for some $\eps>0$, $m_{n,k}$ is smaller than $s_n:= k\log n - 2\eps \sqrt {\log n}$ for infinitely many $n%$'s, we obtain
%$$
%\P[ s_n \leq \Lah(n,k) \leq  s_n + \eps  \sqrt {\log n}]
%\geq
%\P[k\log n \leq \Lah(n,k) \leq  k\log n + \eps  \sqrt {\log n}],
%$$
%for infinitely many  $n$'s, which contradicts the CLT stated in Theorem~\ref{theo:clt}. Similarly, it is not possible to have $m_{n,k}\geq k\log n + %2\eps \sqrt {\log n}$ for infinitely many $n$'s.
%Thus, $m_n= k\log n + o(\sqrt {\log n}) =d + c\sqrt d + o(\sqrt d)$.
If $c>0$, then it follows that $d<m_{n,k}$ for sufficiently large $d$ and hence
$$
\P[\Lah(n,k) \leq d-1] \leq 2 \P[\Lah(n,k) \in \{d,d-2,d-4,\ldots\}] \leq \P[\Lah(n,k) \leq d+1].
$$
Applying the CLT to both sides as explained in~\eqref{eq:proof_crit_case_clt} completes the proof for $c>0$. If $c<0$, we can pass to the complementary events via~\eqref{eq:E_f_k_Lah_dual} and argue analogously. Finally, the case $c=0$ follows by a sandwich argument.
\end{proof}

The above Theorem~\ref{theo:phase_trans_expect_constant_k} deals with expected face numbers only. More interesting is to prove that neighborliness holds with high probability rather than only in expectation. The next result is a strong threshold in the terminology of Donoho and Tanner~\cite{donoho_neighborliness_proportional,donoho_tanner_neighborliness,donoho_tanner,DonohoTanner}.

\begin{theorem}[Strong threshold in the constant $k$ regime]
Fix $k\in \N$. Let $n=n(d) > d$ be an integer sequence.  If $n(d) \leq  \eee^{d/(k(\eee + \eps))}$ for some $\eps>0$ and all sufficiently large $d$, then
\begin{equation}\label{eq:strong_thresh_fin_k1}
\binom{n}{k}  - \E f_{k-1}(C_{n-1,d}) = O(n^{-\eta})
\end{equation}
for some $\eta>0$, and the polytope $C_{n-1,d}$ is $k$-neighborly with probability approaching $1$, more precisely
\begin{equation}\label{eq:strong_thresh_fin_k2}
\P\left[f_{k-1}(C_{n-1,d}) = \binom nk \right] = 1 - O(n^{-\eta}).
\end{equation}
If, on the other hand, $n(d) \geq  \eee^{d/(k(\eee - \eps))}$ for some $\eps\in (0,\eee)$ and all sufficiently large $d$, then
\begin{equation}\label{eq:strong_thresh_fin_k3}
\lim_{d\to\infty} \left(\binom{n}{k} - \E f_{k-1}(C_{n-1,d})\right) = +\infty.
\end{equation}
\end{theorem}
\begin{proof}
Let us prove~\eqref{eq:strong_thresh_fin_k1} and~\eqref{eq:strong_thresh_fin_k2}.
Since on the event $f_{k-1}(C_{n-1,d}) \neq \binom nk$ we even have $\binom nk - f_{k-1}(C_{n-1,d}) \geq 1$, the following estimate holds:
$$
\binom nk -\E f_{k-1}(C_{n-1,d}) = \E\left[\binom nk - f_{k-1}(C_{n-1,d})\right] \geq  \P\left[f_{k-1}(C_{n-1,d}) \neq \binom nk \right].
$$
It will be convenient to write this inequality in the form
\begin{equation}\label{eq:est_for_strong_threshold}
\P\left[f_{k-1}(C_{n-1,d}) \neq \binom nk \right] \leq \binom nk \left(1- \frac{\E f_{k-1}(C_{n-1,d})}{\binom nk}\right).
\end{equation}
The same argumentation as in the proof of Theorem~\ref{theo:phase_trans_expect_constant_k} yields then
$$
\P\left[f_{k-1}(C_{n-1,d}) \neq \binom nk \right]
\leq
\binom nk \left(1- \frac{\E f_{k-1}(C_{n-1,d})}{\binom nk}\right)
\leq 2 n^k \P[\Lah(n,k)\geq  k x_n \log n],
$$
where $x_n:= d/(k\log n) \geq \eee +\eps$. Note that the convex function $f(x) := x\log x - x +1$, $x>0$, satisfies $f(\eee) = 1$ and  $f(\eee + \eps)>1$.  It follows from Theorem~\ref{theo:precise_ldp} that
$$
\P\left[f_{k-1}(C_{n-1,d}) \neq \binom nk \right]
\leq
2 n^k n^{- k f(\eee + \eps) + o(1)}
=
2 n^{k(1 - f(\eee + \eps)) + o(1)}
=
O(n^{-\eta}),
$$
which proves~\eqref{eq:strong_thresh_fin_k1} and~\eqref{eq:strong_thresh_fin_k2}.
To prove~\eqref{eq:strong_thresh_fin_k3}, we assume that  $n \geq  \eee^{d/(k(\eee - \eps))}$. Making $\eps$ smaller, if necessary, we may assume that $\eps\in (0,\eee-1)$. By~\eqref{eq:E_f_k_Lah_dual}, we have
\begin{equation}\label{eq:wspom123}
\binom{n}{k} - \E f_{k-1}(C_{n-1,d}) = \binom{n}{k} \left(1 - \frac {\E f_{k-1}(C_{n-1,d})}{\binom nk}\right)
=
2 \binom{n}{k} \P[\Lah(n,k) \in \{d+2,d+4,\ldots\}].
\end{equation}
Let $m_{n,k}$ be the largest mode of $\Lah(n,k)$. Then $m_{n,k} = k \log n + O(1)$ by Proposition~\ref{prop:mode}. Fix some $\eta\in (0, \eee-\eps - 1)$. Without loss of generality we may assume that $d > (1+\eta) m_{n,k}$ for all sufficiently large $d$. Indeed, if $d \leq (1+\eta)m_{n,k}$ along some subsequence, then we may increase $d$ by an even number without destroying the condition $n \geq  \eee^{d/(k(\eee - \eps))}$ and such that $d$ becomes larger than $(1+\eta)m_{n,k}$. Since the right-hand side of~\eqref{eq:wspom123} decreases under this operation, and since we intend to show that it diverges to infinity, we can and do assume that $d > (1+\eta) m_{n,k}$ for all sufficiently large $d$. Using the unimodality of the Lah distribution, we have
$$
\binom{n}{k} - \E f_{k-1}(C_{n-1,d})
\geq
\binom{n}{k} \P[\Lah(n,k) \geq d+2]
\geq
c n^k \P[\Lah(n,k) \geq k x_n \log n],
$$
where $c>0$ is sufficiently small and $x_n:= d/(k\log n)$ has all its limit points in $[1+\eta, \eee-\eps]$.
Then, Theorem~\ref{theo:precise_ldp} yields~\eqref{eq:strong_thresh_fin_k3}.
\end{proof}

Let us finally mention a conjecture which we verified numerically for all $d\leq 50$, $1\leq k\leq d$, $n\leq k+100$. Its part (b) is quite surprising in view of~\eqref{eq:E_f_k_Lah_dual} and Proposition~\ref{prop:stoch_monotonicity}.
%but unfortunately does not follow from these two results because the summation on the right-hand side of~\eqref{eq:E_f_k_Lah}, being quite similar to the distribution function of $\Lah(n,k)$, is not equal to it.
\begin{conjecture}
Fix $d\in \N$ and $k\in\{1,\ldots,d\}$. Then:
\begin{itemize}
\item[(a)] The function $n\mapsto \E f_{k-1}(C_{n-1,d})$ is increasing for $n\geq k$.
\item[(b)] The function $n\mapsto \frac {\E f_{k-1}(C_{n-1,d})}{\binom nk}$ is decreasing  (if $d-k$ is even) and increasing (if $d-k$ is odd), for all $n\geq k$.
\end{itemize}
\end{conjecture}
%%Numerical verification: Lah distribution 1.nb
In the setting of Cover-Efron and Schl\"afli random cones, a function similar to that appearing in (b) is always decreasing, as has been recently shown by Hug and Schneider~\cite{HugSchneiderThresholdPhenomenaPart2}.

\subsection{Threshold phenomena for face numbers: the regime of linearly growing \texorpdfstring{$k$}{k}}
Let us now turn to the proportional growth regime. It has been first studied by Vershik and Sporyshev~\cite{vershik_sporyshev_asymptotic_faces_random_polyhedra1992} in the context of random projections of the regular simplex.
\begin{theorem}[Weak threshold in the linear regime]\label{theo:phase_trans_expect_linear_k}
Let $d\to\infty$ and $k=k(d)$, $n=n(d)$  be functions of $d$ such that
$$
k\sim \alpha n  \qquad \text{ and } \qquad  d\sim \beta n, \qquad  \text{ as } d\to\infty,
$$
for some constants $0 <\alpha < \beta <1$.
Then,
$$
\lim_{d\to\infty} \frac {\E f_{k-1}(C_{n-1,d})}{\binom{n}{k}}
=
\begin{cases}
1, &\text{ if } \beta > -\frac{\alpha \log \alpha}{1-\alpha},\\
0, &\text{ if } \beta < -\frac{\alpha \log \alpha}{1-\alpha}.
\end{cases}
$$
In the critical case, more precisely, when $k\sim \alpha n$ and $d= \E \Lah (n,k) + c\sqrt n + o(\sqrt n)$ for some $\alpha\in (0,1)$ and $c\in \R$, we have
$$
\lim_{d\to\infty} \frac {\E f_{k-1}(C_{n-1,d})}{\binom{n}{k}} =
\frac1 {\sqrt{2\pi}} \int_{-\infty}^{c/\sigma(\alpha)}\eee^{-x^2/2} \dd x,
$$
where $\sigma^2(\alpha) := -\frac{\alpha}{1-\alpha} - \frac{\alpha(\alpha+1)\log \alpha}{(1-\alpha)^2} - \frac{\alpha^2 \log^2 \alpha}{(1-\alpha)^3}$ is the variance appearing in Theorem~\ref{thm:clt_central_regime}.
\end{theorem}
\begin{proof}
By~\eqref{eq:E_f_k_Lah} we have
\begin{equation}\label{eq:wspom1}
\frac {\E f_{k-1}(C_{n-1,d})}{\binom nk}= 2 \P[\Lah(n,k) \in \{d,d-2,d-4,\ldots\}] \leq 2 \P[\Lah(n,k) \leq d] = 2 \P[\Lah(n,k) \leq (\beta + o(1))n].
\end{equation}
If now $\beta < -\frac{\alpha \log \alpha}{1-\alpha}$, then~\eqref{eq:LDP_Lah_leq} is applicable and shows that the probability on the right-hand side converges to $0$ exponentially fast.

On the other hand, by~\eqref{eq:E_f_k_Lah_dual} we have
\begin{equation}\label{eq:wspom2}
1 - \frac {\E f_{k-1}(C_{n-1,d})}{\binom{n}{k}}=2 \P[\Lah(n,k) \in \{d+2,d+4,\ldots\}] \leq 2\P[\Lah(n,k) \geq d] = 2 \P[\Lah(n,k) \geq (\beta + o(1))n].
\end{equation}
If $\beta > -\frac{\alpha \log \alpha}{1-\alpha}$, then~\eqref{eq:LDP_Lah_geq} is applicable and shows that the probability on the right-hand side converges to $0$ exponentially fast.

Consider finally the critical case. Using the unimodality of the Lah distribution (with the mode satisfying $m_{n,k} = \E \Lah(n,k) + o(\sqrt n)$ by Theorem~\ref{thm:clt_central_regime}) and arguing as in the proof of Theorem~\ref{theo:phase_trans_expect_constant_k}, our task reduces to showing that
$$
\lim_{d\to\infty} \P[\Lah(n,k) \leq d] = \frac1 {\sqrt{2\pi}} \int_{-\infty}^{c/\sigma(\alpha)}\eee^{-x^2/2} \dd x.
$$
But this follows by writing
$$
\P[\Lah(n,k) \leq d]
=
\P\left[\frac{\Lah(n,k)-\E \Lah (n,k)}{\sqrt n} \leq \frac{d- \E \Lah (n,k)}{\sqrt n}\right]
=
\P\left[\frac{\Lah(n,k)-\E \Lah (n,k)}{\sqrt n} \leq c+o(1)\right]
$$
and applying Theorem~\ref{thm:clt_central_regime}.
\end{proof}

\begin{remark}\label{rem:weak_threshold}
Let us restate the above results in the notation consistent with the one used by Vershik and Sporyshev~\cite{vershik_sporyshev_asymptotic_faces_random_polyhedra1992} and Donoho and Tanner~\cite{donoho_tanner_neighborliness,donoho_tanner}. Following these papers, define
\begin{equation}\label{eq:rho_delta}
\rho := \lim_{n\to\infty} \frac{k(n)}{d(n)} = \frac {\alpha}{\beta} \in (0,1),
\qquad
\delta:=\lim_{n\to\infty} \frac{d(n)}{n} = \beta \in (0,1).
\end{equation}
Similarly to these papers, we say that a function $\delta \mapsto \rho_{\text{weak}}(\delta)$ defines a \textit{weak threshold} for convex hulls of random walks if
\begin{align}
\E f_{k-1}(C_{n-1,d}) &= (1-o(1)) \cdot \binom nk, & \quad &\text{ provided that } \rho < \rho_{\text{weak}} (\delta), \label{eq:weak_threshold_1}\\
\E f_{k-1}(C_{n-1,d}) &= o(1) \cdot  \binom nk, & \quad    &\text{ provided that } \rho > \rho_{\text{weak}} (\delta).
\label{eq:weak_threshold_2}
\end{align}
Using Theorem~\ref{theo:phase_trans_expect_linear_k}, we are able to identify the weak threshold explicitly in terms of the Lambert function $W_{-1}(x)$ which is defined as follows: For $-1/\eee < x < 0$, the equation $w\eee^{w} = x$ has two solutions, $w=W_0(x)$ and $w=W_{-1}(x)$, satisfying $W_{-1}(x) < -1 < W_0(x)<0$ and defining two branches of the Lambert function.  Then, we claim that the function
$$
\rho_{\text{weak}}(\delta) := h^{-1}(\delta)/\delta = -1/ W_{-1}(-\delta \eee^{-\delta}), \qquad \delta\in (0,1),
$$
is the weak threshold; see formula~\eqref{eq:h_def} and Figure~\ref{fig:thresholds} (solid line). Indeed,  $\rho < \rho_{\text{weak}}(\delta)$ is equivalent to $\alpha/\beta < h^{-1}(\beta) / \beta$, which is equivalent to $\beta > h(\alpha)$. Knowing this, Theorem~\ref{theo:phase_trans_expect_linear_k} applies and yields~\eqref{eq:weak_threshold_1}, while~\eqref{eq:weak_threshold_2} follows similarly. Note that $\rho = \rho_{\text{weak}}(\delta)$ is   the unique solution to $(1/\rho) \eee^{-1/\rho} = \delta \eee^{-\delta}$ with $\rho \in (0,1)$, which can be compared to~\cite[Theorem~1]{vershik_sporyshev_asymptotic_faces_random_polyhedra1992}, where a similar characterization of the threshold (involving the Mills ratio function) is given for random projections of the regular simplex.
Regarding the behavior of the weak threshold as $\delta\downarrow 0$, it is easy to check that $\rho_{\text{weak}}(\delta) \sim 1/ |\log \delta|$, compare~\cite[Theorems~1.2, 1.4]{donoho_tanner}, where similar asymptotics are stated for weak thresholds of Gaussian polytopes, namely $\rho_{\text{weak}}^{\text{GP}}(\delta) \sim 1/(2|\log \delta|)$, and their symmetric analogues.
\end{remark}

\begin{figure}[t]
\begin{center}
\includegraphics[width=0.5\textwidth ]{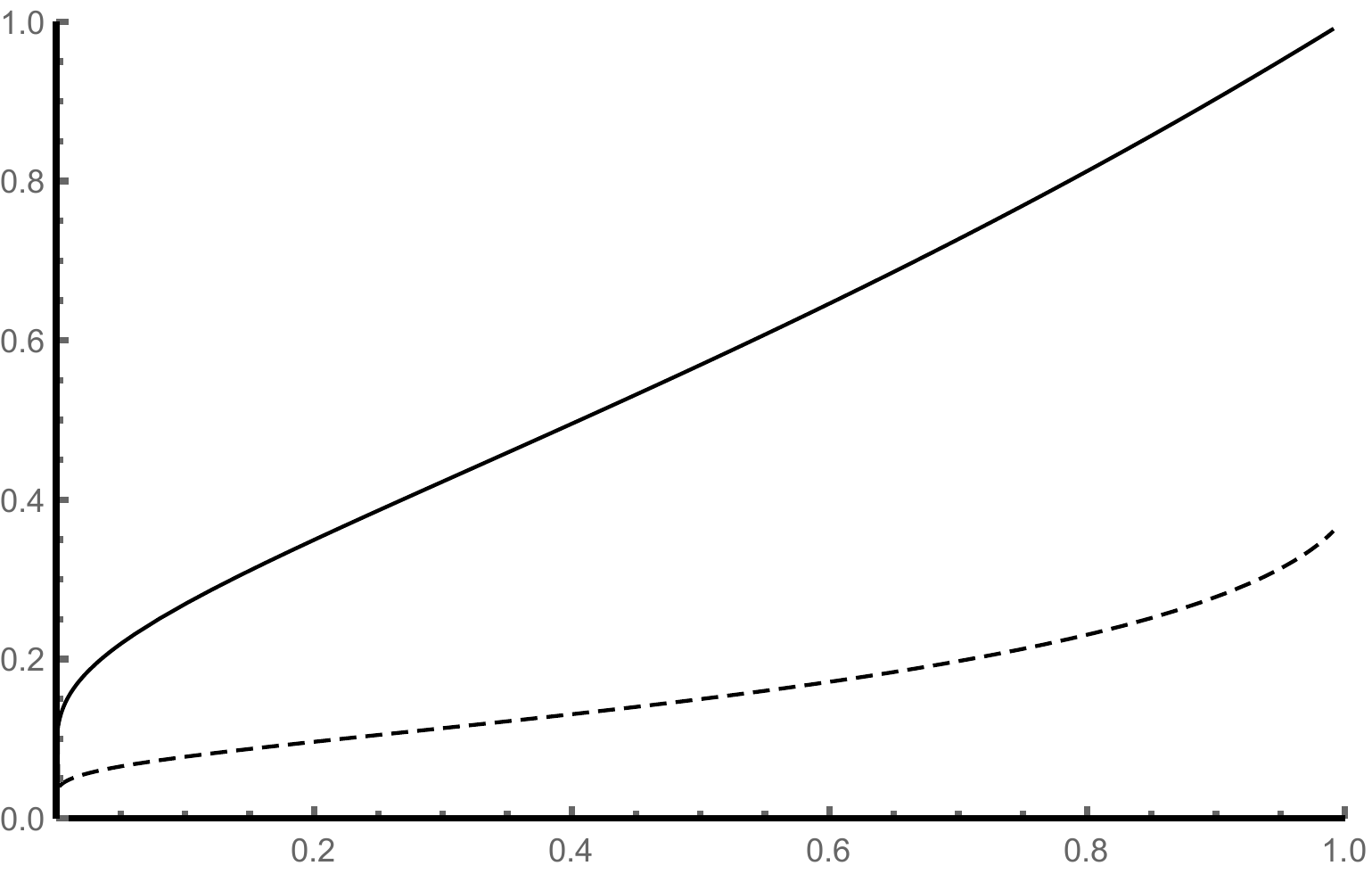}
%Lah Distributions 1.nb
\end{center}
\caption
{
Thresholds for convex hulls of random walks from top to bottom: $\rho_{\text{weak}}(\delta)$ (solid) and $\rho_{\text{strong}}(\delta)$ (dashed).
}
\label{fig:thresholds}
\end{figure}

%\begin{remark}
%B refining the analysis used in the proof of Theorem~\ref{theo:expect_asympt}, it is possible to show that in the regime when $k/n$ %stays bounded away from $0$ and $1$, we have
%$$
%\E \Lah (n,k) = \frac{k \log (n/k)}{1-k/n} + O(1), \qquad \text{ as } n\to\infty.
%$$
%It follows that in the critical case of Theorem~\ref{theo:phase_trans_expect_linear_k} one can replace $\E \Lah (n,k)$ by its %asymptotic expression $\frac{k \log (n/k)}{1-k/n}$.
%\end{remark}

\begin{theorem}[Strong threshold in the linear regime]\label{theo:phase_trans_expect_linear_k_strong}
Let $d\to\infty$ and $k=k(d)$, $n=n(d)$  be functions of $d$ such that
$$
k\sim \alpha n  \qquad \text{ and } \qquad  d\sim \beta n, \qquad  \text{ as } d\to\infty,
$$
for some constants $0 < \alpha < \beta <1$.
If $I_{\alpha}(\beta) + \alpha \log \alpha + (1-\alpha) \log (1-\alpha) >0$, where $I_\alpha(\beta)$ is the rate function from Theorem~\ref{theo:ldp_central}, then
\begin{equation}\label{eq:strong_thresh1}
\binom{n}{k}  - \E f_{k-1}(C_{n-1,d}) = O(\eee^{-\eta n})
\end{equation}
for some $\eta>0$, and the polytope $C_{n-1,d}$ is $k$-neighborly with probability converging to $1$,  more precisely,
\begin{equation}\label{eq:strong_thresh2}
\P\left[f_{k-1}(C_{n-1,d}) = \binom nk \right] = 1- O(\eee^{-\eta n}).
\end{equation}
If, on the other hand, $I_{\alpha}(\beta) + \alpha \log \alpha + (1-\alpha) \log (1-\alpha) < 0$, then
\begin{equation}\label{eq:strong_thresh3}
\lim_{d\to\infty} \left(\binom{n}{k} - \E f_{k-1}(C_{n-1,d})\right) = +\infty.
\end{equation}
\end{theorem}

\begin{proof}
Let $I_{\alpha}(\beta) + \alpha \log \alpha + (1-\alpha) \log (1-\alpha) >0$. First of all, we argue that this implies $\beta > -\frac{\alpha \log \alpha}{1-\alpha}=h(\alpha)$. Since the function $\beta \mapsto I_{\alpha}(\beta)$ decreases as $\beta$ moves from $\alpha$ to $h(\alpha)$, see Remark~\ref{rem:I_alpha_beta_properties}, it suffices to show that
$$
I_{\alpha}(\alpha) + \alpha \log \alpha + (1-\alpha) \log (1-\alpha) = \log \left(1- h^{-1}(\alpha)\right) -\alpha \log \left( - \log h^{-1}(\alpha)\right) \leq 0,
$$
see~\eqref{eq:I_s_alpha} for the first identity. The inequality follows from the estimate
$$
(1-w)^{1/\alpha}\leq 1-w\leq -\log w,\quad w\in(0,1],\quad \alpha\in (0,1],
$$
upon substitution $w=h^{-1}(\alpha)\in (0,1)$ and taking the logarithms.

We now proceed to the proof of~\eqref{eq:strong_thresh1} and~\eqref{eq:strong_thresh2}.
Using~\eqref{eq:est_for_strong_threshold} and~\eqref{eq:wspom2}, we obtain
\begin{equation}
\P\left[f_{k}(C_{n-1,d}) \neq \binom nk \right] \leq \binom nk \left(1- \frac{\E f_{k-1}(C_{n-1,d})}{\binom nk}\right)
\leq
2 \binom nk  \P[\Lah(n,k) \geq (\beta + o(1))n].
\end{equation}
By the Stirling formula,
\begin{equation}\label{eq:binom_stirl}
\binom nk = \eee^{ - (\alpha \log \alpha + (1-\alpha) \log (1-\alpha)) n +o(n)}, \qquad n\to\infty.
\end{equation}
Under the condition $\beta > -\frac{\alpha \log \alpha}{1-\alpha}$  we can apply~\eqref{eq:LDP_Lah_geq} which yields
\begin{equation}\label{eq:wspom_LDP}
\P[\Lah(n,k) \geq (\beta + o(1))n]  =  \eee^{- I_\alpha(\beta) n +o(n)}.
\end{equation}
Taking everything together, we obtain the claims~\eqref{eq:strong_thresh1} and~\eqref{eq:strong_thresh2}.

To prove~\eqref{eq:strong_thresh3}, assume that $I_{\alpha}(\beta) + \alpha \log \alpha + (1-\alpha) \log (1-\alpha) < 0$. By~\eqref{eq:E_f_k_Lah_dual}, we have
$$
\binom{n}{k} - \E f_{k-1}(C_{n-1,d}) = \binom{n}{k} \left(1 - \frac {\E f_{k-1}(C_{n-1,d})}{\binom nk}\right)
=
2 \binom{n}{k} \P[\Lah(n,k) \in \{d+2,d+4,\ldots\}].
$$
We may assume that $\beta > -\frac{\alpha \log \alpha}{1-\alpha}$ since otherwise we may increase $d$ by an even number (without changing the corresponding $n$ and $k$),  which makes the probability on the right-hand side smaller.  Under $\beta > -\frac{\alpha \log \alpha}{1-\alpha}$, we can use the unimodality of $\Lah (n,k)$ as in the proof of Theorem~\ref{theo:phase_trans_expect_constant_k} to estimate
$$
\binom{n}{k} - \E f_{k-1}(C_{n-1,d}) \geq  \binom{n}{k} \P[\Lah(n,k) \geq d+2]
=
\binom{n}{k}\P[\Lah(n,k) \geq (\beta + o(1))n].
$$
It follows from~\eqref{eq:binom_stirl} and~\eqref{eq:wspom_LDP} that the right-hand side is larger than $\eee^{\eta n}$, for some $\eta>0$ and all sufficiently large $n$.
\end{proof}

\begin{figure}[t]
\begin{center}
\includegraphics[width=0.4\textwidth ]{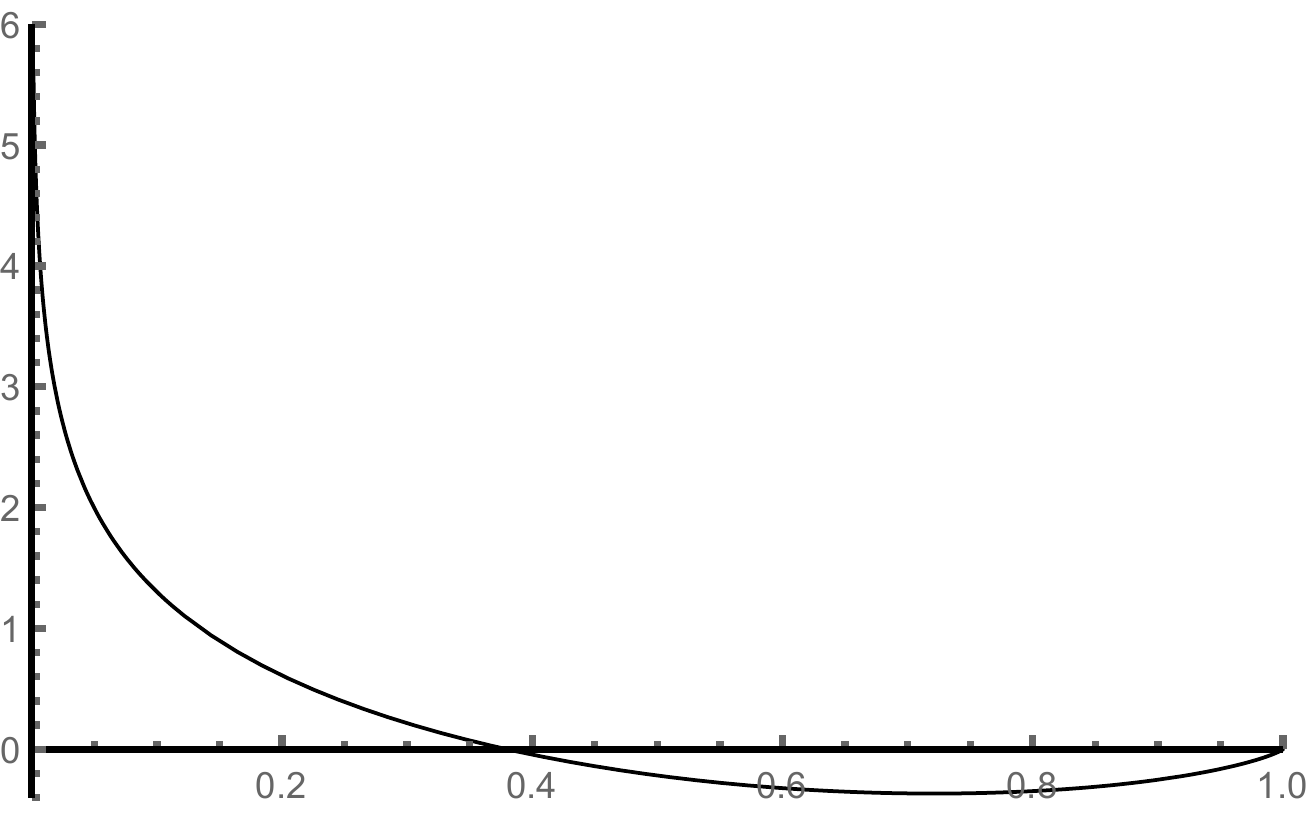}
\includegraphics[width=0.4\textwidth ]{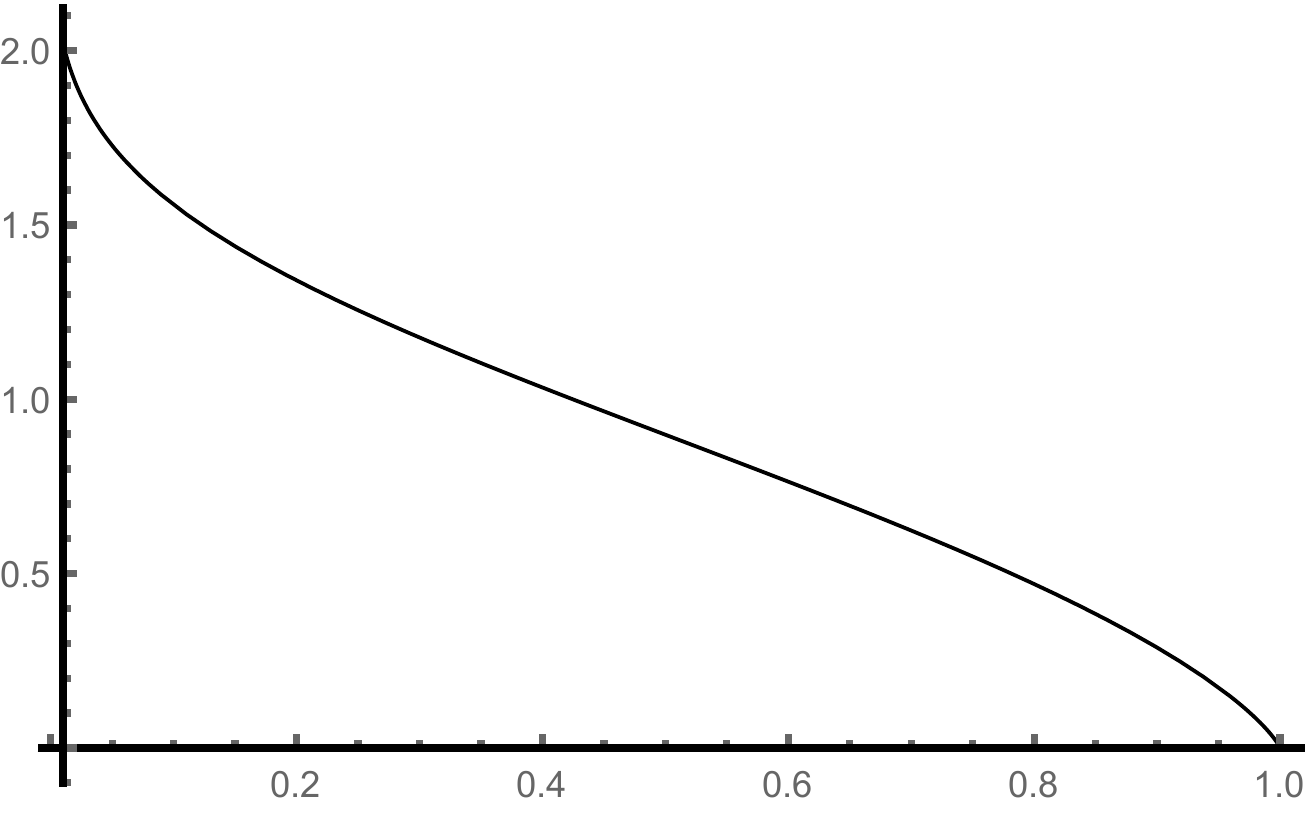}
%Lah Distributions 1.nb
\end{center}
\caption
{
Functions needed to define the strong threshold. Left: $\rho \mapsto -\rho \log (h^{-1}(1/\rho) - 1) + \log \log h^{-1}(1/\rho)$. Right: $\delta \mapsto -(1/\delta) \log (1 - h^{-1}(\delta))+ \log(-\log h^{-1}(\delta))$.
}
\label{fig:thresholds_strong}
\end{figure}

%\begin{remark}
%If we fix $\beta\in (0,1)$, then the strong neighborliness holds for a sufficiently small $\alpha>0$. Proof: follows fromt he next remark, therefore omitted.
%\end{remark}

\begin{remark}\label{rem:strong_threshold}
Let us restate the above results in the notation of Donoho and Tanner~\cite{donoho_tanner_neighborliness,donoho_tanner}. We assume~\eqref{eq:rho_delta}. A function  $\delta\mapsto \rho_{\text{strong}}(\delta)$ is said to be a \textit{strong threshold} for convex hulls of random walks if
\begin{align}
\lim_{n\to\infty} \left(\binom nk - \E f_{k-1}(C_{n-1,d})\right) &= 0, & \quad &\text{ provided that } \rho < \rho_{\text{strong}} (\delta), \label{eq:strong_threshold_1}\\
\lim_{n\to\infty} \left(\binom nk - \E f_{k-1}(C_{n-1,d})\right) &= +\infty, & \quad    &\text{ provided that } \rho > \rho_{\text{strong}} (\delta).
\label{eq:strong_threshold_2}
\end{align}
Theorem~\ref{theo:phase_trans_expect_linear_k_strong} yields the following description of the strong threshold:  $\rho = \rho_{\text{strong}}(\delta)\in (0,1)$ is the solution of the equation
\begin{equation}\label{eq:rho_delta_implicit}
-\rho \log (h^{-1}(1/\rho) - 1) + \log \log h^{-1}(1/\rho) = -(1/\delta)  \log (1 - h^{-1}(\delta))+ \log(-\log h^{-1}(\delta)),
\end{equation}
for $\delta\in (0,1)$; see Figure~\ref{fig:thresholds} (dashed line). The plots shown in Figure~\ref{fig:thresholds_strong}  suggest that the right-hand side, viewed as a function of $\delta\in (0,1)$, decreases from $+\infty$ to $0$, whereas the left-hand side, viewed as a function of $\delta \in (0, \rho_*)$ with $\rho_*= 0.3798\ldots$, decreases from $+\infty$ to $0$, even though we did not verify these claims rigorously. Hence, the solution to the above equation~\eqref{eq:rho_delta_implicit} exists and is unique. Now, let us prove~\eqref{eq:strong_threshold_1} and~\eqref{eq:strong_threshold_2}.  Recalling~\eqref{eq:I_2} and~\eqref{eq:rho_delta_implicit} we have
$$
\frac{I_\alpha(\beta) + \alpha \log \alpha + (1-\alpha) \log (1-\alpha)}\beta
=
-\rho \log \left(h^{-1}\left(\frac 1\rho\right) - 1\right) + \frac 1 \delta \log \left(1- h^{-1}(\delta)\right)
+
\log \left( - \frac{\log h^{-1}(\frac 1 \rho)}{\log h^{-1}(\delta)}\right).
$$
If $\rho < \rho_{\text{strong}}(\delta)$, respectively, $\rho > \rho_{\text{strong}}(\delta)$, then the equality in~\eqref{eq:rho_delta_implicit} should be replaced by $>$, respectively, $<$, which is equivalent to $I_\alpha(\beta) + \alpha \log \alpha + (1-\alpha) \log (1-\alpha)>0$, respectively, $<0$. With this at hand, we can apply Theorem~\ref{theo:phase_trans_expect_linear_k_strong} which yields~\eqref{eq:strong_threshold_1}, respectively,~\eqref{eq:strong_threshold_2}.
\end{remark}

\begin{remark}
For $\delta=1/2$, that is, when the number of vertices is twice as large as the dimension, the thresholds computed in Remarks~\ref{rem:weak_threshold} and~\ref{rem:strong_threshold} are $\rho_{\text{weak}}(1/2) = 0.5693\ldots$ and $\rho_{\text{strong}}(1/2) = 0.1498\ldots$. Let us mention that for the Gaussian polytopes, respectively their symmetric versions, the thresholds are known~\cite[pp.~6,7]{donoho_tanner} to be
$$
\rho_{\text{weak}}^{\text{GP}}(1/2) = 0.5581\ldots,
\quad
\rho_{\text{strong}}^{\text{GP}}(1/2) = 0.1335\ldots,
\quad
\rho_{\text{weak}}^{\pm}(1/2) = 0.3848\ldots,
\quad
\rho_{\text{strong}}^{\pm}(1/2) = 0.0894\ldots.
$$
Numerically, $\rho_{\text{weak}}(\delta) > \rho_{\text{weak}}^{\text{GP}}(\delta)$ for all $\delta\in (0,1)$, and the difference of these functions is surprisingly close (but not equal to) $0$.   Thus, convex hulls of random walks are slightly more neighborly than Gaussian polytopes.
\end{remark}

\subsection{Threshold phenomena for face numbers: the intermediate regime}
Let us now take some very large dimension $d$ and look at the number of $(k-1)$-dimensional faces, where $k\to\infty$ but $k=o(d)$. The next theorem states that a phase transition occurs if $n$ is near $k \eee^{d/k}$.

\begin{theorem}[Weak threshold in the intermediate regime]\label{theo:phase_trans_expect_intermediate_k}
Let $d\to\infty$ and $k=k(d)$ be a function of $d$ such that
$$
\lim_{d\to\infty} k(d) = \infty
\qquad
\text{ and }
\qquad
\lim_{d\to\infty} \frac{k(d)}{d} = 0.
$$
If an integer  sequence $n=n(d)$ is such that $d <  n(d) \leq k \eee^{(1-\eps) d / k}$ for some $\eps>0$ and all sufficiently large $d$, then
\begin{equation}\label{eq:intermediate_to_1}
\lim_{d\to\infty} \frac {\E f_{k-1}(C_{n-1,d})}{\binom nk}
=
1.
\end{equation}
On the other hand, if  $n(d) \geq k \eee^{(1+\eps) d / k}$ for some $\eps>0$ and sufficiently large $d$, then
\begin{equation}\label{eq:intermediate_to_0}
\lim_{d\to\infty} \frac {\E f_{k-1}(C_{n-1,d})}{\binom nk}
=
0.
\end{equation}
\end{theorem}
\begin{proof}
Let first $n(d) \leq k \eee^{(1-\eps) d / k}$. Then $k \log (n/k) \leq  k \log \eee^{(1-\eps) d / k} = (1-\eps) d$ for sufficiently large $d$ and hence,  recalling~\eqref{eq:E_f_k_Lah_dual}, we can write
\begin{align*}
1 - \frac {\E f_{k-1}(C_{n-1,d})}{\binom nk}
&=
2 \P[\Lah(n,k) \in \{d+2,d+4,\ldots\}]
\leq
2 \P[\Lah(n,k) >d]\\
&
\leq
2 \P\left[\frac{\Lah(n,k)}{k \log (n/k)} > \frac{d}{k \log (n/k)}\right]
\leq
2 \P\left[\frac{\Lah(n,k)}{k \log (n/k)} > \frac{1}{1-\eps}\right]
\tond 0,
\end{align*}
where the last step holds by Theorem~\ref{theo:weak_LLN_intermediate}. This proves~\eqref{eq:intermediate_to_1}.
Let now $n(d) \geq k \eee^{(1+\eps) d / k}$. Then $k \log (n/k) > k \log \eee^{(1+\eps) d / k} = (1+\eps) d$ for sufficiently large $d$ and hence, in view of~\eqref{eq:E_f_k_Lah} we obtain
\begin{align*}
\frac {\E f_{k-1}(C_{n-1,d})}{\binom{n}{k}}
&=
2 \P[\Lah(n,k) \in \{d,d-2,d-4,\ldots\}]
\leq
2 \P[\Lah(n,k) \leq d]\\
&\leq
2 \P\left[\frac{\Lah(n,k)}{k \log (n/k)} \leq \frac{d}{k \log (n/k)}\right]
\leq
2 \P\left[\frac{\Lah(n,k)}{k \log (n/k)} \leq  \frac{1}{1+\eps}\right]
\tond 0,
\end{align*}
where in the last step we used Theorem~\ref{theo:weak_LLN_intermediate}. This proves~\eqref{eq:intermediate_to_0}.
\end{proof}
In the setting of Gaussian polytopes, the intermediate regime has been studied in~\cite{donoho_tanner}.
Note  that the central limit theorem conjectured in Remark~\ref{rem:CLT_intermedate_conj} would imply a formula for the limit in the critical window.

\section{Conic intrinsic volume sums of Weyl chambers}\label{sec:weyl}
Let us mention an interpretation of the Lah distribution in terms of conic intrinsic volumes. To each convex cone $C\subset \R^n$ it is possible to associate a sequence of quantities $\upsilon_0(C),\ldots, \upsilon_n(C)$ which are called conic intrinsic volumes; see~\cite[Section~6.5]{SW} and~\cite{AmelunxenLotz,ALMT14} for their definition and properties. The conic intrinsic volumes form a probability distribution meaning that they are non-negative and sum up to $1$.  For the Weyl chamber of type $A$, which is the convex cone defined by
$$
A^{(n)} := \{(x_1,\ldots,x_n)\in \R^n: x_1\geq x_2\geq \ldots \geq x_n\},
$$
the conic intrinsic volumes are well known to form the $\Lah(n,1)$-distribution meaning that
\begin{equation}\label{eq:upsilon_j_weyl_chamber}
\upsilon_j(A^{(n)}) = \P[\Lah(n,1) = j] = \frac 1 {n!} \stirling{n}{j}
\end{equation}
for all $j=1,\ldots,n$; see, e.g., \cite[Theorem~4.2]{KVZ17}. To state a more general identity involving $\Lah (n,k)$ with arbitrary $k=1,\ldots,n$, we denote by $\mathcal F_k(C)$ the set of $k$-dimensional faces of a polyhedral cone $C$, and let $T_F(C)$ be the tangent cone of $C$ at its face $F$.  The next theorem was obtained in~\cite[Theorem~3.3]{godland_kabluchko_schlaefli}; see also~\cite{godland_kabluchko_permutohedra} for related results.

\begin{theorem}[Conic intrinsic volume sums of $A^{(n)}$]\label{theo:conic_intrinsic_volumes_sums}
For all $k\in \{1,\ldots,n\}$ and $j\in \{k,\ldots,n\}$, we have
$$
\sum_{F\in \mathcal F_k(A^{(n)})} \upsilon_j(T_F(A^{(n)}))
=
\binom {n-1}{k-1} \P[\Lah(n,k) = j]
=
\frac{k!}{n!} \stirling{n}{j} \stirlingsec{j}{k}.
$$
\end{theorem}
Note that for $k=1$ we recover~\eqref{eq:upsilon_j_weyl_chamber} because the only one-dimensional face of $A^{(n)}$ is the line $\{x_1=\ldots=x_n\}$. The proof of Theorem~\ref{theo:conic_intrinsic_volumes_sums} given in~\cite{godland_kabluchko_schlaefli} used generating functions. Let us give a combinatorial proof relying on the  construction of the Lah distribution given in~\eqref{eq:representation_as_a_sum_over_blocks_of_composition}.
\begin{proof}
By~\cite[Lemma~3.12]{godland_kabluchko_schlaefli}, the collection of the tangent cones $T_F(A^{(n)})$, where $F\in \mathcal F_k(A^{(n)})$ runs through all $k$-dimensional faces of $A^{(n)}$, coincides (up to isometries) with the collection of direct products of the form $A^{(i_1)}\times \ldots \times A^{(i_k)}$, where $(i_1,\ldots,i_k)$ runs through all compositions of $n$ in $k$ summands. Recalling that $(b_1^{(n)}, \ldots, b_k^{(n)})$ denotes a uniform random composition of $n$ in $k$ summands, we can write $$
\sum_{F\in \mathcal F_k(A^{(n)})} \upsilon_j(T_F(A^{(n)})) = \binom {n-1}{k-1} \E \upsilon_j(A^{(b_1^{(n)})}\times \cdots \times A^{(b_k^{(n)})}).
$$
Recalling from~\eqref{eq:representation_as_a_sum_over_blocks_of_composition} that $(Z^{(m)}_{i})_{i,m\in\N}$ are independent random variables with $\P[Z^{(m)}_{i}=\ell] = \frac 1 {i!}\stirling{i}{\ell} = \upsilon_\ell(A^{(i)})$, and using the formula for the conic intrinsic volumes of direct products, see formula (2.9) in \cite{AmelunxenLotz}, we get
$$
\upsilon_j(A^{(i_1)}\times \ldots \times A^{(i_k)}) =\sum_{\substack{j_1,\ldots,j_k\in \N_0\\j_1+\ldots+j_k = j}}  \upsilon_{j_1}(A^{(i_1)}) \cdot \ldots \cdot \upsilon_{j_k}(A^{(i_k)}) = \P[Z_{i_1}^{(1)} + \ldots +  Z_{i_k}^{(k)} = j].
$$
Combining everything together, we obtain
$$
\sum_{F\in \mathcal F_k(A^{(n)})} \upsilon_j(T_F(A^{(n)}))
=
\binom {n-1}{k-1} \P\left[Z_{b_1^{(n)}}^{(1)} + \ldots +  Z_{b_k^{(n)}}^{(k)} = j\right]
=
\binom {n-1}{k-1} \P[\Lah(n,k) = j]
=
\frac{k!}{n!}
 \stirling{n}{j} \stirlingsec{j}{k},
$$
where we applied the representation of the Lah distribution given in~\eqref{eq:representation_as_a_sum_over_blocks_of_composition}.
\end{proof}

In~\cite{GNP17} it has been shown that, under a minor condition, the conic intrinsic volumes of any sequence of convex cones whose dimension diverges to $\infty$ satisfy a central limit theorem. One may ask whether there is a natural convex cone $U_{n,k}$ whose conic intrinsic volumes are given by $\upsilon_j(U_{n,k}) = \P[\Lah(n,k) = j]$, for all $j\in \{k,\ldots,n\}$. We do not know how to answer this question, but Theorem~\ref{theo:conic_intrinsic_volumes_sums} states that $\P[\Lah(n,k) = j]$ is the  \textit{expected} $j$-th conic intrinsic volume of a uniformly selected \textit{random} $k$-dimensional face of the Weyl chamber $A^{(n)}$; see also~\cite[Theorem~3.1]{godland_kabluchko_schlaefli} and~\cite[Corollary~2.4]{godland_kabluchko_positive_hulls} for other examples of this type. Let us also mention that in~\cite[Lemma~6.5]{AmelunxenLotz} and~\cite[Theorem~3.14]{godland_kabluchko_permutohedra} (which look similar at a first sight) the Stirling numbers appear in a different order, that is, in the form $\stirlingsec{n}{j} \stirling{j}{k}$; see~\cite{knezevic} for a review of identities involving this and other types  of products.
%$\{x_1\geq x_2\geq \ldots \geq x_n\}$.  %see~\cite{,godland_kabluchko_schlaefli} for this and other representations in terms of Schl\"afli orthoschemes and permutohedra, as well as~\cite{AmelunxenLotz} for a closely related representation in terns of $j$-th characteristic polynomials of Weyl chambers.

\section{Appendix}
\subsection{Proof of Theorem~\ref{thm:clt_compositions}}
Recall that the distribution of the random uniform composition $(b_1^{(n)},\ldots,b_k^{(n)})$ can be represented as
\begin{equation*}%\label{eq:conitional_law1}
\P[(b_1^{(n)},\ldots,b_k^{(n)})\in\cdot]=\P[(G_1,\ldots,G_k)\in\cdot|G_1+\cdots+G_k=n],
\end{equation*}
where $G_1,\ldots,G_k$ are independent random variables having the same geometric law on $\N$ with parameter $\theta$. This representation holds for arbitrary $\theta\in(0,1)$ and we are free to choose $\theta:=\theta_n=k/n$.
%Let us first show that the centring $k\P[b_1^{(n)}=j]$ in Theorem~\ref{thm:clt_compositions} can be replaced by %$k\theta_n(1-\theta_n)^{j-1}=k\P[G_1=j]$. Indeed, for every fixed $j\in\N$,
%\begin{align*}
%\P[b_1^{(n)}=j]&=\frac{\binom{n-j-1}{k-2}}{\binom{n-1}{k-1}}=\frac{k-1}{n-1}\left(1-\frac{k-2}{n-j}\right)\left(1-\frac{k-2}{n-j+1}\right)\cdots\left(1-\frac{k-2}{n-2}\right)\\
%&=\left(\frac{k}{n}+O\left(\frac{1}{n}\right)\right)\left(1-\frac{k}{n}+O\left(\frac{1}{n}\right)\right)\left(1-\frac{k}{n}+O\left(\frac{1}{n}\right)\right)\cdots\left(1-\frac{k}{n}+O\left(\frac{1}{n}\right)\right)\\
%&=\frac{k}{n}\left(1-\frac{k}{n}\right)^{j-1}+O\left(\frac{1}{n}\right)=\theta_n\left(1-\theta_n\right)^{j-1}+O\left(\frac{1}{n}\right),
%\end{align*}
%where the constants in the Landau symbols may depend on $j$. Therefore,
%$$
%\left|k\P[b_1^{(n)}=j]-k\theta_n(1-\theta_n)^{j-1}\right|=O(1),\quad n\to\infty,
%$$
%which demonstrates that it is enough to prove
%
As we demonstrated in Lemma~\ref{eq:centerings_equivalent}, it suffices to show that
\begin{equation*}
\left(\frac{N_j^{(n)}-k\theta_n(1-\theta_n)^{j-1}}{\sqrt{k}}\right)_{j\geq 1}\todistr \left(\mathcal{N}_j\right)_{j\geq 1}.
\end{equation*}
By the Cram\'{e}r--Wold device the last display is equivalent to
$$
\frac{\sum_{l=1}^{M}\beta_l (N_l^{(n)}-k\theta_n(1-\theta_n)^{l-1})}{\sqrt{k}}\todistr \sum_{l=1}^{M}\beta_l \mathcal{N}_l,
$$
for arbitrary fixed $M\in\N$ and $\beta_1,\beta_2,\ldots,\beta_M\in\R$. Put
$$
f_n(x):=\sum_{l=1}^{M}\beta_l (\1_{\{x=l\}}-\theta_n(1-\theta_n)^{l-1}),\quad x\in\N,
$$
and, further,
\begin{equation*}
S_{n,k}:=\sum_{j=1}^{k}G_j,\quad T_{n,k}:=\sum_{j=1}^{k}f_n(G_j).
\end{equation*}
The subsequent analysis relies on the following representation:
\begin{multline*}
\E \exp\left(\ii t k^{-1/2} \left(\sum_{l=1}^{M}\beta_l (N_l^{(n)}-k\theta_n(1-\theta_n)^{l-1})\right)\right)=\E \exp\left(\ii t k^{-1/2} \sum_{j=1}^k f_n(b_j^{(n)})\right)\\
\overset{\eqref{eq:conitional_law}}{=}\E \exp\left(\ii t k^{-1/2} \sum_{j=1}^k f_n(G_j)\Big|S_{n,k}=n\right)=\E \exp\left(\ii t k^{-1/2} T_{n,k}\Big|S_{n,k}=n\right).
\end{multline*}
Thus, it suffices to prove that, for every fixed $t\in\R$,
\begin{equation}\label{eq:clt_composition_proof1}
\lim_{n\to\infty}\E \exp\left(\ii t k^{-1/2} T_{n,k}\Big|S_{n,k}=n\right)=\E\exp\left(\ii t \sum_{l=1}^{M}\beta_l \mathcal{N}_l\right).
\end{equation}
According to Theorem 1 in \cite{holst1979} we have
\begin{align}
\E\left(\eee^{\ii t k^{-1/2}T_{n,k}}\Big| S_{n,k}=n\right)
&=
\frac{1}{2\pi\P[S_{n,k}=n]}\int_{-\pi}^{\pi}\E \eee^{\ii s (S_{n,k}-n)+\ii tk^{-1/2}T_{n,k}}{\rm d}s\notag
\\
&=\frac{1}{2\pi\sqrt{k}\P[S_{n,k}=n]}\int_{-\pi\sqrt{k}}^{\pi\sqrt{k}}\E \eee^{\ii u k^{-1/2}(S_{n,k}-n)+\ii tk^{-1/2}T_{n,k}}{\rm d}u. \label{eq:holst_representation}
\end{align}
Using the Lindeberg--Feller central limit theorem we obtain
$$
\left(\frac{S_{n,k}-n}{\sqrt{k}},\frac{T_{n,k}}{\sqrt{k}}\right)\todistr (\widetilde{N}_1,\widetilde{N}_2),
$$
where $(\widetilde{N}_1,\widetilde{N}_2)$ is a centred Gaussian vector with the following variances and covariance:
\begin{align*}
\sigma_1^2&:=\Var \widetilde{N}_1= \lim_{n\to\infty}\Var(G_1)=\frac{1-\alpha}{\alpha^2},\\
\sigma_2^2&:=\Var \widetilde{N}_2=\lim_{n\to\infty}\Var(f_n(G_1))=\sum_{l=1}^{M}\beta^2_l\alpha(1-\alpha)^{l-1}-\left(\sum_{l=1}^{M}\beta_l \alpha (1-\alpha)^{l-1}\right)^2,
\end{align*}
and
\begin{align*}
r
&:=\Cov(\widetilde{N}_1,\widetilde{N}_2)=\lim_{n\to\infty}\Cov(f_n(G_1),G_1))=\lim_{n\to\infty}\Cov\left(\sum_{l=1}^{M}\beta_l \1_{\{G_1=l\}},\sum_{l=1}^{M}l \1_{\{G_1=l\}}\right)
\\
&=\sum_{l=1}^{M}\beta_l l \alpha(1-\alpha)^{l-1}-\alpha^{-1}\sum_{l=1}^{M}\beta_l \alpha(1-\alpha)^{l-1}=\sum_{l=1}^{M}\beta_l (l\alpha-1)(1-\alpha)^{l-1}.
\end{align*}
Since $S_k-k$ has the negative binomial distribution, direct calculation shows that the limit $\lim_{n\to\infty}\sqrt{k}\P[S_k=n]$ exists and is positive. Thus, by the Lebesgue dominated convergence theorem, we deduce from~\eqref{eq:holst_representation} that
\begin{align}
\lim_{n\to\infty}\E\left(\eee^{\ii t k^{-1/2}T_{n,k}}\Big| S_{n,k}=n\right)
&={\rm const}\cdot\int_{-\infty}^{\infty}\E\exp\left(\ii u \widetilde{N}_1+\ii t\widetilde{N}_2\right){\rm d}{u}\notag
\\
&={\rm const}\cdot \int_{-\infty}^{\infty}\exp\left(-\frac{u^2\sigma_1^2+t^2\sigma_2^2+2rut}{2}\right){\rm d}u=\exp\left(-\frac{\sigma_1^2 \sigma_2^2 -r^2}{2\sigma_1^2}t^2\right). \label{eq:clt_composition_proof2}
\end{align}
To ensure applicability of the dominated convergence (which is non-trivial), one can argue as in the paper of Holst~\cite{holst1979} who relies on~\cite{lecam1958}.
It remains to note that the right-hand sides of~\eqref{eq:clt_composition_proof1} and~\eqref{eq:clt_composition_proof2} coincide as is readily seen by comparing the variances.

\section*{Acknowledgement}
ZK acknowledges support by the German Research Foundation under Germany's Excellence Strategy  EXC 2044 -- 390685587, \textit{Mathematics M\"unster: Dynamics - Geometry - Structure}  and by the DFG priority program SPP 2265 \textit{Random Geometric Systems}. AM was supported by the National Research Foundation of Ukraine (project 2020.02/0014 ``Asymptotic regimes of perturbed random walks: on the edge of modern and classical probability'').  The authors thank Thomas Godland for useful discussions and the anonymous referee for useful suggestions.

\bibliography{bibliography}

\begin{thebibliography}{10}

\bibitem{alsmeyer_kabluchko_Marynych_LCM}
G.~Alsmeyer, Z.~Kabluchko, and A.~Marynych.
\newblock Limit theorems for the least common multiple of a random set of
  integers.
\newblock {\em Trans. Amer. Math. Soc.}, 372(7):4585--4603, 2019.

\bibitem{AmelunxenLotz}
D.~Amelunxen and M.~Lotz.
\newblock Intrinsic volumes of polyhedral cones: A combinatorial perspective.
\newblock {\em Discrete Comput. Geom.}, 58(2):371--409, 2017.

\bibitem{ALMT14}
D.~Amelunxen, M.~Lotz, M.~B. McCoy, and J.~A. Tropp.
\newblock Living on the edge: phase transitions in convex programs with random
  data.
\newblock {\em Inf. Inference}, 3(3):224--294, 2014.

\bibitem{baldi_vershynin}
P.~Baldi and R.~Vershynin.
\newblock A theory of capacity and sparse neural encoding.
\newblock {\em Neural Networks}, 143:12--27, 2021.

\bibitem{barbour_kowalski_nikeghbali}
A.~D. Barbour, E.~Kowalski, and A.~Nikeghbali.
\newblock Mod-discrete expansions.
\newblock {\em Probab. Theory Related Fields}, 158(3-4):859--893, 2014.

\bibitem{BV94}
Y.~M. Baryshnikov and R.~A. Vitale.
\newblock Regular simplices and {G}aussian samples.
\newblock {\em Discrete Comput. Geom.}, 11(2):141--147, 1994.

\bibitem{NBerestycki:09}
N.~Berestycki.
\newblock {\em Recent progress in coalescent theory}, volume~16 of {\em Ensaios
  Matem\'aticos [Mathematical Surveys]}.
\newblock Sociedade Brasileira de Matem\'atica, Rio de Janeiro, 2009.

\bibitem{billingsley}
P.~Billingsley.
\newblock {\em Convergence of Probability Measures}.
\newblock John Wiley \& Sons, 1999.

\bibitem{bulinski_shashkin_book}
A.~Bulinski and A.~Shashkin.
\newblock {\em Limit theorems for associated random fields and related
  systems}, volume~10 of {\em Advanced Series on Statistical Science \& Applied
  Probability}.
\newblock World Scientific Publishing Co. Pte. Ltd., Hackensack, NJ, 2007.

\bibitem{cilleruelo}
J.~Cilleruelo, J.~Ru\'{e}, P.~\v{S}arka, and A.~Zumalac\'{a}rregui.
\newblock The least common multiple of random sets of positive integers.
\newblock {\em J. Number Theory}, 144:92--104, 2014.

\bibitem{daboul}
S.~Daboul, J.~Mangaldan, M.~Z. Spivey, and P.~J. Taylor.
\newblock The {L}ah numbers and the {$n$}th derivative of {$e^{1/x}$}.
\newblock {\em Math. Mag.}, 86(1):39--47, 2013.

\bibitem{delbaen_kowalski_nikeghbali}
F.~Delbaen, E.~Kowalski, and A.~Nikeghbali.
\newblock Mod-$\phi$ convergence.
\newblock {\em Int. Math. Res. Not. IMRN}, 2015(11):3445--3485, 2015.

\bibitem{dembo_zeitouni_book}
A.~Dembo and O.~Zeitouni.
\newblock {\em Large deviations techniques and applications}, volume~38 of {\em
  Stochastic Modelling and Applied Probability}.
\newblock Springer-Verlag, Berlin, 2010.
\newblock Corrected reprint of the second (1998) edition.

\bibitem{diaconis_freedman}
P.~Diaconis and D.~Freedman.
\newblock A dozen de {F}inetti-style results in search of a theory.
\newblock {\em Ann. Inst. H. Poincar\'{e} Probab. Statist.}, 23(2,
  suppl.):397--423, 1987.

\bibitem{donoho_neighborliness_proportional}
D.~L. Donoho.
\newblock High-dimensional centrally symmetric polytopes with neighborliness
  proportional to dimension.
\newblock {\em Discrete Comput. Geom.}, 35(4):617--652, 2006.

\bibitem{donoho_tanner_neighborliness}
D.~L. Donoho and J.~Tanner.
\newblock Neighborliness of randomly projected simplices in high dimensions.
\newblock {\em Proc. Natl. Acad. Sci. USA}, 102(27):9452--9457, 2005.

\bibitem{donoho_tanner_sparse_nonnegative_sol}
D.~L. Donoho and J.~Tanner.
\newblock Sparse nonnegative solution of underdetermined linear equations by
  linear programming.
\newblock {\em Proc. Natl. Acad. Sci. USA}, 102(27):9446--9451, 2005.

\bibitem{donoho_tanner}
D.~L. Donoho and J.~Tanner.
\newblock Counting faces of randomly projected polytopes when the projection
  radically lowers dimension.
\newblock {\em J. Amer. Math. Soc.}, 22(1):1--53, 2009.

\bibitem{DonohoTanner}
D.~L. Donoho and J.~Tanner.
\newblock Counting the faces of randomly-projected hypercubes and orthants,
  with applications.
\newblock {\em Discrete Comput. Geom.}, 43(3):522--541, 2009.

\bibitem{donoho_tanner_observed_universality}
D.~L. Donoho and J.~Tanner.
\newblock Observed universality of phase transitions in high-dimensional
  geometry, with implications for modern data analysis and signal processing.
\newblock {\em Philos. Trans. R. Soc. Lond. Ser. A Math. Phys. Eng. Sci.},
  367(1906):4273--4293, 2009.
\newblock With electronic supplementary materials available online.

\bibitem{donoho_tanner_exponential_bounds}
D.~L. Donoho and J.~Tanner.
\newblock Exponential bounds implying construction of compressed sensing
  matrices, error-correcting codes, and neighborly polytopes by random
  sampling.
\newblock {\em IEEE Trans. Inform. Theory}, 56(4):2002--2016, 2010.

\bibitem{feray_meliot_nikeghbali_book}
V.~F\'{e}ray, P.-L. M\'{e}liot, and A.~Nikeghbali.
\newblock {\em Mod-{$\phi$} convergence: Normality zones and precise
  deviations}.
\newblock SpringerBriefs in Probability and Mathematical Statistics. Springer,
  2016.

\bibitem{Fields:1970}
J.~L. Fields.
\newblock The uniform asymptotic expansion of a ratio of two gamma functions.
\newblock In {\em Proc. Inter. Conf. Constructive Func. Theory}, pages
  171--176, 1970.

\bibitem{Flajolet_book}
P.~Flajolet and R.~Sedgewick.
\newblock {\em Analytic Combinatorics}.
\newblock Cambridge University Press, 2009.

\bibitem{godland_kabluchko_permutohedra}
T.~Godland and Z.~Kabluchko.
\newblock Projections and angle sums of permutohedra and other polytopes, 2020.
\newblock Preprint at arXiv: 2009.04186.

\bibitem{godland_kabluchko_schlaefli}
T.~Godland and Z.~Kabluchko.
\newblock Angle sums of {S}chl{\"a}fli orthoschemes.
\newblock {\em Discrete Comput. Geom., to appear}, 2022+.
\newblock Preprint at arXiv: 2007.02293v3.

\bibitem{godland_kabluchko_positive_hulls}
T.~Godland and Z.~Kabluchko.
\newblock Positive hulls of random walks and bridges.
\newblock {\em Stoch. Proc. and Appl.}, 147:327--362, 2022.

\bibitem{GKT2020_HighDimension1}
T.~Godland, Z.~Kabluchko, and C.~Th{\"a}le.
\newblock Random cones in high dimensions {I}: {D}onoho-{T}anner and
  {C}over-{E}fron cones.
\newblock {\em Discrete Analysis, to appear}, 2022+.
\newblock Preprint at arXiv: 2012.06189.

\bibitem{GNP17}
L.~Goldstein, I.~Nourdin, and G.~Peccati.
\newblock Gaussian phase transitions and conic intrinsic volumes: {S}teining
  the {S}teiner formula.
\newblock {\em Ann. Appl. Probab.}, 27(1):1--47, 2017.

\bibitem{Graham1994}
R.~L. Graham, D.~E. Knuth, and O.~Patashnik.
\newblock {\em Concrete Mathematics: A Foundation for Computer Science}.
\newblock Addison-Wesley Publishing Company, Inc., USA, 2nd edition, 1994.

\bibitem{gruenbaum_book}
B.~Gr\"{u}nbaum.
\newblock {\em Convex polytopes}, volume 221 of {\em Graduate Texts in
  Mathematics}.
\newblock Springer-Verlag, New York, second edition, 2003.
\newblock Prepared and with a preface by V. Kaibel, V. Klee and G. M. Ziegler.

\bibitem{harper}
L.~H. Harper.
\newblock Stirling behavior is asymptotically normal.
\newblock {\em Ann. Math. Statist.}, 38:410--414, 1967.

\bibitem{holst1979}
L.~Holst.
\newblock Two conditional limit theorems with applications.
\newblock {\em Ann. Stat.}, 7(3):551--557, 1979.

\bibitem{holst_urn_models}
L.~Holst.
\newblock A unified approach to limit theorems for urn models.
\newblock {\em J. Appl. Probab.}, 16(1):154--162, 1979.

\bibitem{HugSchneiderThresholdPhenomenaPart2}
D.~Hug and R.~Schneider.
\newblock Another look at threshold phenomena for random cones.
\newblock {\em Studia Sc. Math. Hungarica}, 58(4):489 -- 504, 2021.

\bibitem{HugSchneiderThresholdPhenomena}
D.~Hug and R.~Schneider.
\newblock Threshold phenomena for random cones.
\newblock {\em Discrete Comput. Geom.}, 67:564--594, 2022.

\bibitem{ivchenko}
G.~I. Ivchenko.
\newblock On the random coverage of the circle: a discrete model.
\newblock {\em Discrete Math. Appl.}, 4(2):147--162, 1994.

\bibitem{jacod_kowalski_nikeghbali_mod_Gauss}
J.~Jacod, E.~Kowalski, and A.~Nikeghbali.
\newblock Mod-{G}aussian convergence: new limit theorems in probability and
  number theory.
\newblock {\em Forum Math.}, 23(4):835--873, 2011.

\bibitem{joag_dev_proschan}
K.~Joag-Dev and F.~Proschan.
\newblock Negative association of random variables, with applications.
\newblock {\em Ann. Statist.}, 11(1):286--295, 1983.

\bibitem{johnson_kotz_balakrishnan_book}
N.~L. Johnson, S.~Kotz, and N.~Balakrishnan.
\newblock {\em Discrete multivariate distributions}.
\newblock Wiley Series in Probability and Statistics: Applied Probability and
  Statistics. John Wiley \& Sons, Inc., New York, 1997.
\newblock A Wiley-Interscience Publication.

\bibitem{kabluchko_marynych_sulzbach_mode}
Z.~Kabluchko, A.~Marynych, and H.~Sulzbach.
\newblock Mode and {E}dgeworth expansion for the {E}wens distribution and the
  {S}tirling numbers.
\newblock {\em J. Integer Seq.}, 19(8):Art. 16.8.8, 17, 2016.

\bibitem{kabluchko_marynych_sulzbach_edgeworth}
Z.~Kabluchko, A.~Marynych, and H.~Sulzbach.
\newblock General {E}dgeworth expansions with applications to profiles of
  random trees.
\newblock {\em Ann. Appl. Probab.}, 27(6):3478--3524, 2017.

\bibitem{KVZ17}
Z.~Kabluchko, V.~Vysotsky, and D.~Zaporozhets.
\newblock Convex hulls of random walks: expected number of faces and face
  probabilities.
\newblock {\em Adv. Math.}, 320:595--629, 2017.

\bibitem{KVZ17b}
Z.~Kabluchko, V.~Vysotsky, and D.~Zaporozhets.
\newblock Convex hulls of random walks, hyperplane arrangements, and {W}eyl
  chambers.
\newblock {\em Geom. Func. Anal.}, 27(4):880--918, 2017.

\bibitem{knezevic}
M.~Kne\v{z}evi\'{c}, V.~Kr\v{c}adinac, and L.~Reli\'c.
\newblock Matrix products of binomial coefficients and unsigned {S}tirling
  numbers, 2020.
\newblock Preprint at arXiv: 2012.15307.

\bibitem{kolcin_book_allocations}
V.~F. Kolchin, B.~A. Sevastyanov, and V.~P. Chistyakov.
\newblock {\em Random allocations}.
\newblock V. H. Winston \& Sons, Washington, D.C.; distributed by Halsted Press
  [John Wiley \& Sons], New York-Toronto, Ont.-London, 1978.

\bibitem{kolcin_gen_allocation}
V.~F. Kol\v{c}in.
\newblock A certain class of limit theorems for conditional distributions.
\newblock {\em Litovsk. Mat. Sb.}, 8:53--63, 1968.

\bibitem{kolcin_branching}
V.~F. Kol\v{c}in.
\newblock Branching processes, random trees and a generalized particle
  distribution scheme.
\newblock {\em Mat. Zametki}, 21(5):691--705, 1977.

\bibitem{kowalski_nikeghbali_poisson}
E.~Kowalski and A.~Nikeghbali.
\newblock Mod-{P}oisson convergence in probability and number theory.
\newblock {\em Int. Math. Res. Not. IMRN}, 2010(18):3549--3587, 2010.

\bibitem{kowalski_nikeghbali_zeta}
E.~Kowalski and A.~Nikeghbali.
\newblock Mod-{G}aussian convergence and the value distribution of
  {$\zeta(\frac12+it)$} and related quantities.
\newblock {\em J. Lond. Math. Soc. (2)}, 86(1):291--319, 2012.

\bibitem{Lah1954}
I.~Lah.
\newblock A new kind of numbers and its application in the actuarial
  mathematics.
\newblock {\em Boletim do Instituto dos Actu\'{a}rios Portugueses}, 9:7--15,
  1954.

\bibitem{lecam1958}
L.~Le~Cam.
\newblock Un th\'{e}or\`eme sur la division d'un intervalle par des points pris
  au hasard.
\newblock {\em Publ. Inst. Statist. Univ. Paris}, 7(3-4):7--16, 1958.

\bibitem{lebowitz_etal}
J.~L. Lebowitz, B.~Pittel, D.~Ruelle, and E.~R. Speer.
\newblock Central limit theorems, {L}ee-{Y}ang zeros, and graph-counting
  polynomials.
\newblock {\em J. Combin. Theory Ser. A}, 141:147--183, 2016.

\bibitem{louchard_first_kind}
G.~Louchard.
\newblock Asymptotics of the {S}tirling numbers of the first kind revisited: a
  saddle point approach.
\newblock {\em Discrete Math. Theor. Comput. Sci.}, 12(2):167--184, 2010.

\bibitem{louchard_second_kind}
G.~Louchard.
\newblock Asymptotics of the {S}tirling numbers of the second kind revisited.
\newblock {\em Appl. Anal. Discrete Math.}, 7(2):193--210, 2013.

\bibitem{meliot_nikeghbali_statmech}
P.-L. M\'eliot and A.~Nikeghbali.
\newblock Mod-{G}aussian convergence and its applications for models of
  statistical mechanics.
\newblock In {\em In memoriam {M}arc {Y}or---{S}\'eminaire de {P}robabilit\'es
  {XLVII}}, volume 2137 of {\em Lecture Notes in Math.}, pages 369--425.
  Springer, Cham, 2015.

\bibitem{moser_wyman}
L.~Moser and M.~Wyman.
\newblock Asymptotic development of the {S}tirling numbers of the first kind.
\newblock {\em J. London Math. Soc.}, 33:133--146, 1958.

\bibitem{moser_wyman_second}
L.~Moser and M.~Wyman.
\newblock Stirling numbers of the second kind.
\newblock {\em Duke Math. J.}, 25:29--43, 1958.

\bibitem{narumi}
S.~{Narumi}.
\newblock {On a power series having only a finite number of algebraico
  logarithmic singularities on its circle of convergence}.
\newblock {\em {T\^ohoku Math. J.}}, 30:185--201, 1929.

\bibitem{sachkov_book_combinatorial}
V.~N. Sachkov.
\newblock {\em Combinatorial methods in discrete mathematics}, volume~55 of
  {\em Encyclopedia of Mathematics and its Applications}.
\newblock Cambridge University Press, Cambridge, 1996.

\bibitem{sachkov_book_probabilistic}
V.~N. Sachkov.
\newblock {\em Probabilistic methods in combinatorial analysis}, volume~56 of
  {\em Encyclopedia of Mathematics and its Applications}.
\newblock Cambridge University Press, Cambridge, 1997.

\bibitem{SW}
R.~Schneider and W.~Weil.
\newblock {\em Stochastic and integral geometry}.
\newblock Probability and its Applications (New York). Springer-Verlag, Berlin,
  2008.

\bibitem{sibuya}
M.~Sibuya.
\newblock Log-concavity of {S}tirling numbers and unimodality of {S}tirling
  distributions.
\newblock {\em Ann. Inst. Statist. Math.}, 40(4):693--714, 1988.

\bibitem{sloane}
N.~J.~A. Sloane~(editor).
\newblock {The {O}n-{L}ine {E}ncyclopedia of {I}nteger {S}equences}.
\newblock https://oeis.org.

\bibitem{Timashov}
A.~N. Timashev.
\newblock On asymptotic expansions of {S}tirling numbers of the first and
  second kinds.
\newblock {\em Discrete Math. Appl.}, 8(5):533--544, 1998.

\bibitem{trunov}
A.~N. Trunov.
\newblock Limit theorems in the problem of allocation of identical particles
  among different cells.
\newblock {\em Proc. Steklov Inst. Math.}, 177:157--175, 1988.

\bibitem{vershik_sporyshev_estimation_simplex1983}
A.~M. Vershik and P.~V. Sporyshev.
\newblock Estimation of the mean number of steps in the simplex method, and
  problems of asymptotic integral geometry.
\newblock {\em Dokl. Akad. Nauk SSSR}, 271(5):1044--1048, 1983.

\bibitem{vershik_sporyshev_asymptotic_estimate_1986}
A.~M. Vershik and P.~V. Sporyshev.
\newblock An asymptotic estimate of the average number of steps of the
  parametric simplex method.
\newblock {\em USSR Comput. Math. and Math. Physics}, 26(3):104--113, 1986.

\bibitem{vershik_sporyshev_asymptotic_faces_random_polyhedra1992}
A.~M. Vershik and P.~V. Sporyshev.
\newblock Asymptotic behavior of the number of faces of random polyhedra and
  the neighborliness problem.
\newblock {\em Selecta Math. Soviet.}, 11(2):181--201, 1992.
\newblock Selected translations.

\bibitem{Vershik+Yakubovich:2003}
A.~M. Vershik and Y.~V. Yakubovich.
\newblock Asymptotics of the uniform measures on simplices and random
  compositions and partitions.
\newblock {\em Func. Anal. Appl.}, 37(4):273--280, 2003.

\end{thebibliography}
\bibliographystyle{abbrv}

\end{document}